\documentclass[sn-mathphys-num]{sn-jnl}
\usepackage{anyfontsize}
\usepackage{amscd}
\usepackage{natbib}
\usepackage{tikz}
\usetikzlibrary{patterns, positioning, arrows}
\usetikzlibrary{arrows,decorations.pathmorphing,positioning}
\usepackage{tikz-cd}
\usepackage{import}
\usepackage[all,cmtip]{xy}
\usepackage{graphicx}%
\usepackage{multirow}%
\usepackage{amsmath,amssymb,amsfonts}%
\usepackage{amsthm}%
\usepackage{mathrsfs}%
\usepackage[title]{appendix}%
\usepackage{xcolor}%
\usepackage{textcomp}%
\usepackage{manyfoot}%
\usepackage{booktabs}%
\usepackage{algorithm}%
\usepackage{algorithmicx}%
\usepackage{algpseudocode}%
\usepackage{tabularx}
\usepackage{listings}%

\usepackage[capitalize]{cleveref}

\newtheorem{thm}{Theorem}[section]    
\newtheorem{lem}[thm]{Lemma}          
\newtheorem{com}[thm]{Comments}  
\theoremstyle{definition}
\newtheorem{defn}[thm]{Definition}    
\newtheorem{corollary}[thm]{Corollary}
\newtheorem{remark}[thm]{Remark}
\newtheorem{proposition}[thm]{Proposition}

\title[Article Title]{Any K\"ahler metric is a Fisher information metric}

\author*{\fnm{Emmanuel} \sur{Gnandi}}\email{kpanteemmanuel@gmail.com, gnandi@insa-toulouse.fr}

\affil{\orgdiv{INSA de Toulouse}, \orgname{Département de Génie Mathématique}, \orgaddress{\street{135 Avenue de Rangueil}, \city{Toulouse}, \postcode{31077}, \country{France}}}

\abstract{The Fisher information metric or the Fisher-Rao metric corresponds to a natural Riemannian metric defined on a parameterized family of probability density functions. As in the case of Riemannian geometry, we can define a distance in terms of the Fisher information metric, called the Fisher-Rao distance. The Fisher information metric has a wide range of applications in estimation and information theories. Indeed, it provides the most informative Cramer-Rao bound for an unbiased estimator. The Goldberg conjecture is a well-known unsolved problem which states that any compact Einstein almost K\"ahler manifold is necessarily a K\"ahler-Einstein. Note that, there is also a known odd-dimensional analog of the Goldberg conjecture in the literature. The main objective of this paper is to  establish a new characterization of coK\"ahler manifolds and K\"ahler manifolds; our characterization is statistical in nature. Finally, we corroborate that every, K\"ahler and co-K\"ahler manifolds, can be viewed as being a parametric family of probability density functions, whereas K\"ahler and coK\"ahler metrics can be regarded as Fisher information metrics. In particular, we prove that, when the K\"ahler metric is real analytic, it is
always locally the Fisher information of an exponential family. We also tackle
the link between K\"ahler potential and Kullback-Leibler divergence.}

\begin{document}
\maketitle
\section{Introduction}
In this section, certain relevant notations and useful definitions related to information geometry are recalled. Information geometry is an emerging field which provides a certain correspondence between differential geometry and statistics through the Fisher information matrix (cf.\cite{amari1987differential}, \cite{amari2007methods}, \cite{rao1992information}, \cite{lauritzen1987statistical}). Several authors have examined the process of constructing a proper distance between probability distributions referring to its significance in both practical, theoretical statistics, information theory, signal processing, quantitative finance, and machine Learning. To solve this problem, Rao \cite{rao1992information} proposed the geodesic distance induced by the Fisher information metric. This geodesic distance is basically known in literature as the Fisher-Rao distance or Rao’s distance. Let us consider a statistical model $\mathcal{P}_{d}(\Xi)$, which is a set of parametric densities defined as follows:
\begin{equation}
    \mathcal{P}_{d}(\Xi)=\{p_{\theta}:\,\theta\in \Theta\},
\end{equation}
where $\Theta$(the parametric space) is an open set of $\mathbb{R}^{d}$ and $\Xi$ is a sample space, with parameter $\theta=(\theta^{1},....,\theta^{d})$.
Assuming that $p$ never vanishes, the likelihood function is expressed as follows:
$$l_{\theta}(x)=\log p_{\theta}(x).$$
The Fisher matrix tensor is indicated in terms of:
\begin{equation}
    g^{F}(\theta)_{ij}=E_{\theta}(\frac{\partial l_{\theta}(x)}{\partial\theta_{i}}\frac{\partial l_{\theta}(x)}{\partial\theta_{j}})=-E_{\theta}(\frac{\partial^{2}l_{\theta}(x)}{\partial\theta_{i}\partial\theta_{j}})\quad\text{for all}\quad\frac{\partial}{\partial\theta^{i}},\frac{\partial}{\partial\theta^{j}}\in T_{\theta}\Theta=\mathbb{R}^{d}.
\end{equation}
A simple computation reveals that the Fisher-Information matrix can be stated as follows:
\begin{equation}
    g^{F}(\theta)_{ij}=4\int_{\Omega}\frac{\partial\sqrt{p_{\theta}(x)}}{\partial\theta_{i}}\frac{\partial\sqrt{p_{\theta}(x)}}{\partial\theta_{j}}dx
    \label{eq:Fisher}
\end{equation}
The obtained geometry is called the informative geometry of the parametric 
family $\mathcal{P}_{d}(\Xi)$. The Fisher information matrix can be regarded as a measure of the amount of information present in the data about a parameter $\theta$. \citet{rao1992information} demonstrated that, under elementary regularity hypotheses, $ \mathcal{P}_{d}(\Xi)$ or $\Theta$ 
is a Riemannian manifold with a Riemannian metric, defined by the Fisher-information matrix, called the Fisher-information metric or the Fisher-Rao metric. The Fisher–information metric is invariant under the action of the diffeomorphism group. [\citet{cencov2000statistical},p. 156] proved the  uniqueness
result for Fisher’s information metric on finite sample
spaces. Subsequently, in \citet{ay2015information}, the authors extended Cencov's result to infinite sample spaces. The Fisher information metric is invariant under parametrization on a sample space $\Xi$ and covariant under parametrization on a parametric space $\Theta$, see \cite{calin2014geometric}. The Fisher-Rao distance is inherently designed to be invariant under diffeomorphisms of both the sample space $\Xi$ 
and the parameter space $\Theta$.

\vspace{1ex}
Quantum information geometry has been equally developed, which is a quite natural outcome, mainly as the standard quantum mechanics is a probabilistic-statistical theory. It has been corroborated using the geometrical formulation of quantum mechanics that, the Fisher quantum information metric (SLD metric) on the space of pure  quantum states coincides with the Fubini-Study metric of the complex projective space (up to constant)\cite{facchi2010classical}. The above remark indicates a close relationship between Information geometry and K\"ahler geometry. As far as the current research work is concerned, our central focus is upon the statistical characterization of K\"ahler and coK\"ahler metrics. This characterization generalizes the well known characterization of K\"ahler metrics in Hermitian manifolds.
\vspace{1ex}

This paper is organized as follows.  In Section \ref{se:section1}, we recall the relevant background in quantum analogue of information geometry. In Section \ref{se:Ham},
we recall the notion of Hamiltonian structure on odd dimensional manifold and establish the main result of this section:
\vspace{1ex}

\begin{thm}\label{mainTheorem} 
An odd dimensional manifold $M^{2d+1}$ admits a Hamiltonian structure, if and only if, its canonical bundle admits a maximal parallel section.
\end{thm}
\vspace{1ex}
The proof is based on the symplecization technique and the Gromov h-principle.
This theorem reveals that any Hamiltonian structure on an odd-dimension manifold $M^{2d+1}$ arises from SMAT structure on $M^{2d+1}$. Hamiltonian manifold refers to the odd dimensional analogue
of a symplectic manifold. Thus, a link between symplectic geometry and
information geometry is etabhished.

Section \ref{se:stable H and smat} provides an overview on certain basic concepts of a stable Hamiltonian structure that are highly
needed throughout this paper. In this section, our initial target is to prove that on every shs manifold, there are always canonical connections called shs connections. These connections will be useful in the sequel of the paper.
\vspace{1ex}
\begin{thm}
     On every shs manifold $(M^{2d+1},\Omega,\alpha)$, there are always connections such that:
\begin{equation}
    \nabla E=0\,,\nabla\Omega=0.
\end{equation}
\end{thm}
The proof of this theorem relies on the symplectization technique and we use the symplectic connections on the product $M^{2d+1}\times\mathbb{R}$ to build up the desired connections on the shs manifold. Next, we set forward a description of the set of shs connections in \cref{eq:112} and we confirm that there are shs connections invariant by the Reeb field $E$ in \cref{prop:invar}. Our second target is to prove that on every shs manifold, both orthogonal distributions $\ell_{\Omega}$ and $\ell_{\alpha}$ are always adapted to SMAT connections
in \cref{eq:190}. We obtain the local components of the torsion of the dual of an shs connection in local Darboux charts in \cref{pr:torsiondual}. Our third target lies therefore in proving that there are shs manifolds which do not contain statistical SMAT shs connections. These results are obtained with regard to \cref{prop:contact}, \cref{eq:128}, \cref{prop:fundamenta} and 
   \cref{prop:fundabis}. In \cref{sec:cokahlerstructure}, our first goal is to introduce a new characterization of a co-K\"ahler manifold:
\vspace{1ex}
\begin{thm}
    An shs manifold $(M^{2d+1},\Omega,\alpha)$ is a co-Kähler manifold, if and only if, it is a statistical shs manifold.
\end{thm}
\vspace{1ex}
According to the above Theorem, we deduce that shs manifolds that contain statistical shs connections are co-Kähler manifolds. In \cref{pro:geode} and \cref{cor:stat}, we examine some statistical properties of shs manifolds, from which we derive a description of the space of statistical shs connections on a statistical shs manifold (co-K\"ahler manifold). In \cref{se:Kal and cok}, we
extend the arguments of \cref{thm:principal} to provide a statistical characterization of K\"ahler manifold. This result is obtained from \cref{th:AlmsKha}. Statistical characterization of K\"ahler metrics in terms of parallel transport is displayed in \cref{subsect:stacha}. In \cref{subsection:Kobay}, we report a generalization of the theorem of Kobayashi. In \cref{subsec:Gol}, we face our statistical characterization of K\"ahler manifold with the integrability condition of S.Golberg. In \cref{subsec:Posit}, we invest two approaches to prove that any K\"ahler and co-K\"ahler metric is a Fisher-Information metric:
\vspace{1ex}

\begin{thm}
    Any K\"ahler and co-K\"ahler metric is a Fisher information metric.
\end{thm}
\vspace{1ex}
The above result reveals the powerful link between K\"ahler geometry and Information geometry. In \cref{subsec:kahsub}, we explore the nature of K\"ahler and coK\"ahler manifolds as submanifolds of statistical models. Departing from results drawn in \cref{prop:kaandco}, we obtain a lower bound on both geodesic distance of K\"ahler metrics and the geodesic distance of co-K\"ahler metrics in \cref{cor:Hell}. In \cref{subsec:Stasymp}, we address relations between K\"ahler potential and statistical symplectic connections. We tackle the link between K\"ahler potential and Kullback-Leibler divergence in \cref{subs:KahlandKL}. Finally in \cref{subs:analy}, we demonstrate that:
\vspace{1ex}

\begin{thm}
    Any real analytic K\"ahler metric is locally the Fisher information of an exponential family.
\end{thm}
\vspace{1ex}

The above result establishes a strong link between real analytic K\"ahler manifold  and exponential families distributions.

\section{Background materials and motivation of the paper}
\label{se:section1}
Let $\mathcal{H}$ ($\text{dim}_{\mathbb{C}}\mathcal{H}=d$) be a finite dimensional Hilbert space. Denote by $\mathcal{L}\mathcal{H}$ and $\mathcal{L}_{h}\mathcal{H}$ the set of linear operators and the set of Hermitian operators of $\mathcal{H}$, respectively. 
The family of density operators (quantum states) is defined as follows:
$$ \mathcal{M}=\{ \rho\in\mathcal{L}_{h}\mathcal{H},\,  \rho\geq 0,\, \text{Tr}(\rho)=1.\}
$$
The quantum states $\mathcal{M}$ are partitioned into $\mathcal{M}=\cup_{k=0}^{d}\mathcal{M}_{k}$, where $\mathcal{M}_{k}=\{ \rho\in\mathcal{M},\text{rank}(\rho)=k\}$ is a $(2dk-k^{2}-1)$-dimensional real manifold. We notably treat $\mathcal{M}_{d}$(the space of faithful  quantum  states) and $\mathcal{M}_{1}$(the space of pure quantum states) in the sequel. We first consider the space of faithful  quantum  states. Clearly, $\mathcal{M}_{d}=\{ \rho\in \mathcal{M},\,  \rho>0\}$ is an open subset of the affine space $\{\rho\in\mathcal{L}_{h}\mathcal{H}, \text{Tr}(\rho)=1\}$. From this perspective, we can naturally induce that a flat connection on $\mathcal{M}_{d}$, called a mixture connection, is defined by: 
\begin{equation}
    (\nabla^{m}_{X}Y)\rho=X(Y\rho)\quad \forall\, X,Y\in\mathcal{X}(\mathcal{M}_{d}).
\end{equation}
It is well known departing from the work of \citet{petz1996monotone} that there exists an infinite number of metrics on $\mathcal{M}_{d}$ called quantum monotone metric tensors. These metrics were thoroughly classified by \citet{petz1996monotone}. Based on \citet{lesniewski1999monotone}, every quantum monotone metric tensor can be obtained by appropriately expanding a quantum relative entropy. Now, selecting a quantum monotone metric $g$, one can define a dualistic structure $(g,\nabla^{m},\nabla^{e})$ on $\mathcal{M}_{d}$ by:
   \begin{equation}
       g(\nabla^e_{X}Y,Z)=X.g(Y,Z)-g(Y,\nabla^{m}_{X}Z) \quad \forall\,X,Y,Z\in\mathcal{X}(\mathcal{M}_{d}).
       \label{eq:smatqua}
   \end{equation} 
Let us consider the quantum version of the classical Fisher metric, called SLD (Symmetric Logarithmic Derivative) metric, which is defined in \cite{petz2007quantum} as follows:

$$ g^{SLD}(X,Y)=\frac{1}{2}\text{Tr}\rho(L_{X}L_{Y}+L_{Y}L_{X}),\quad \forall\,X,Y\in\mathcal{X}(\mathcal{M}_{d}).$$

It is not difficult to notice that the dual connection of $\nabla^{m}$ with respect to the SLD metric defined in the above is expressed as:
$$
    (\nabla^e_{X}Y)\rho=\frac{1}{2}\{\rho(XL_{Y}-\text{Tr}\rho(XL_{Y})+(XL_{Y}-Tr\rho(XL_{Y}))\rho         \}, \quad \forall\,X,Y\in\mathcal{X}(\mathcal{M}_{d}).
$$
The connection $\nabla^e$ is called the exponential connection and referring to \cite{amari2000methods}, a direct computation reveals that the torsion of $\nabla^e$ is determined by:

   $$T^{\nabla^e}(X,Y)\rho=\frac{1}{4}[[L_{X},L_{Y}],\rho].$$
Finally, $\mathcal{M}_{d}$ is endowed with a dualistic structure $(g^{SLD},\nabla^{m},\nabla^{e})$ such that $\nabla^{m}$ is a flat connection and $\nabla^{e}$ is not symmetric but its curvature vanishes. It is significant to note that according to \cite{amari2000methods}, there exists a unique quantum monotone metric $g^{BKM}$ called Bogoliubov-Kubo-Mori metric (BKM metric) such that $(\mathcal{M}_{d},g^{BKM},\nabla^{m},\nabla^{e})$ is a dually flat manifold. In this particular case, the
associated quantum relative entropy corresponds to the von Neumann-Umegaki relative entropy defined by:
$$\text{D}(\rho,\sigma)=\text{Tr}[\rho(log \rho-log \sigma)]\quad     \forall \rho,\sigma\in \mathcal{M}_{d}.$$ 
In other words, the KMB Fisher metric is
derived from $K$ as follows:
$$
g^{BKM}_{ij}(\theta)=\frac{\partial^{2}\text{D}(\rho_{\theta},\rho_{\hat{\theta}})}{\partial\theta^{j}\partial\hat{\theta}^{i}}|_{\hat{\theta}=\theta}.
$$ 
 Let us consider the space of pure states $\mathcal{M}_{1}=\{ \rho\in \mathcal{M},\text{rank}(\rho)=1\}$. With respect to \cite{facchi2010classical} and \cite{cirelli1984hamiltonian}, it is well known that the space of pure states $\mathcal{M}_{1}$ is diffeomorphic with the complex projective space $CP(\mathcal{H})$ and the SLD metric on $\mathcal{M}_{1}$ is
identical up to a constant factor to the Fubini-Study 
metric on $CP(\mathcal{H})$:

\begin{equation}
g^{SLD}_{j\Bar{k}}=c\frac{(1+z^{l}\Bar{z}^{l})\delta_{jk}-z^{k}\Bar{z}^{j}}{(1+z^{l}\Bar{z}^{l})^{2}},
\label{eq:FS}
\end{equation} with $c\in\mathbb{R}$.

The above \cref{eq:FS} shows that the Fubini-Study metric can be precisely interpreted as a quantum information metric on pure quantum
states.
From the above remark, a natural link is established between statistics and K\"ahler structure of the projective space. The idea that the K\"ahler structure of projective space  may originate from the intrinsic geometry of a statistical manifold is an appealing one. Nevertheless, in the literature, a few links have been established between statistics and K\"ahler structure, except for the natural connection between the Hessian and Kähler geometries. As reported by
Dombrowski and Shima, the tangent bundle over a Hessian manifold admits a Kähler metric induced by the Hessian metric \cite{dombrowski1962geometry}, \cite{shima2007geometry}. Departing from the above arguments, it is legitimate to ask this natural question.

\begin{equation}
    \text{Question 1: Is a K\"ahler metric always a Fisher information metric ?}
    \label{eq:Ques1}
\end{equation}
\vspace{0,5ex}
As the author knows that such a question has not been previously addressed, the main target of this paper resides in answering this question. We are now prepared to introduce one of the central concepts of the paper.\\
Motivated and inspired by the richness of the  construction of dualistical structure in the quantum states, \citet{kurose2007statistical} introduced the notion of statistical manifold admitting torsion (SMAT) as a generalisation of statistical manifold with the basic goal of developing a quantum analogue of information geometry. \textit{A statistical manifold admitting a torsion} (SMAT) is a manifold endowed with a dualistic structure, where one of the dual must be torsion-free but the other is not necessarily (cf. \cite{henmi2019statistical,matsuzoe2012quasi}, \cite{zhang2019hessian}, \cite{zhang2019new}). In [\cite{amari2000methods}, p.19], Amari and Nagaoka asserted "The incorporation of torsion into the framework of information geometry, which would relate it to such fields as quantum mechanics (non commutative probability theory) and systems theory, is an interesting topic for the future.” As far as the current research paper is concerned, we tackle the impact of the torsion of dual connections on the symplectic, contact and Kahler geometry.
\vspace{1ex}
   \begin{defn}\cite{kurose2007statistical}
\label{th:smat}   
       Let $M$ be a manifold, and  let $(g,\nabla,\nabla^{*})$ be a dualistic structure on $M$. The quadruplet $(M,g, \nabla,\nabla^{*})$ is a SMAT if the following equivalent conditions are satisfied: 
\begin{equation}
    \nabla \text{ is torsion-free},
\end{equation}
\begin{equation}
    (\nabla_{X} g)(Y,Z)-(\nabla_{Y} g)(X,Z)=g(T^{\nabla^*}(X,Y),Z),
\end{equation}
\begin{equation}
     (\nabla^{*}_{X} g) (Y,Z)-(\nabla^{*}_{Y} g)(X,Z)=-g(T^{\nabla^*}(X,Y),Z).
\end{equation}
  \end{defn} 
  
\vspace{1ex}
Throughout this paper, the dual connections $(\nabla,\nabla^{*})$ of a SMAT will be called the SMAT connections. If the curvature of one (or equivalently both) of them vanishes, the SMAT $(M,g,\nabla,\nabla^{*})$ will be called partially flat( see \cite{henmi2019statistical}). As a statistical structure may be induced by a contrast function (divergence function), a SMAT structure can be equally induced by a pre-contrast function $\text{D}$ that is defined in \cite{henmi2011geometry}. For a given pre-contrast function $\text{D}$, we can define a dualistic structure $(g,\nabla,\nabla^*)$ by:
\begin{equation*}
\begin{cases}
    & g_{ij}(\theta)=-\frac{\partial \text{D}}{\partial \theta_{2}^{j}}|_{\theta=\theta^{2}=\theta^{1}},\\
    & \Gamma_{kj,i}^{*}(\theta)=g(\nabla_{\partial_{k}}^{*}\partial_{j},\partial_{i})=-\frac{\partial^{2} \text{D}}{\partial \theta^{2 k}\partial \theta^{2 j}}|_{\theta=\theta^1=\theta^2},\\
    &\Gamma_{ki,j}(\theta)=g(\nabla_{\partial_{k}}\partial_{i},\partial_{j})=-\frac{\partial^{2} \text{D}}{\partial \theta^{1 k}\partial \theta^{2 j}}|_{\theta=\theta^1=\theta^2} .
\end{cases}
\end{equation*}
It is easy to infer that the connection $\nabla^{*}$ is symmetric, and the connection $\nabla$ is not necessarily symmetric. Hence, it follows from the definition \ref{th:smat} that $(g,\nabla,\nabla^{*})$ is a SMAT structure on $M$. In \cite{henmi2011geometry}, Henmi, M., Matsuzoe, H proved that the SMAT arrives in “classical”
statistics through an estimating function. In particular, both authors demonstrated that on a parametric statistical model, an estimating function defines a pre-contrast function and obviously a SMAT.
\vspace{0.2cm}

Figure \ref{fig:connections} summarises all the different connections used in Information Geometry which will be useful in this paper.

\begin{figure}[H]
    \centering
    \begin{tikzpicture}[yscale = 0.7]
    \draw (0,0) circle(5);
    \draw (0,0.25) circle (4 and 3);
    \fill[gray!20, opacity = 0.5] (0,0.25) circle (4 and 3);
    \draw (0,-2) circle (4 and 2.3);
    \fill[gray!20, opacity = 0.5] (0,-2) circle (4 and 2.3);
    \draw (0,-1) circle(1.5 and 1);
        \draw (0,-0.5) circle(2);
    \fill[gray!20, opacity = 0.5] (0,-0.5) circle(2);
    \fill[gray!20, opacity = 0.5] (0,-1) circle(1.5 and 1);
    \node at (0,4){\begin{tabular}{c}Dual connections 
    \end{tabular}};
    \node at (0,2.5){Smat connections };
   \node at (0,-3.3){\begin{tabular}{c}Auto-dual connections 
    \end{tabular}};
    \node at (0,0.8){Statistical connections };
    \node at (0,-1){Levi-Civita };
\end{tikzpicture}
    \caption{The connections in information Geometry.}
    \label{fig:connections}
\end{figure}
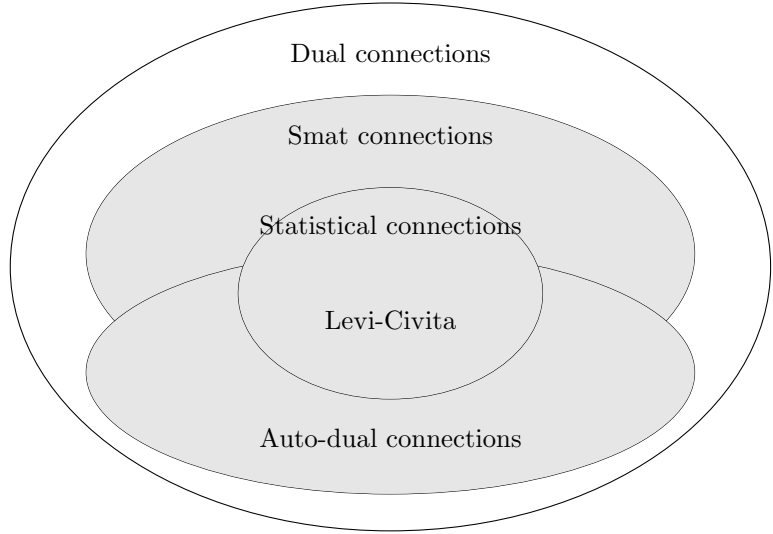

\section{ Hamiltoniann structure and SMAT connections}
\label{se:Ham}
In this section, we shall clarify that smat structures are found in multiple geometric structures, and a thorough investigation of them yields Kahler geometry results. We shall provide a new characterization of Hamiltonian structure in terms of SMAT connections. 
\subsection{Parallel section on the canonical bundle of SMAT }
 Let's select a symmetric connection $\nabla$ on the tangent bundle $TM$ and denote by $^*\nabla$ the dual connection of $TM^{*}$ induced by $\nabla$. It is defined by:
\begin{equation}
< {}^*\nabla_{Z}\beta,Y>+\beta(\nabla_{Z}Y)=Z(\beta(Y))\quad Y,Z\in \Gamma(TM), \beta\in \Gamma(TM^{*}),
\end{equation}
where $<,>$ denotes the pairing of vectors and covectors. Relying on \cite{abe1985general}, it is well known that there exists a unique connection on $TM\otimes TM^{*}$ defined by: 

\begin{equation}
    \widehat{\nabla}(X\otimes \beta)=\nabla X\otimes \beta+X\otimes\nabla^{*}(\beta),\,\, \forall X\in \Gamma(TM), \beta\in \Gamma(TM^{*}).
\end{equation}

\vspace{2ex}
Identify $TM\otimes TM^{*}$ with $\text{End}(TM)$(the bundle of endomorphism of $TM$). One can therefore define a connection $\widetilde{\nabla}$ on $\text{End}(TM)$ by:

\begin{equation}
\nabla^{*}_{X}\Theta Y=(\widetilde{\nabla}_{X}\Theta)Y+\Theta\nabla_{X}Y,\quad \forall X,Y\in \Gamma(TM)
\label{eq:parrar}
\end{equation}
where $\Theta\in\Gamma(\text{End}(TM))$. Now let's opt for some Riemannian metrics $g$ on $M$ and denote the musical isomorphism from $TM$ to $TM^{*}$ by $\flat$, which is indicated as follows: 

\begin{equation}
\flat:TM\rightarrow TM^{*}<\flat(X),Y>=g(X,Y),\quad X,Y\in \Gamma(TM)
\label{eq:flat}
\end{equation}

We can define a connection on $TM$ by:

 $$\nabla^{*}=\flat^{*}(^*\nabla)=\flat^{-1}\circ\nabla\circ\flat.$$
With regard to the above construction, we deduce that $(M,g,\nabla,\nabla^{*})$ is a SMAT.
\vspace{2ex}

\begin{defn}
 A given section $\Theta\in \Gamma(\text{End}(TM))$ is said to be parallel if $\widetilde{\nabla}\Theta=0$, i.e., if
\begin{equation}
\widetilde{\nabla}_{X}\Theta Y=\Theta\widetilde{\nabla}_{X}Y\,, \, \forall X,Y\in \Gamma(TM).
\end{equation}
\label{def:para}
\end{defn}
Grounded on \cref{eq:parrar}, the previous equation is equivalent to:
\begin{equation}
\nabla^{*}_{X}\Theta Y =\Theta\nabla_{X}Y,\,\, \forall X,Y\in\Gamma(TM).
\label{eq:31}
\end{equation}
Using the difference tensor $U\in TM^{*}\otimes TM^{*}\otimes TM$, defined by: 
$$ U=\nabla^{*}-\nabla;$$
 \cref{eq:31} is equivalent to either of the two equations:
        \begin{align}\label{eqparallel}
        & \nabla \Theta = -U(.,\Theta),\\
        & \nabla^*\Theta = \Theta U(.,.).
 \end{align}
 \label{eq:17}
  \begin{proposition}
Let's consider a SMAT $(M,g,\nabla,\nabla^{*})$, and let $\Theta$ be a parallel section on $\text{End}(TM)$. The following assertions are therefore equivalent:
\begin{enumerate}
    \item $\nabla \Theta=0$,
    \item $\nabla$ is the Levi-Civita connection of the metric $g$. 
\end{enumerate}     
 \end{proposition}
\begin{proof}
 By a direct computation, we get:
    \begin{equation}
        (\nabla^{*}_{Z}g)(\Theta X,Y)=Z(g(\Theta X,Y))-g(\nabla^{*}_{Z}\Theta X,Y)-g(\Theta X,\nabla^{*}_{Z}Y)
    \end{equation}
    Now, using the duality condition, we obtain:
\begin{eqnarray}
    (\nabla^{*}_{Z}g)(\Theta X,Y)&=&-g(\nabla^{*}_{Z}\Theta X-\nabla_{Z}\Theta X,Y)\\ &
    =&-g(U(Z,\Theta X),Y)\\ &
    =&g((\nabla_{Z}\Theta)X,Y)
\end{eqnarray}
Hence, by a simple calculation:
\begin{equation}
    (-\nabla_{Z} g)(\Theta X,Y)=(\nabla^{*}_{Z} g)(\Theta X,Y)=g((\nabla_{Z}\Theta)X,Y).
    \label{eq:28}
\end{equation}
The proposition follows from \cref{eq:28}.
\end{proof}

Let's consider $K_{M}$ the skew-symmetric endomorphisms of $TM$. We  can naturally endow $K_{M}$ with the connection $\widetilde{\nabla}$ defined in definition \ref{def:para}. In this paper, the vector $K_{M}$ will be called a canonical bundle of $M$.
\vspace{2ex}

\begin{defn} 
Consider a $2d+1$-dimensional manifold $M^{2d+1}$. A parallel section $\Theta$ on the canonical bundle $K_{M^{2d+1}}$ is called maximal if its rank is $2d$.
\end{defn}

\subsection{SMAT connections and existence of Hamiltoniann structure}
In this subsection, our central focus is to identify a Hamiltonian structure which is a basic notion in this paper. Subsequently, we prove the main theorem of this section.
\vspace{2ex}

\begin{defn}
    A Hamiltonian structure on $M^{2d+1}$ is a closed 2-form $\Omega$  of maximal rank, \textit{i.e.} such that $\Omega^{d}$ is nowhere-vanishing.   
We refer to the pair $(M^{2d+1}, \Omega)$ as a Hamiltonian manifold.
\end{defn}
\vspace{2ex}

 Let's associate with the Hamiltonian structure $\Omega$ the 1-dimensional foliation  $\ell_{\Omega}\subset TM^{2d+1}$ expressed by: 
 $$\ell_{\Omega}=\cup_{p}\{p,\textnormal{ker}(\Omega_{p})\},$$
 where $\text{ker}(\Omega_{p})$ is the kernel of the linear map $\Omega_{p}: T_{p}M^{2d+1}\rightarrow T^{*}_{p}M^{2d+1}$ defined by $w\rightarrow\Omega_{p}(w,.).$ Note that for each $p\in M^{2d+1}$, the 2-form $\Omega_{p}$ is nondegenerate on $\text{T}_{p}M^{2d+1}/\text{ker}(\Omega_{p})$. Hence, it is a symplectic form on this vector space and therefore defines an orientation on it. If $M^{2d+1}$ is oriented, there exists a global 1-form $\alpha$ on $M^{2d+1}$ such that $\text{T}_{p}M^{2d+1}/\text{ker}(\Omega_{p})=\text{ker}(\alpha_{p})$. In this case, the Hamiltonian structure can be defined as an almost cosymplectic structure( almost contact) $(\Omega,\alpha)$ such that $\Omega$ is closed (cf. \cite{libermann1959automorphismes}). 
A Hamiltonian manifold is the odd dimensional analogue of a symplectic manifold. To explore this relation, let's consider $(M^{2d+2}, \Omega)$ a symplectic manifold and $M^{2d+1}$ a hypersurface on $M^{2d+2}$. The restriction $\Omega_{M^{2d+1}}$ of $\Omega$ is a Hamiltonian structure on $M^{2d+1}$. If we assume that $\ell_{\Omega}$ is a regular foliation or in other words if $M^{2d+1}/\ell_{\Omega_{M^{2d+1}}}$ is manifold; the manifold $M^{2d+1}/\ell_{\Omega_{M^{2d+1}}}$ naturally inherits a symplectic structure. In the literature, Hamiltonian manifolds often bear other names, quasi-contact manifold \cite{casals2015almost}, or 2-calibrated manifold \cite{ibort2004lefschetz}. For further information about Hamiltonian manifolds refer back to\,\cite{mcduff2017introduction}, \,\cite{cieliebak2010stable}, \,\cite{eliashberg2006geometry},\,\cite{wendl2010open},\, \cite{acakpo2022stable}.\\
Let us now describe the main results of this section.
\vspace{1ex}

\begin{thm}\label{th:mainTheorem} 
An odd dimensional manifold $M^{2d+1}$ admits a Hamiltonian structure if and only if its canonical bundle admits a maximal parallel section.
\end{thm}
\begin{proof}
 Let
$\pi: W^{2d+2}=M^{2d+1}\times\mathbb{R} \rightarrow M^{2d+1}$ be the
canonical projection and denote by $i(x)=(x,0):M^{2d+1}\rightarrow W^{2d+2}$ the zero section. Consider a Hamiltonian structure $\Omega$ on $M^{2d+1}$. Now we can construct an almost symplectic form $\widetilde{\Omega}$ on $W^{2d+2}$ such that $i^{*}\Tilde{\Omega}=\Omega$. Applying Gromov's $h$-Principle \cite{gromov1969stable} on the open manifold $W^{2d+2}$, there exists a family of almost symplectic $\widetilde{\omega}_{t\in[0,1]}$ on $W^{2d+2}$ such that $\widetilde{\omega}_{0}=\widetilde{\Omega}$ and $\widetilde{\omega}_{1}=\omega$ is closed. According to \cite{lichnerowicz1982deformations,bieliavsky2006symplectic}, the symplectic manifold $(W^{2d+1},\omega)$
admits infinitely mumerous symplectic connections. A construction is provided in \cite{bieliavsky2006symplectic} by: 
\begin{equation}
\widetilde{\nabla}_{X}Y= \nabla^0_{X}Y+\frac{1}{3}B(X,Y)+ \frac{1}{3}B(Y,X)\quad,
\end{equation}
where $\nabla^0$ is a symmetric connection on $W^{2d+2}$ and $B\in T^{2}_{1}(W^{2d+2})$ is defined by:
\begin{equation}
    \nabla^0_{X}\omega (Y,Z)=\omega(B(X,Y),Z).
\end{equation}
Consider  Darboux charts $(\widetilde{U}^{i}=I\times U^{i},x^{\Bar{0}},x^{i})$ of $(W^{2d+2},\omega)$, such that $(U^{i},\, x^{i})\, i:0,1,..,2d$; where $x^{\Bar{0}}$ is the parameter along $\mathbb{R}$. With these Darboux charts $(\widetilde{U}^{i}=I\times U^{i},x^{\Bar{0}},x^{i})$, we have the following equation:
\begin{equation}
\widetilde{\Gamma}_{ijk}=\widetilde{\Gamma}^{l}_{ij}\omega_{lk}\in S^{3}T^{*}W^{2d+2},
\label{eq:1000}
\end{equation}
where $\widetilde{\Gamma}^{l}_{ij}$ are the Christoffel symbols of the symplectic connection $\widetilde{\nabla}.$ It is well known that the difference  $U\in T^{2}_{1}(W^{2d+2})$ between two symplectic connections $\widetilde{\nabla},\widetilde{\widetilde{\nabla} }$ on $(W^{2d+2},\omega)$
satisfies $\omega(U(.,.),.)\in S^{3}T^{*}W^{2d+2}$. Therefore, with Darboux coordinates, we get:
\begin{equation}
\widetilde{U}_{ijk}=U^{l}_{ij}\omega_{lk}\in S^{3}T^{*}W^{2d+2}.
\label{eq:2000}
\end{equation}

In the Darboux charts $(\widetilde{U}^{i},x^{\Bar{0}},x^{i})$, the symplections connections $\widetilde{\nabla},\widetilde{\widetilde{\nabla} }$ are provided by:
\begin{equation}
\widetilde{\widetilde{\Gamma}}^{l}_{\Bar{0}\Bar{0}}=\widetilde{\Gamma}^{l}_{\Bar{0}\Bar{0}}+U^{l}_{\Bar{0}\Bar{0}}.
\end{equation}

Now, referring to \cref{eq:1000,eq:2000}, we can select $\widetilde{U}$ as follows:

\begin{equation}
\widetilde{U}_{ijk}=0\,,\widetilde{U}_{i0j}=0\, ,\widetilde{U}_{00i}=-\widetilde{\Gamma}^{l}_{00}\omega_{li}.
\end{equation}

Hence, we obtain:

\begin{equation}
    \widetilde{\widetilde{\Gamma}}^{l}_{\Bar{0}\Bar{0}}=0.
\end{equation}

From the above analysis, we deduce that on the symplectic manifold $(W^{2d+2},\omega)$, we can find a symplectic connection $\widetilde{\widetilde{\nabla}}$ satisfying in a local Darboux chart the following condition: 
 \begin{equation}
\widetilde{\widetilde{\nabla}}_{\partial_{\Bar{0}}}\partial_{\Bar{0}}=0.
 \end{equation}
 
Now, identify $M^{2d+1}$ with the leaf $x^{\Bar{0}}=0$ of $W^{2d+2}$. We can define by restriction on this leaf a connection by:

$$\nabla_{X}Y=\widetilde{\widetilde{\nabla}}_{X}Y,$$ 
for any vectors fields $X,Y$ on $M^{2d+1}$. The restriction of the symplectic form $\omega$ on $M^{2d+1}$ is a Hamiltonian form which satisfies:
\begin{equation} 
\nabla\omega=\widetilde{\widetilde{\nabla}}\omega=0.
\label{eq:37}
\end{equation}
Select an almost Hermitian $(\Tilde{g},\Tilde{\Theta})$ on $(W^{2d+2},\omega)$ and denote by $(g,\Theta)$  the restriction of $(\widetilde{g},\widetilde{\Theta})$ on $M^{2d+1}$. It is straightforward to check that:

 \begin{equation}
(\nabla_{Z}\omega)(X,Y)= g(\nabla^{*}_{Z}\Theta X-\Theta \nabla_{Z}X,Y),
\label{eq:38}
 \end{equation}
 
for any vector field $X,Y,Z$ on $M^{2d+1}$. Relying upon \cref{eq:37}, we obtain 

\begin{equation*}
    \nabla^{*}_{X}\Theta Y =\Theta\nabla_{X}Y
\end{equation*}
for any vector field $X,Y$ on $M^{2d+1}$. Thus, we conclude that the canonical bundle  of $M^{2d+1}$ admits a maximal parallel section.\vspace{0.2cm}

The converse direction is an immediate consequence of construction of parallel section on $K_{M^{2d+1}}$. Let us consider $\Theta$ a maximal parallel section on the canonical bundle of $M^{2d+1}$. We can therefore define a 2-form  $\Omega_{\Theta}$ by:
$$ \Omega_{\Theta}(X,Y)=g(\Theta X,Y),$$
for all vector fields $X, Y$ on $M^{2d+1}$. It follows from a direct computation that
   \begin{align}
       (\nabla_Z \Omega_{\Theta})(X,Y) & = g(\nabla^{*}_{Z}\Theta X-\Theta \nabla_{Z}X,Y),
       \label{eq:para}
   \end{align}
for all $X,Y,Z$ on $M^{2d+1}$. With regard to the above \cref{eq:para}, we obtain: 

   \[d\Omega_{\Theta}=0.\]

Therefore we deduce that $(M^{2d+1},\Omega_{\Theta})$ is a Hamiltonian manifold.
\end{proof}

\begin{remark}
There exists an alternative technique to prove the necessary parts of the previous theorem in the case where $M^{2d+1}$ is oriented. In that case, the Hamiltonian manifold $(M^{2d+1},\Omega)$ can be viewed as almost cosymplectic $(M^{2d+1},\Omega,\alpha)$ such that $d\Omega=0$ and  $\Omega^{d}\wedge\alpha$ is a volume form on $M^{2d+1}.$ 
On the co-dimension 1 distribution $\ell_{\alpha}=\text{ker}(\alpha)$, there always exists a compatible almost complex structure $\Theta:\ell_{\alpha}\rightarrow \ell_{\alpha}$. Given an almost complex structure $\Theta$ on $\ell_{\alpha}$, we can extend it on $TM^{2d+1}$ by requiring $\Theta|_{\ell_{\Omega}}=0$ and $\ell_{\alpha}=\text{Im}(\Theta)\subset TM^{2d+1}$. Now choose a Riemaniann metric $g$ on $M^{2d+1}$ such that $g|_{{\ell_{\alpha}}}=\Omega(.,\Theta)$. Taking into account the above construction of the metric $g$, it follows that the Hamiltonian 2-form $\Omega$ is finally defined by $\Omega=g(\Theta,.)$. Noting that for any point $p\in M^{2d+1}$, there exists neighborhood $U^{i}$ of $p$ and local Darboux coordinates $ (z,x^{i},y^{i})$ such that $\Omega=dx^{1}\wedge dy^{1}+...+dx^{d}\wedge dy^{d}$ (see \citet{libermann2012symplectic}, \citet{godbillon1969geometrie}, \citet{torres2009geometry}). Now, let's
define on $U^{i}$ a connection $\nabla^{i}$ such that $\Gamma^{l}_{ij}=0.$ Based on this choice, we obviously get :
\begin{equation}
\partial_{k}\Omega_{ij}=\Omega_{il}\Gamma^{l}_{kj}+\Omega_{lj}\Gamma^{l}_{ki}.
\end{equation}

Hence, $\nabla^{i}$ is a local symplectic connection(ie $\nabla^{i}\Omega=0$). Now, using a covering of $M^{2d+1}$ by local Darboux coordinates neighborhoods, moving to a locally finite refinement  $(U_{i})_{i\in \mathcal{A}}$ and opting for a partition of unity $(\rho_{i})_{i\in \mathcal{A}}$ subordinate to the open covering $(U_{i})_{i\in \mathcal{A}}$, we can glue up the local connection symplectic $\nabla^{i}$ on $U^{i}$ to a global one $\nabla$ by
\begin{equation}
    \nabla=\sum_{i\in \mathcal{A}}\rho_{i}\nabla^{i}.
\end{equation}
Obviously, $\nabla\Omega=0.$ Now, proceedings with the same computation as in the above theorem, we get:
\begin{equation}
(\nabla_{Z}\Omega)(X,Y)= g(\nabla^{*}_{Z}\Theta X-\Theta \nabla_{Z}X,Y)\,\,\,  \forall X,Y\in \Gamma(TM^{2d+1}).
 \end{equation}

Thus,
$$\nabla^{*}_{Y}\Theta X=\Theta \nabla_{Y}X,\,\,\,  \forall X,Y\in \Gamma(TM^{2d+1}).$$

\end{remark}

To summarise this section, Theorem \ref{th:mainTheorem} indicates that any Hamiltonian structure on an odd-dimensional always arises from a maximal parallel section of its canonical bundle. Theorem \ref{th:mainTheorem} manifests a link between the symplectic geometry and quantum information geometry. 

\section{Stable Hamiltonian structures and SMAT adapted connections }
\label{se:stable H and smat}
In this section, we review certain basic concepts of a stable Hamiltonian structure needed throughout this paper. A stable Hamiltonian structure (SHS) stands for a generalization of contact, cosymplectic and coKhaler structures. It originates from the conjecture of  \citet{WEINSTEIN1979353} which states that a simply connected hypersurface $W^{2d+1}$ of a symplectic manifold $(M^{2d+2},\Omega)$ which carries a contact form $\alpha$ such that $d\alpha=j^{*}\Omega$ (such  submanifolds  are called "hypersurface of contact type") admits at least a closed orbit. Here, $j:W^{2d+1}\rightarrow M^{2d+2}$ is the inclusion map. \citet{hofer1994symplectic} introduced stable Hamiltonian structures as a hypersurface for which the Weinstein conjecture can be proved. Further interest in stable
Hamiltonian structures stems from the extension of the proof of the Weinstein conjecture in dimension three in
\cite{hutchings2009weinstein,rechtman2010existence}. Their theorems state that for a closed oriented connected 3-dimensional stable Hamiltonian manifold which is not a $\mathbb{T}^{2}$-bundle over $\mathbb{S}^{1}$, the associated Reeb vector field
has a closed orbit. Topological perspectives of stable Hamiltonian manifold were extensively investigated in \cite{cieliebak2015first}, \cite{eliashberg2010introduction}, \cite{bourgeois2003compactness}, \cite{eliashberg2006geometry}, \cite{latschev2011algebraic}, \cite{niederkruger2011weak}, \cite{siefring2011intersection}.\vspace{0.2cm}

Let $(M^{2d+1},\Omega)$ be a Hamiltonian manifold of dimension $2d + 1$. A 1-form $\alpha$ is called a stabilizing 1-form, if it satisfies the following conditions:
    \begin{align}
  &  \alpha\wedge\Omega^{d}\ne 0,\\
  &  \text{ker}(\Omega)\subset\text{ker}(d\alpha).
 \end{align}
 
\begin{defn}
    A Hamiltonian manifold $(M^{2d+1},\Omega)$ is called stable, if there exists a stabilizing 1-form. The pair $(\Omega,\alpha)$ is called a stable Hamiltonian structure.
\end{defn}

A stable Hamiltonian structure $(\Omega,\alpha)$ induces a canonical Reeb vector field $E$ such that:
\begin{equation}
    i_{E}\Omega=0,\quad i_{E}\alpha=1.
\end{equation}
The stabilizability condition can be reformulated as: 
\begin{equation}
    i_{E}\alpha=1\quad\text{and}\quad i_{E}d\alpha=0.
    \label{eq:sta}
\end{equation}

Based on a straightforward computation, the previous condition implies that:
\begin{equation}
  i_{E}(\alpha\wedge\Omega^{d})=\Omega^{d}\quad\text{and}\quad L_{E}(\alpha\wedge\Omega^{d})=0
\end{equation}

Note that the condition $\alpha\wedge\Omega^{d}\ne 0$ implies that $\alpha$ is nowhere vanishing. Then, it induces the hyperplane distribution defined by:

$$\ell_{\alpha}=\text{ker}(\alpha)$$

The stabilzation condition implies that the pair $(\Omega,\alpha)$ determines a splitting of the tangent bundle of $M^{2d+1}$ as follows:

\begin{equation}
    TM^{2d+1}=(\ell_{\Omega},E)\oplus (\ell_{\alpha},\Omega),
    \label{eq:decomp}
\end{equation} where $(\ell_{\Omega},E)$ is a line bundle and $(\ell_{\alpha},\Omega)$ is a symplectic vector bundle. There exists a natural isomorphism defined by:
$$\flat_{\Omega,\alpha}: TM^{2d+1}\rightarrow TM^{*2d+1}:X\rightarrow i_{X}\Omega+\alpha(X)\alpha$$ 
Therefore, we have:
$$\flat_{\Omega,\alpha}(E)=\alpha.$$

Note that there is an analogue of Darboux theorem for an shs manifold. The Darboux theorem states that, locally at $p\in M^{2d+1}$, there exist local charts $(U^{i},x^{0},x^{i},x^{i+d})$, called Darboux charts such that:

\begin{equation}
    E=\frac{\partial}{\partial{x^{0}}},\,\alpha_{U_i}=dx^{0}+\sum_{i}\alpha^{i}dx^{i}+\alpha^{i+d}dx^{i+d}\quad\text{and}\quad \Omega_{U_i}=\sum_{i}dx^{i}\wedge dx^{i+d}
    \label{eq:locass}
\end{equation}
(see \citet{libermann2012symplectic}, \citet{godbillon1969geometrie}, \citet{torres2009geometry}, \citet{siefring2011intersection}). The stabilization condition \cref{eq:sta} translates to:

\begin{equation}
    \frac{\partial\alpha^{i}}{\partial x^{0}}=\frac{\partial\alpha^{i+d}}{\partial x^{0}}=0.
    \label{eq:51}
\end{equation}

Denote $\partial_{i}=\frac{\partial}{\partial x_{i}}$. In Darboux charts, $(U^{i},(\partial_{i})_{i=0}^{i=2d})$, $\partial_{0}=E$ and $(\partial_{i})_{i=1}^{i=2d}$ is a symplectic basis of $\ell_{\alpha}$ adapted to the symplectic form $\Omega$. Note that locally, we have:

\[
\ell_{\Omega}=<\frac{\partial}{\partial x^{0}}>,\,\ell_{\alpha}=<\frac{\partial}{\partial x^{i}}-\alpha^{i}\frac{\partial}{\partial x^{0}}, \frac{\partial}{\partial x^{i+d}}-\alpha^{i+d}\frac{\partial}{\partial x^{0}}>,
\]
(see \citet{de2023review} ). According to \citet{libermann1962quelques},\cite{libermann1959automorphismes}, an shs structure $(\Omega,\alpha)$ on $M^{2d+1}$, can be viewed as an almsot-cosymplectic structure $(\Omega,\alpha)$ satisfying:

$$d\Omega=0\,\,\text{and}\,\,\ell_{\Omega}\subset 
 \ell_{d\alpha}.$$
 
 Hence, in the language of $G$-structures, an shs structure is a ${1}\times Sp(d,\mathbb{R})$-structure. Now, in the intersection of two Darboux charts, namely $(V^{j},y^{0},y^{i},y^{i+d})$ and $(U^{i},x^{0},x^{j},x^{j+d})$
such that $U^{i}\cap V^{j}\ne\emptyset$, we have:
$$
\frac{\partial}{\partial{y}^{0}}=\frac{\partial}{\partial{x^{0}}},\,\, \frac{\partial}{\partial y^{i}}=B^{i}_{j}\frac{\partial}{\partial x^{j}}+C^{i}_{j}\frac{\partial}{\partial x^{j+d}},\,\, \frac{\partial}{\partial y^{i+d}}=-C^{i}_{j}\frac{\partial}{\partial x^{j}}+B^{i}_{j}\frac{\partial}{\partial x^{j+d}}$$
and 
$$
dy^{0}=dx^{0}\,\,, dy^{i}=B^{i}_{j}dx_{j}-C^{i}_{j}dx_{j+d}\,\,,dy^{i+d}=C^{i}_{j}dx_{j}+B^{i}_{j}dx_{j+d},$$
where $B,C\in GL(d,;\mathbb{R}).$ Hence, we can cover $M^{2d+1}$ by the atlas $\mathcal{A}=\{U^{i}, x^{i}\}_{i\in A}$ such that the shs structure $(\Omega,\alpha)$ on each  $U^{i}$ satisfies the condition of \cref{eq:51}.\vspace{0.2cm}

 We note that in the state of art works, stable Hamiltonian manifolds are often referred to as orientable Hamiltonian manifolds such that the foliation $\ell_{\Omega}$ is geodesible. The stability of the Hamiltonian structure $\Omega$  uniquely depends on the geodesibilty of the foliation $\ell_{\Omega}$ (see \cite{gluck2006dynamical}, \cite{rechtman2010existence} \cite{wadsley1975geodesic}, \cite{sullivan1978foliation}).

\subsection{Basic examples}
We will provide and illustrate fundamental examples
of stable Hamiltonian structures.\vspace{0.2cm}

\textbf{Contact manifolds:}  $(M^{2d+1},\Omega,\alpha)$  such that $\Omega=d\alpha$.\vspace{0.2cm}

\textbf{Contact type hypersurfaces:}
Any hypersurface $M^{2d+1}$ of contact type in a symplectic manifold $(W^{2d+2},\Omega)$ is endowed with a
stable Hamiltonian structure with $\Omega=d\alpha.$\vspace{0.2cm}

\textbf{Cosymplectic manifolds }\cite{libermann1959automorphismes}: $(M^{2d+1},\Omega,\alpha)$ such that $d\alpha=0$. With reference to \citet{li2008topology}, when $M^{2d+1}$ is closed, 
$M^{2d+1}$ is a symplectic mapping torus \textit{i.e.} ($M^{2d+1}=\frac{Z^{2d}\times[0,1]}{(p,0)\sim (\varphi(p),1)})$, where $Z^{2d}$ is a leaf of the codimension one foliation $\ell_\alpha)$, and $\varphi$ is a symplectomorphism of $(Z^{2d},\widetilde{\Omega})$. By the projection $\pi_{1}:Z^{2d}\times [0,1]\rightarrow Z^{2d}$, $\Omega$ is obtained by glueing up $\widehat{\Omega}=\pi^{*}\widetilde{\Omega}$ by $\varphi$ and resting on \citet{tischler1970fibering}, $\alpha=p^{*}(d\theta)$, with  $\theta$ being the angle coordinate on $\mathbb{S}^{1}$ and $p:M^{2d+1}=Z^{2d}_{\varphi}\rightarrow\mathbb{S}^{1}$. Note that the leaves of $\ell_\alpha$, are all symplectomorphic to each other under the flow of $E$, and the universal cover $\widetilde{M^{2d+1}}$ of $M^{2d+1}$ is diffeomorphic to $\mathbb{R}\times\widetilde{Z^{2d}}$,
where $\widetilde{Z^{2d}}$ is the universal cover of a leaf of the foliation $\ell_\alpha$. For additional details on cosymplectic structure, consult (\cite{reeb1951solutions}, \cite{takizawa1963contact}, \cite{de1996universal}, \cite{albert1989theoreme}, \cite{libermann1962quelques}).\vspace{0.2cm}

\textbf{Circle bundles:} Let $(B^{2d},\hat{\Omega})$ be a symplectic manifold satisfying the prequantization condition$(Per(\hat{\Omega})=c\mathbb{Z}$, with $c\in \mathbb{R}$). Departing from \citet{kostant1973quantization}, there exists a corresponding prequantization space, that is a principal $\mathbb{S}^{1}$-bundle $p: M^{2d+1}\rightarrow B^{2d}$, with a principal connection $\alpha\in\Omega^{1}(M^{2d+1},\mathbb{R})$ satisfying $\Omega=p^{*}\hat{\Omega}=d\alpha$. Hence, $(M^{2d+1},\Omega,\alpha)$ is a SHS manifold. Note that from \citet{gromov1971topological}, there exists a symplectic immersion $j$ from  $(B^{2d},\hat{\Omega})$ into the K\"ahler manifold $(\mathbb{C}P^{N},\Omega^{0})$ for $N$ sufficiently large, such that  $\hat{\Omega}=j^{*}\Omega^{0}.$ The 2-form $\Omega^{0}$ is the Fubini-Study K\"ahler form  defined by $\Omega^{0}=\frac{i}{2\pi}\partial\Bar{\partial}log(1+|z_{i}|^{2})$. 

\subsection{ Canonical connections on shs manifold}

 In \cite{lichnerowiczgeometrie}, the authors proved that on  contact manifold $(M^{2d+1},\alpha)$, there exists a symmetric affine connection such that 
 \begin{equation}
     \nabla E=0,\,\nabla d\alpha=0.
 \end{equation}
 A.Lichnerowicz called these connections contact connections. In this section, we introduce a generalisation of the theorem of \citet{lichnerowiczgeometrie}. Our theorem states that every shs manifold admits infinite connections called  shs connections, which preserves the shs structure $(\Omega,\alpha)$. The shs connections coincide with the contact connections in the contact manifold case. The idea of the proof relies on symplectization processes.
 \vspace{1ex}
 
\begin{thm}
     On every shs manifold $(M^{2d+1},\Omega,\alpha)$, there are always connections such that:
\begin{equation}
    \nabla E=0\,,\nabla\Omega=0.
\label{eq:50}
\end{equation}
\label{th:shs}
\end{thm}
\vspace{-6pt}
\begin{proof}
Assume that $M^{2d+1}$ admits an shs structure $(\Omega,\alpha)$. By symplectization, we can always realize $M^{2d+1}$ as hypersurface in the symplectic manifold $(W^{2d+2}=M^{2d+1}\times ]-\epsilon,\epsilon[,   \widetilde{\Omega}=\Omega-d(t\alpha))$ for sufficiently small $\epsilon> 0$. Here, $t$ denotes the coordinate on $]-\epsilon,\epsilon[$.
By a direct computation, we get 

$$\widetilde{\Omega}=\Omega-dt\wedge\alpha-td\alpha.$$

We denote $t=x^{\hat{0}}$. Using \cref{eq:locass}, locally in Darboux charts  $(\widetilde{U}^{i}=U^{i}\times I ,x^{\hat{0}},x^{i})$ $i=0,1,..,2d$ , we obtain:

$$\widetilde{\Omega}=dx^{i}\wedge dx^{i+d}-dt\wedge(dx^{0}+\alpha^{i}dx^{i}+\alpha^{i+d}dx^{i+d})-td(\alpha^{i}dx^{i}+\alpha^{i+d}dx^{i+d}).$$

It follows from the stabilization condition \cref{eq:51} that

\begin{equation}
    \widetilde{\Omega}=(1+t(\frac{\partial\alpha^{i}}{\partial x^{i+d}}-\frac{\partial\alpha^{i+d}}{\partial x^{i}}))dx^{i}\wedge dx^{i+d}-dt\wedge dx^{0}-\alpha^{i}dt\wedge dx^{i}-\alpha^{i+d}dt\wedge dx^{i+d}.
    \label{eq:locap}
\end{equation}

 With reference to \cref{eq:locap}, the only non-zero components of $\widetilde{\Omega}$ are:
 
\begin{equation}
\widetilde{\Omega}_{i,i+d}=1+x^{\hat{0}}(\frac{\partial\alpha^{i}}{\partial x^{i+d}}-\frac{\partial\alpha^{i+d}}{\partial x^{i}}),\, \widetilde{\Omega}_{\hat{0},0}=-1, \,\widetilde{\Omega}_{\hat{0},i}=-\alpha^{i}, \,\widetilde{\Omega}_{\hat{0},i+d}=-\alpha^{i+d}.
\label{eq:localcomp}
\end{equation}

Now, we shall use the same strategy as in Theorem \ref{mainTheorem} through replacing the closed 2-form $\omega$ constructed by the use of Gromov's $h$-principle with $\widetilde{\Omega}$. Let us select a symplectic connection $\widetilde{\nabla}$ on  $(W^{2d+2},\widetilde{\Omega})$ constructed in theorem \ref{thm:principal}.
With Darboux coordinates on $(\widetilde{U}^{i}= U^{i}\times I,x^{\hat{0}},x^{i})$ $i=0,..,2d$ of $(W^{2d+2},\widetilde{\Omega})$, we get:

$$\widetilde{\Gamma}^{l}_{ij}\widetilde{\Omega}_{lk}=\widetilde{\Gamma}^{l}_{ik}\widetilde{\Omega}_{lj},\, \text{for}\,  i,j,k,l=\hat{0},0,1,.,2d.$$
Relying upon \cref{eq:localcomp}, it follows:

\begin{equation}
\widetilde{\Gamma}^{0}_{i0}=0.
\label{eq:coponul}
\end{equation}

Now, identify $M^{2d+1}$ with the leaf $t=0$ of $W^{2d+2}$. We can define restriction connections $\nabla$ on this leaf by $\Gamma^{l}_{ij}=\widetilde{\Gamma}^{l}_{ij}$. Hence, there are connections $\nabla$ on $M^{2d+1}$ satisfying 
 \begin{equation}
     \nabla\Omega=0.
     \label{eq:54}
 \end{equation}
 
From  the above analysis, in Darboux charts $(U^{i},x^{0},x^{1},...,x^{2d})$ on $(M^{2d+1},\Omega,\alpha)$, we have:

$$\Gamma^{l}_{ij}\Omega_{lk}=\Gamma^{l}_{ik}\Omega_{lj},\, \text{for}\,  i,j,k,l=0,1,.,2d.$$
\vspace{1ex}
Bear in mind that $i_{E}\Omega=0$. Therefore in Darboux charts, we have $\Omega_{l0}=0$. Consequently, we deduce that
$$\Gamma^{l}_{i0}\Omega_{lk}=0.$$ 
In Darboux charts, we have $\Omega_{i,i+d}=1$. Then,

\begin{equation}
\Gamma^{l}_{i0}=0\,, \text{for}\,
l=1,..,2d.
\end{equation}
Now, resting on \cref{eq:coponul}, we get:

$$\Gamma^{0}_{i0}=\widetilde{\Gamma}^{0}_{i0}=0.$$

Consequently, we have:

$$\Gamma_{i,0}^{l}=0\,, \text{for}\,\,
l=0,1,..,2d.$$
or equivalently 
\begin{equation}
    \nabla E=0.
\label{eq:Reebpa}
\end{equation}

Finally, with respect to \cref{eq:Reebpa,eq:54}, we obtain

$$ \nabla E=0\,,\nabla\Omega=0.$$

This therefore proves our assertion.
\end{proof}

\begin{remark}
As presented in section \ref{se:Ham}, locally we can also define a symmetric
connection $\nabla^{i}$  to have zero Christoffel symbols in Darboux charts $(U^{i},x^{i})\,\,\forall i:0,1,...,2d$. We hence get
$$
\nabla^{i}E_{U_i}=0,\quad \nabla^{i}\Omega_{U_i}
=0
$$
Globally, we can glue $\nabla^{i}$ using a partition of unity $(\rho_{i})_{i\in\mathcal{A}}$ subordinate to the open covering $(U_{i})_{i\in \mathcal{A}}$, as $\nabla=\sum_{i}\rho_{i}\nabla^{i}$. We thus deduce that 
$$\nabla E=0\quad\text{and}\quad\nabla\Omega=0.$$
The desired connections are hence obtained.
\end{remark}

The space of  connections satisfying \cref{eq:50} will be denoted by $\mathcal{C}M^{2d+1}_{(\Omega,\alpha)}$. These connections will be called \textbf{ shs-connections}. 
 In what follows, we attempt to prove that there are $E$-invariant shs connections on an shs manifold.
\vspace{2ex}

\begin{proposition}
    Let $(M^{2d+1},\Omega,\alpha)$ be an shs manifold such that the leaves of $\ell_{\Omega}$ are closed. The shs manifold $(M^{2d+1},\Omega,\alpha)$   admits $E$-invariant shs connections.
    \label{prop:invar}
\end{proposition}
    
\begin{proof}
First of all, the leaves of the foliation $\ell_{\Omega}$ are closed. Hence, the space of leaves $M^{2d+1}/\ell_{\Omega}$ is a smooth manifold. Denote the manifold $M^{2d+1}/\ell_{\Omega}$ by $B^{2d}$ and the canonical projection by $\pi:X^{2d+1}\rightarrow B^{2d}.$ Since $L_{E}\Omega=0$, thus there exists on $B^{2d}$ a closed 2-form $\widetilde{\Omega}$ such that $\Omega=\pi^{*}\widetilde{\Omega}$. As  $\Omega^{d}=\pi^{*}\widetilde{\Omega^{d}}$ is non-vanishing, then $(B^{2d},\widetilde{\Omega})$ is a symplectic manifold. The stabilization condition $L_{E}\alpha=0$ implies that there is on $B^{2d}$ a 1-form $\widetilde{\alpha}$ such that $\alpha=\pi^{*}\widetilde{\alpha}$. Now, consider  Darboux charts $(\widetilde{U}^{i},x^{i})$ $i=1,..,2d$ on the symplectic manifold $(B^{2d},\widetilde{\Omega})$. We can construct a Darboux chart $U^{i}=\pi^{-1}(\widetilde{U}^{i})$ with coordinates $(x^{0},x^{i})$.
On $U^{i}$, we have:
\begin{equation} 
E=\frac{\partial}{\partial x_{0}},\alpha_{U^{i}}=dx^{0}+\sum_{i}\alpha^{i}dx^{i}+\alpha^{i+d}dx^{i+d},\quad\Omega_{U^{i}}=\sum_{i}dx^{i}\wedge dx^{i+d},
\end{equation}
where $\partial_{0}\alpha^{i}=\partial_{0}\alpha^{i+d}=0$. Let us choose $\widetilde{\nabla}$, a symplectic connection on $(B^{2d},\widetilde{\Omega})$. Naturally, on $(\widetilde{U}^{i},x^{i})$, we have

\begin{equation}
\tilde{\Gamma}^{s}_{kj}\widetilde{\Omega}_{is}=\tilde{\Gamma}^{s}_{ki}\widetilde{\Omega}_{js},
\label{eq:regul}
\end{equation}
where $\widetilde{\Gamma}$ denote the Christoffel symbols of $\widetilde{\nabla}.$
We can define a symmetric connection $\nabla$ on $U^{i}=\pi^{-1}(\widetilde{U}^{i})$ by:

\begin{equation}
\Gamma^{s}_{ij}=\widetilde{\Gamma}^{s}_{ij},\ \Gamma^{s}_{j0}=0,\ \Gamma^{0}_{ij}=0.
\label{eq:regularde}
\end{equation}

Departing from \cref{eq:regul,eq:regularde}, it follows that on $U^{i}$: 
\begin{equation}
\Gamma^{s}_{kj}\Omega_{is}=\Gamma^{s}_{ki}\Omega_{js}.
\end{equation}

Thus, we deduce: 
\begin{equation}
    (\nabla_{\partial_k}\Omega)(\partial_{i},\partial_{j})=0\quad \forall i,j,k=1,..,2d.
\end{equation}

Through a straightforward calculation, we get:
\begin{equation}
    (\nabla_{\partial_j}\Omega)(E,\partial_i)=\partial_{j}.\Omega(E,\partial_i)-\Omega(\nabla_{\partial_j}E,E)-\Omega(E,\nabla_{\partial_j}\partial_i).
\end{equation}

Since $i_{E}\Omega=0$, we obtain: 
\begin{equation}
    (\nabla_{\partial_j}\Omega)(E,\partial_i)=0.
\end{equation}

Using a direct computation, we have:
\begin{equation}
    (\nabla_{E}\Omega)(\partial_i,\partial_j)=\partial_{0}\Omega(\partial_i,\partial_j)-\Gamma^{s}_{0,i}\Omega_{sj}-\Gamma^{s}_{0,j}\Omega_{is}.
\end{equation}

 We obtain from \cref{eq:regularde} that:
 \begin{equation}
     (\nabla_{\partial_0}\Omega)(\partial_i,\partial_j)=0\,\, \forall i,j=1,..,2d.
 \end{equation}
 
Finally, We deduce from the previous identities that: 
\begin{equation}
    \nabla\Omega=0.
\end{equation}

Now, based on \cref{eq:regularde}, we have $\Gamma^{s}_{j0}=0$. Thus,

\begin{equation}
    \nabla E=0.
\end{equation}

As a matter of fact, connection $\nabla$ is an shs connection. Resting on a direct calculation, we have:

 $$(L_{E}\nabla)(Y,Z)=\nabla_{E}\nabla_{Y}Z-\nabla_{\nabla_{Y}Z}E-\nabla_{[E,Y]}Z-\nabla_{Y}\nabla_{E}Z+\nabla_{Y}\nabla_{Z}E,$$
 for any vector fields $Y,Z$ on $M^{2d+1}.$
With the local basis, we have:

\begin{equation}
    (L_{E}\nabla)(\partial_{i},\partial_{j})=R^{\nabla}(E,\partial_{i})\partial_{j}+(\nabla^{2}E)(\partial_{i},\partial_{j}),\,\, \forall i,j=0,1,..,2d.
    \label{eq:68}
\end{equation}
Now, by a direct computation, we get:

    $$(L_{E}\nabla)(\partial_{i},\partial_{j})=\partial_{0}(\Gamma^{s}_{ij})\partial_s.$$
Using \cref{eq:regularde}, we obtain:

$$\partial_0\Gamma^{s}_{ij}=\partial_0\widetilde{\Gamma}^{s}_{ij}=0.$$
It follows from the above that:

   $$ L_{E}\nabla=0.$$
The result follows at once.
\end{proof}

\subsubsection{Description of the space of shs connections}
 A covariant tensor $S$ of rank $q$ is $E$-transverse, if:
 $$i_{E}S=0.$$
Denote by $\mathcal{T}_{E}(M^{2d+1})=\oplus_{p}\mathcal{T}^{q}_{E}(M^{2d+1})$ the tensor algebra of covariant $E$-transverse tensor on $M^{2d+1}$. For an shs structure, $(\Omega,\alpha)$, $(\Omega,d\alpha)\in \mathcal{T}^{2}_{E}(M^{2d+1})\times\mathcal{T}^{2}_{E}(M^{2d+1})$. Each space  $\mathcal{T}^{q}_{E}(M^{2d+1})$ contains the 
subspace $\mathcal{S}_{E}^{q}(M^{2d+1})$ of totally symmetric  covariant tensor fields of rank $q$.
\begin{proposition}\label{eq:112}
The set of shs connections $\mathcal{C}M_{(\Omega,\alpha)}^{2d+1}$ on the shs manifold $(M^{2d+1},\Omega,\alpha)$ is an affine space modelled on the vector space $\mathcal{S}_{E}^{3}(M^{2d+1})$.
\end{proposition}

\begin{proof}
Let us consider $\nabla$ a connection satisfying \cref{eq:50}. It is well known that any connection $\widetilde{\nabla}$ can be expressed as:

$$\widetilde{\nabla}=\nabla+B\,\, \text{where},\,\, B\in T^{2}_{1}M^{2d+1}.$$

 It follows from a straightforward computation that $\widetilde{\nabla}$ satisfies the condition \cref{eq:50} if and only if
  $$ A(X,Y,Z)=\Omega(B(X,Y),Z)\in S^{3}TM^{2d+1}\quad\text{and}\quad i_{E}B=0.$$

  Since $i_{E}\Omega=0$, then we deduce that $ A(X,Y,Z)=\Omega(B(X,Y),Z)\in \mathcal{S}_{E}^{3}(M^{2d+1}).$
  This proves our assertion.
\end{proof}

Let $\nabla$ be an shs connection and let $R^{\nabla}$ be its curvature tensor. Referring to \cite{vaisman1985symplectic}, we can define the analogues of the symplectic curvature for shs connections. The symplectic curvature of an shs connection will be indicated by:

$$S^{\nabla}(X,Y,Z,W)=\Omega(X,R^{\nabla}(Z,W)Y),$$
for any vector fields $X,Y,Z,W$ on $M^{2d+1}$. The components of the curvature on the local Darboux charts are of the form:

\begin{equation}
S^{\nabla}_{ijkl}=\Omega(\partial_i,R^{\nabla}(\partial_k,\partial_l)\partial_j)=\Omega_{im}R^{m}_{jkl},
\end{equation}
\vspace{1ex}
where $R^{\nabla}(\partial_{j},\partial_{k})\partial_{i}=R^{m}_{ijk}\partial_{m}.$ From a direct computation, we get:

\vspace{1ex}

\[ R^{\nabla}(\partial_i,\partial_j)E=(\nabla^{2}E)(\partial_i,\partial_j)-(\nabla^{2}E)(\partial_j,\partial_i), \quad i,j=0,1,..,2d.\]
\vspace{1ex}

Now using \cref{eq:50}, we obtain from a straightforward calculation
that
\[
    (\nabla^2E)(\partial_i,\partial_j) = \nabla_{\partial_i}(\nabla_{\partial_{j}} E) - \nabla_{\nabla_{\partial_{i}} \partial_{j}} E=0,\,\, i,j=0,1,..,2d.
    \]
    
  Thus, we get:
   \begin{equation}
    S^{\nabla}_{i0kl}=0,\,\, \text{for},i,k,l=0,1,..,2d.
    \label{eq:symplcutnu}
    \end{equation}
   
   Resting upon \cref{eq:50}, we have:
   
   \begin{equation}
\partial_{k}\Omega_{ij}=\Omega_{il}\Gamma^{l}_{kj}+\Omega_{lj}\Gamma^{l}_{ki}\,\,\text{for}\, i,j,k=1,..,2d.
\end{equation}
Hence, grounded on \citet{lemlein1957spaces}, \citet{vaisman1985symplectic}, the symplectic curvature $S^{\nabla}$ satisfies symmetry properties as:
\begin{equation}
    S^{\nabla }_{ijkl}=S^{\nabla }_{jikl},\,\,S^{\nabla }_{ijkl}=-S^{\nabla }_{ijkl} \quad \text{for}\,
   i,j,k,l=1,....,2d.
   \label{eq:symplcur}
\end{equation}

\vspace{1ex}
  Based on \cref{eq:symplcur,eq:symplcutnu}, the components of the symplectic curvature of an shs connection in the local Darboux charts satisfy the following identities:

   \begin{equation}
        S^{\nabla}_{i0kl}=0\,\,,  S^{\nabla }_{ijkl}=S^{\nabla }_{jikl},\,\,S^{\nabla }_{ijkl}=-S^{\nabla }_{ijkl} \quad\,\text{for}\, 
   i,j,k,l=0,1,....,2d.
   \end{equation}
   \vspace{1ex}
Resting on \cite{vaisman1985symplectic}, we deduce that the Ricci tensor $Ric^{\nabla}$ of an shs connection is symmetric
\begin{equation} Ric^{\nabla}_{ij}=Ric^{\nabla}_{ji} ,\quad \text{for}\, 
   i,j,k,l=0,1,....,2d.
\end{equation}
\subsection{Compatible SMAT connections on shs manifolds}
The central objective of this subsection is to corroborate that on any shs manifold $(M^{2d+1},\Omega,\alpha)$, there always are smat connections adapted to the distributions $\ell_\Omega$ and  $\ell_{\alpha}$. Remember that a linear connection $\nabla$ on a manifold $M$ is said to be adapted to a distribution $H$ or parallel with respect to $\nabla$, or $\nabla$-geodesic if
$$\nabla_{Y}S\in\Gamma(H),\quad \forall Y\in\Gamma(TM),\forall S\in\Gamma(H),$$
(see \cite{bejancu2006foliations,bejancu2012geometry, walker1958connexions}). 
With reference to \cite{de2011methods}, \cite{sasaki1985shigeo},  \cite{blair2010riemannian}, note that for an shs pair $(\Omega,\alpha)$,  there exist a tensor field $\Theta\in\Gamma(TM^{2d+1})$ and a
compatible Riemannian metric $g$ satisfying:
\begin{equation}
    \Theta^{2}=-I+\alpha\otimes E
    \label{eq:72}
\end{equation}
\begin{equation}
    \Theta E=0,\,\alpha\circ \Theta=0
    \label{eq:73}
\end{equation}
\begin{equation}
    \alpha=g(.,E)
    \label{eq:74}
\end{equation}
\begin{equation}
\Theta^{2}|_{\ell_{\alpha}}=-1
\label{eq:75}
\end{equation}
\begin{equation}
    g=\Omega(,\Theta )+\alpha\otimes\alpha.
    \label{eq:76}
\end{equation}
Therefore, the Hamiltonian form is indicated by:

$$\Omega=g(\Theta .,.)$$

We hence state that $(\Theta, E, \alpha, g)$ is the almost contact Riemannian structure associated with the shs structure $(\Omega,\alpha)$. Note that it is, generally, not unique. The triple $(\Theta, E, \alpha)$, is called an almost contact
structure on $M^{2d+1}.$ In \cite{ogiue1967cocomplex}, the triple $(\Theta,E,\alpha)$ is also called an almost co-complex structure and the compatible metric $g$ is defined in \cref{eq:76} as a co-Hermitian metric. Let's denote by $\mathcal{H}(\ell_{\alpha})$ the space of $\Omega$-compatible almost complex structure of the symplectic bundle $\ell_{\alpha}$ and let's consider a Hermitian metric $g_{|\ell_{\alpha}}=\Omega(.,\Theta)$ on the symplectic bundle $\ell_{\alpha}$. Extend $\Theta$ on $TM^{2d+1}$ as $\Theta|_{\ell_{\Omega}}=0$ and $\ell_{\alpha}=Im(\Theta)\subset TM^{2d+1}$. According to the decomposition in \cref{eq:decomp}, any vector fields $X,Y$ in $TM^{2d+1}$ can be stated as follows:

\begin{equation}
    X=X_{\ell_{\alpha}}+(i_{X}\alpha) E,\, Y=Y_{\ell_{\alpha}}+(i_{Y}\alpha) E,
\end{equation}
where $X_{\ell_{\alpha}}$, $Y_{\ell_{\alpha}}$ 
are the $\ell_{\alpha}$–components of the vector fields $X$ and $Y$. The Hermitian metric $g_{\ell_{\alpha}}$ can be extended as follows:

\begin{equation}
    g(X,Y)=g_{|\ell_{\alpha}}(X_{\ell_{\alpha}},Y_{\ell_{\alpha}})+ (i_{X}\alpha)(i_{Y}\alpha).
\end{equation}

We can easily prove that $\Omega=g(.,\Theta)$ and $\alpha=g(.,E).$ We deduce that $(\Theta, E, \alpha, g)$ constructed above is the almost contact Riemannian structure associated with
the shs structure $(\Omega,\alpha)$. Finally from the above construction, there exists a bijective correspondence between the set of almost contact Riemannian
structures associated with the shs structure $(\Omega,\alpha)$ and the set of $\Omega$-compatible almost complex structure of the symplectic bundle $\ell_{\alpha}.$\vspace{0.2cm}

As a matter of fact, the shs structure $(\Omega,\alpha)$
can also be written as $(\Omega,\alpha,g)$ or $(\Omega,\alpha,\Theta)$ in the sequel.
\vspace{1ex}
\begin{proposition}
  On an shs manifold $(M^{2d+1},\Omega,\alpha)$, there are always SMAT connections suitably adapted to the distributions $\ell_{\Omega}$ and $\ell_{\alpha}.$
   \label{eq:190}
\end{proposition}
   \begin{proof}
   Select an SHS connection $\nabla$ on $(M^{2d+1},\Omega,\alpha)$ and a co-Hermitian metric $g$ on $(M^{2d+1},\Omega,\alpha)$. Using a direct computation, the dual connection $\nabla^{*}$ of $\nabla$ with respect to $g$ satisfies:
   
\begin{equation}
    (\nabla_{X}^{*}\alpha)Y=Xg(Y,E)-g(\nabla^{*}_{X}Y,E),
\end{equation}

for any vector field $X,Y$ in $M^{2d+1}$.
According to the duality condition, we obtain
\begin{equation}
(\nabla_{X}^{*}\alpha)Y=g(Y,\nabla_{X}E),
\end{equation}
for any vector fields $X,Y$ on $M^{2d+1}$. We infer departing from the condition of \cref{eq:50} that:
\begin{equation}
\nabla\Omega=0,\, \nabla^{*}\alpha=0.
\label{eq:77}
\end{equation} 

  With reference to \cref{eq:77}, we obtain by a straightforward computation
  \begin{equation}
   \nabla i_{X}\Omega=i_{\nabla X}\Omega,\, \forall X\in \ell_{\Omega},
   \label{eq:firs}
   \end{equation}
   \begin{equation}
      \nabla^{*} i_{X}\alpha=i_{\nabla^{*} X}\alpha,  \forall X\in \ell_{\alpha}.
      \label{eq:second}
   \end{equation}
Therefore, from \cref{eq:firs,eq:second}, we get:
$$
\nabla \ell_\Omega\subset\ell_\Omega\,, \nabla^{*}\ell_{\alpha}\subset \ell_{\alpha}.$$
The desired adapted smat shs connections are therefore obtained.
\end{proof}

\begin{figure}[H]
    \centering
\begin{center}
    \begin{tikzpicture}[xscale = 2]
    \fill[gray!20] (0,0) -- (2,0) -- (3,2) -- (1,2) -- cycle;
    \draw (0,0) -- (2,0) -- (3,2) node[pos = 0.4, below right]{$\ell_\alpha = \text{ker }\alpha$} node[pos = 0.7, below right]{$\nabla^*$-totally geodesic} -- (1,2) -- cycle;
    \draw[thick] (1.5,-1) -- ++(0,1);
    \draw[dashed, thick] (1.5,0) -- ++(0,1);
    \draw[->, thick] (1.5,1) -- ++(0,2) node[pos = 0.9, left=5pt]{$\ell_\Omega = <E>$}node[pos = 0.9, right=5pt]{$\nabla$-totally geodesic};
\end{tikzpicture}
 \end{center}
\caption{Smat connections on a shs manifold.}
    \label{fig:adaptedconnections}
\end{figure}
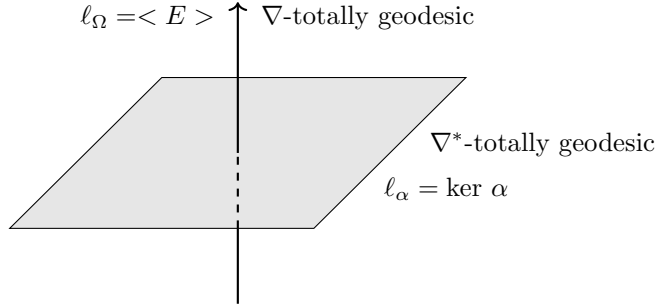
\subsection{The torsion of the dual connection of an SHS connection}
We may denote the Darboux  coordinates by $(x^{0}, z^{i})\in \mathbb{R}\times \mathbb{C}^{d},$ 
where $z^{i}=x_{i}+i x_{i+d}$ and $\Bar{z}^{i}=x_{i}-i x_{i+d}.$
According to Darboux theorem, it is to be noted that in a local Darboux charts $(U^{i}, x^{0}, z^{i})$, we obtain the following local representation of $E,\,\Omega,\, \alpha,\,\Theta:$ 
\begin{align}
    E_{U^{i}}=\frac{\partial}{\partial x^{0}},\, \Omega_{U^{i}}=\frac{i}{2}dz^{i}\wedge d\Bar{z}^{i},\, \alpha_{U^{i}}=dx^{0}+(\frac{\alpha_{i}}{2}+\frac{\alpha_{i+d}}{2i}) dz^{i}+(\frac{\alpha_{i}}{2}-\frac{\alpha_{i+d}}{2i}) d\Bar{z}^{i},\\
    \Theta_{U^{i}}=\frac{1}{2}(\Theta^{j}_{i} dz^{i}\otimes \frac{\partial}{\partial z^{j}}+\Theta^{\Bar{j}}_{i} dz^{i}\otimes\frac{\partial}{\partial\Bar{z}_{j}}+\Theta^{j}_{\Bar{i}} d\Bar{z}^{i}\otimes\frac{\partial}{\partial z_{j}}+ \Theta^{\Bar{j}}_{\Bar{i}}d\Bar{z}^{i}\otimes\frac{\partial}{\partial \Bar{z}_{j}}),  
\end{align}

where $\Theta^{j}_{\Bar{i}}$, $\Theta^{\Bar{j}}_{\Bar{i}}$, $\Theta^{\Bar{j}}_{i}$, $\Theta^{j}_{i}$ are certain functions on $U^{i}$.
\vspace{1ex}
The co-Hermitian metric $g$ satisfies the following identities:
\begin{equation}
   g_{00}=1,\, g_{0i}=0,\,g_{0\Bar{i}}=0,\, g_{ij}=g_{\Bar{i}\Bar{j}}=0\,, g_{i\Bar{j}}=g_{\Bar{j}i}\,,  g_{j\Bar{i}}=\Bar{g}_{i\Bar{j}},
    \label{eq:coherloca}
\end{equation}
Hence locally, we get
\begin{equation} 
g_{U^{i}}=\alpha_{U^{i}}\otimes\alpha_{U^{i}}+2g_{i\Bar{j}}dz^{i}\otimes d\Bar{z}^{j},
\end{equation}
where $g_{i\Bar{j}}$ are certain functions on $U^{i}$.

\vspace{2ex}

\begin{proposition}
    The local components of the torsion tensor $T^{{\nabla^*}}$ of the dual of an shs connection with respect to Darboux charts are provided by:

    \begin{enumerate}
        \item 
        $T^{\nabla^*}(\partial_{0},\partial_{0})=0,$\,\, $T^{\nabla^*}(\partial_{0},\partial_{j})=g^{\Bar{i}p}\partial_{0}g_{\Bar{i}j}\partial_{p}$,\, j=1,..,2d,

        \item $T^{\nabla^{*}}(\partial_i,\partial_k)=\partial_{k}\Theta_{i}^{p}\Theta_{p}^{s}\partial_{s}+\Theta_{i}^{p}\Gamma_{kp}^{m}\Theta_{m}^{l}\partial_l-\partial_i\Theta_ {k}^{p}\Theta^{r}_{p}\partial_r-\Theta_ {k}^{p}\Gamma_{ip}^{n}\Theta^{a}_n\partial_a$,\\\,\, i,k=1,...,2d.
        
    \end{enumerate}
    \label{pr:torsiondual}
\end{proposition}
\begin{proof}

\begin{equation}
    T^{\nabla^{*}}(\partial_{0},\partial_{0})=0,\,T^{\nabla^{*}}(\partial_{0},\partial_{j})=T^{*0}_{0j}\partial_{0}+T^{*l}_{0j}\partial_{l},\, T^{\nabla^{*}}(\partial_{i},\partial_{{j}})=T^{*0}_{ij}\partial_{0}+T^{*l}_{i{j}}\partial_{l}.
\end{equation}

Bearing in mind that $\nabla E=0$ is equivalent to $\nabla^{*}\alpha=0$, then
for any vector field $X, Y$, we get

$$0=(\nabla_{X}^{*}\alpha)Y=X.\alpha(Y)-\alpha(\nabla^{*}_{X}Y),$$
and 
$$0=(\nabla_{Y}^{*}\alpha)X=Y.\alpha(X)-\alpha(\nabla^{*}_{Y}X).$$
From the above analysis, we obtain:

$$X.\alpha(Y)-Y.\alpha(X)-\alpha([X,Y])=\alpha(\nabla^{*}_{X}Y)-\alpha(\nabla^{*}_{Y}X)-\alpha([X,Y]),$$
that is equivalent to:
\begin{equation}
    \label{eq:deriv}
    (d\alpha)(X,Y)=\alpha(T^{\nabla^*}(X,Y)).
\end{equation}

The stabilization condition yields: 

\begin{equation}
    \alpha(T^{*}(E,Y))=0.
    \label{eq:derivO}
\end{equation}

It follows from the above \cref{eq:deriv,eq:derivO} that:

\begin{equation}
   T^{*0}_{0j}=0, \, T^{*0}_{ij}=0,\, \forall i,j.
    \label{eq:torsionpre}
\end{equation}

 Hence, departing from \cref{eq:torsionpre}, we deduce that:

\begin{equation}
  T^{\nabla^{*}}(\partial_{0},\partial_{0})=0,\,\,T^{\nabla^{*}}(\partial_{0},\partial_{j})=T^{*l}_{0,j}\partial_{l}\,\,, T^{\nabla^{*}}(\partial_{i},\partial_{j})=T^{*l}_{ij}\partial_{l}.
\end{equation}

\begin{enumerate}
    \item Now from duality condition, we get:

\begin{equation}
    \partial_{0}g_{ij}=g_{jm}\Gamma^{m}_{0i}+g_{im}\Gamma^{*m}_{0j}
    \label{eq:usedua}
\end{equation}
\vspace{1ex}
 Replacing $j$ by $0$ and $0$ by $j$ in \cref{eq:usedua}, we obtain:
 
\begin{equation}
\partial_{j}g_{i0}=g_{0m}\Gamma^{m}_{ji}+g_{im}\Gamma^{*m}_{j0}.
\label{eq:usedualpermu}
\end{equation}
\vspace{1ex}
Resting upon \cref{eq:usedua,eq:usedualpermu}, we get:
\begin{equation}
  \partial_{0}g_{ij}= g_{im}(\Gamma^{*m}_{0j}-\Gamma^{*m}_{j0})+g_{jm}\Gamma^{m}_{0i}.
\end{equation}
\vspace{1ex}
Bearing in mind that $\Gamma^{m}_{0i}=\Gamma^{m}_{i0}=0$, then we have:

\begin{equation}
     g_{im}(\Gamma^{*m}_{0j}-\Gamma^{*m}_{j0})=\partial_{0}g_{ij}.
     \label{eq:principdual}
\end{equation}

With regard to \cref{eq:principdual}, we get:

\begin{equation}
     g_{i\Bar{p}}(\Gamma^{*\Bar{p}}_{0j}-\Gamma^{*\Bar{p}}_{j0})=\partial_{0}g_{ij}.
\end{equation}     
and therefore,
\begin{equation}
     T^{*\Bar{p}}_{0j}=g^{i\Bar{p}}\partial_{0}g_{ij}.
\end{equation}  

Accordingly, referring to \cref{eq:coherloca}, we obtain:

\begin{equation}
    T^{*\Bar{p}}_{0j}=0.
\end{equation}

Using again formula \cref{eq:principdual} and replacing $i$ by $\Bar{i}$, we get:

\begin{equation}
     g_{\Bar{i}m}(\Gamma^{*m}_{0j}-\Gamma^{*m}_{j0})=\partial_{0}g_{\Bar{i}j};
\end{equation}

and hence,

\begin{equation}
     T^{*p}_{0j}=g^{\Bar{i}p}\partial_{0}g_{\Bar{i}j};
\end{equation}

Consequently, we infer that

\begin{equation}
    T^{\nabla^{*}}(\partial_{0},\partial_{j})=g^{\Bar{i}p}\partial_{0}g_{\Bar{i}j}\partial_{p}.
\end{equation}
\item 
  Departing from a straightforward computation, we have 
  
    \begin{equation}
    \partial_{i}.\Omega(\partial_j,\Theta \partial_k)=\Omega(\nabla_{\partial _i}\partial_j,\Theta \partial_k)+\Omega(\partial_j,\nabla_{\partial_i}\Theta \partial_k), \quad\forall i,k=1,..,2d.
    \label{eq:parall}
    \end{equation}
    
By a simple observation, we get:

\begin{equation}
    \nabla_{\partial_i}\Theta \partial_k=\Theta(\nabla_{\partial_i}\partial_k-\Theta(\nabla_{\partial_i}\Theta)\partial_{k}), \quad\forall i,k=1,..,2d.
\end{equation}

Combining the above equality with \cref{eq:parall}, we get:

$$ \partial_i.\Omega(\partial_j,\Theta \partial_k)=\Omega(\nabla_{\partial_i}\partial_j,\Theta \partial_k)+\Omega(\partial_j,\Theta(\nabla_{\partial_{i}}\partial_{k}-\Theta(\nabla_{\partial_i}\Theta)\partial_k).\,$$

\vspace{1ex}
Eventually, substituting $\Omega(\partial_j,\Theta \partial_k)$ for $g(\partial_j,\partial_k)$, we get:
\begin{equation}
\partial_i.g(\partial_j,\partial_k)=g(\nabla_{\partial_i}\partial_j, \partial_k)+g(\partial_j,\nabla_{\partial_i}\partial_k-\Theta(\nabla_{\partial_i}\Theta)\partial_{k})\,.
\end{equation}
Thus,
\begin{equation}
\nabla^{*}_{\partial_i}\partial_k=\nabla_{\partial_i}\partial_k-\Theta(\nabla_{\partial_i}\Theta)\partial_{k}.
\label{eq:dualexp}
\end{equation}

It follows from the above that:
\begin{equation}
    T^{\nabla^{*}}(\partial_i,\partial_k)=\Theta((\nabla_{\partial_k}\Theta)\partial_{i}-(\nabla_{\partial_i}\Theta)\partial_{k}).
\end{equation}

Consequently, we obtain:

\begin{equation}
    T^{\nabla^{*}}(\partial_i,\partial_k)=\Theta\nabla_{\partial_k}\Theta\partial_{i}-\Theta\nabla_{\partial_i}\Theta\partial_{k}.
\end{equation}

By a direct computation, we get:

\begin{equation}
    T^{\nabla^{*}}(\partial_i,\partial_k)=\Theta\nabla_{\partial_k}\Theta_{i}^{p}\partial_{p}-\Theta\nabla_{\partial_i}\Theta_ {k}^{p}\partial_{p};
\end{equation}

\begin{equation}
=\Theta(\partial_{k}\Theta_{i}^{p}\partial_{p}+\Theta_{i}^{p}\nabla_{\partial_k}\partial_p)-\Theta(\partial_i\Theta_ {k}^{p}\partial_{p}+\Theta_ {k}^{p}\nabla_{\partial_i}\partial_{p});
\end{equation}

Finally, we obtain:
\begin{equation}
  T^{\nabla^{*}}(\partial_i,\partial_k)=\partial_{k}\Theta_{i}^{p}\Theta_{p}^{s}\partial_{s}+\Theta_{i}^{p}\Gamma_{kp}^{m}\Theta_{m}^{l}\partial_l-\partial_i\Theta_ {k}^{p}\Theta^{r}_{p}\partial_r-\Theta_ {k}^{p}\Gamma_{ip}^{n}\Theta^{a}_n\partial_a.
\end{equation}

\end{enumerate}
\end{proof}

\subsection{ Obstructions to the existence of statistical SMAT shs connections on shs manifolds}

We shall now shed light on the work of Lauritzen in the information geometry area.

\subsubsection{Statistical connections}


\citet{amari1987differential,amari2012differential} and \cite{chentsov1982statistical} introduced a family of $\alpha$-connections $\nabla^{\alpha}$, where $\alpha\in\mathbb{R}$ is symmetric but non-metric on the statistical model $ \mathcal{P}_{d}(\Xi)$, such that:

$$\nabla^{\alpha}g^{F}=\alpha T^{A-C},$$

with $T^{A-C}$ being often called the Amari-Chentsov tensor, in honour of both researchers who have worked intensively on this structure. It is a totally symmetric 3-tensor defined by the following formula:

 \begin{equation}
     g_{ij}^{F}=E_{p_\theta}(\partial_{i}\,log\, p_{\theta}(x)\partial_{j}\,log\, p_{\theta}(x))\,\,\forall i,j=1,...,d,
 \end{equation}
\begin{equation}
    T_{ijk}^{A-C}=E_{p_\theta}(\partial_{i}\,log\, p_{\theta}(x)\partial_{j}\,log\, p_{\theta}(x) \partial_{k}\,log\, p_{\theta}(x))\,\,\forall i,j,k=1,..,d.
\end{equation}
 \vspace{1ex}
$\alpha$-connections are defined by:
\begin{equation}
  \Gamma^{\alpha}_{ij:k}(\theta)=E_{p_{\theta}}(\frac{\partial^2}{\partial\theta_{i}\partial\theta_{j}}l_{\theta}\frac{\partial}{\partial_\theta}_{k}l_{\theta})+\frac{1-\alpha}{2}E_{p_{\theta}}(\frac{\partial}{\partial\theta_{i}}l_{\theta}\frac{\partial}{\partial\theta_{j}}l_{\theta}\frac{\partial}{\partial\theta_{k}}l_{\theta}) 
\end{equation}

For $\alpha=0$, we obtain the Levi-Civita connection of the Fisher metric defined as follows:

\begin{equation}
    \Gamma^{lc}_{ij:k}(\theta)=E_{p_{\theta}}(\frac{\partial^2}{\partial\theta_{i}\partial\theta_{j}}l_{\theta}\frac{\partial}{\partial_\theta}_{k}l_{\theta})+\frac{1}{2}E_{p_{\theta}}(\frac{\partial}{\partial\theta_{i}}l_{\theta}\frac{\partial}{\partial\theta_{j}}l_{\theta}\frac{\partial}{\partial\theta_{k}}l_{\theta}).
\end{equation}
\vspace{1ex}
Obviously, in a statistical model $ \mathcal{P}_{d}(\Xi)$, we have the following equation

\begin{equation}
   \partial_{\theta_k}.g^{F}(\partial_{\theta_i},\partial_{\theta_j})=g^{F}(\nabla^{\alpha}_{\partial_{\theta_k}}\partial_{\theta_i},\partial_{\theta_j})+g^{F}(\partial_{\theta_i},\nabla^{-\alpha}_{\partial_{\theta_k}}\partial_{\theta_j}),
\label{eq:dualityinsta} 
\end{equation}
where $g^{F}$ is the Fisher information matrix defined in \cref{eq:Fisher}. It is to be noted that Steffen L.\citet{lauritzen1987statistical} was motivated by the above construction of $\nabla^{\alpha}$, as well as the concept of Amari-Chentsov tensor and he came up with the definition of the notion statistical manifold.
 \vspace{1ex}
 \begin{defn}
     A statistical manifold refers to a triple $(M, g,\nabla,\nabla^*)$   satisfying the following conditions:
     \begin{enumerate}
         \item $\forall X,Y,Z\in \Gamma(TM)$, 
         \begin{equation}
X.g(Y,Z)=g(\nabla_{X}Y,Z)+g(Y,\nabla^{*}_{X}Z).
       \end{equation}
 
         \item  
     Both $\nabla$ and $\nabla^*$ are torsion free. 
     \end{enumerate}
    A statistical manifold  can equivalently be defined in terms of a Riemannian manifold $(M, g)$ with a 3-symmetric tensor $B$.
 \end{defn}
 \vspace{2ex}
\cite{lauritzen1987statistical} For a statistical manifold $(M,g,\nabla,\nabla^*)$, the dual connections $(\nabla,\nabla^*)$ are called statistical connections and the Levi-Civita
connection of the metric $g$ is characterized by:
\begin{equation}
2\nabla^{lc}=\nabla+\nabla^{*}.
\label{eq:moy}
\end{equation}
Alternatively, one may build up statistical connections $(\nabla^{D},\nabla^{*D})$ using smooth functions $D$ (three-times differentiable), which satisfy some conditions (cf. \cite{matumoto1993any}, \cite{eguchi1983second}, \cite{eguchi1992geometry}) as follows:
\begin{align*}
    &\Gamma^{D}_{ijk}(\theta)=-\partial_{\theta_1i}\partial_{\theta_1j} \partial_{\theta_2k}D(\theta_{1},\theta_{2})\restriction_{\theta_1=\theta_2=\theta}\\
    \text{and}\quad&\Gamma^{*D}_{ijk}(\theta)=-\partial_{\theta_2i}\partial_{\theta_2j} \partial_{\theta_1k}D(\theta_{1},\theta_{2})\restriction_{\theta_1=\theta_2=\theta},
\end{align*}
where $\Gamma^{D}_{ijk}(\theta)$, $\Gamma^{*D}_{ijk}(\theta)$
are the symbols of the dual connections $\nabla^{D}$ and $\nabla^{*D}$. In the literature, it is often indicated that the function $D$ is a divergence function (contrast function), or an entropy function. 

\subsubsection{Statistical connections in an shs manifold}

The suitably adapted SMAT connections $(\nabla,\nabla^*)$ to the shs structure $(\Omega,\alpha)$, will be called \textbf{the SMAT shs connections} and the set of all SMAT shs connections on $(M^{2d+1},\Omega,\alpha)$ is expressed as:


$$\mathcal{E}_{(\Omega,\alpha,g)}(M^{2d+1})=\left\{ (\nabla,\nabla^*) \Big{\lvert}
\begin{array}{ccc}
\nabla E=0,\,\nabla \Omega=0\\
\text{or}\\
\nabla^{*}\alpha=0,\,\nabla \Omega=0
\end{array}\right\}$$

We close this section with the investigation of  obstructions to the existence of statistical SMAT shs connections. We prove that in some shs manifolds, SMAT shs connections are not statistical connections.
\vspace{2ex}

\begin{proposition}  Let $(\Omega=d\alpha,\alpha)$ be a contact manifold. The SMAT shs connections are not statistical connections.
\label{prop:contact}
\end{proposition}  
\begin{proof}
    Assume that $(\nabla,\nabla^{*})$ are statistical smat shs connections on the contact manifold. For any $X$ and $Y$, we have: 
$$0=(\nabla_{X}^{*}\alpha)Y=X.\alpha(Y)-\alpha(\nabla^{*}_{X}Y),$$
and 
$$0=(\nabla_{Y}^{*}\alpha)X=Y.\alpha(X)-\alpha(\nabla^{*}_{Y}X).$$

By a simple manipulation, we get:

$$X.\alpha(Y)-Y.\alpha(X)-\alpha([X,Y])=\alpha(\nabla^{*}_{X}Y)-\alpha(\nabla^{*}_{Y}X)-\alpha([X,Y]),$$

which is equivalent to:

\begin{equation}
    \label{eq:closed}
    (d\alpha)(X,Y)=\alpha(T^{\nabla^*}(X,Y)).
\end{equation}

By assumption, $\nabla^{*}$ is a symmetric connection. Accordingly, we have:

$$d\alpha=0.$$

This is a contradiction and the proof of the proposition is hence complete.
\end{proof}
\begin{proposition}
    Let $(M^{2d+1},\Omega,\alpha)$ be a compact shs manifold whose fundamental group is finite. The SMAT shs connections are not statistical connections.
    \label{eq:128}
\end{proposition}
\begin{proof}
Consider a statistical shs connection $(\nabla,\nabla^{*})$ on $M^{2d+1}$. It follows from \cref{eq:closed} that $\alpha$ is closed. By assumption, the fundamental group of $M^{2d+1}$ is finite. We deduce that $\alpha$ is exact. From this perspective, it should vanish
at some point $p\in M^{2d+1}$, owing to the compactness of $M^{2d+1}$. However, this is a contradiction because $\alpha$ is nowhere vanishing. This contradiction ends the proof of the proposition.
\end{proof}

\begin{proposition}
     Let $(M^{2d+1},\Omega,\alpha)$ be a compact shs manifold with a compatible almost contact metric structure $(\Theta, E,\alpha, g)$ such that the $Ric^{g}$ of $g$ is negative definite. SMAT shs connections are not therefore statistical connections.
     \label{prop:fundamenta}
\end{proposition}
 \begin{proof}
     Resting on a direct computation, we have

\begin{equation}
    (L_{E}g)(\partial_{i},\partial_{\Bar{j}})=\partial_{0}g_{i\Bar{j}}.
\end{equation}

Consider a statistical shs connection $(\nabla,\nabla^{*})$ on $M^{2d+1}$. Based on the proposition 
    \ref{pr:torsiondual}, we get
\begin{equation}
    (L_{E}g)(\partial_{i},\partial_{\Bar{j}})=0.
\end{equation}
Therefore, the Reeb vector field $E$ is a Killing vector field. By assumption, $Ric^{g}<0$. Then according to \cite{bochner1946vector}, we deduce that $E$ vanishes identically. This is a contradiction and the result follows at once.
 \end{proof}

\begin{proposition}
   Let $(M^{2d+1},\Omega,\alpha)$ be an shs manifold with its compatible almost contact metric structure $(\Theta, E,\alpha, g)$. If the Ricci tensor $Ric^{g}$ of $g$ is positive definite, the SMAT shs connections are not statistical connections.
   \label{prop:fundabis}
\end{proposition}
\begin{proof}
    Assume that the set of SMAT shs connections contains statistical connections. Departing from \cref{eq:closed}, we infer that $\alpha$ is closed. Then, $(M^{2d+1},\Omega,\alpha,g)$ is a cosymplectic(almost co-Kahler) manifold. It follows from \cite{olszak1981almost}(Theorem 4.2) that $$Ric^{g}(E,E)=-\lVert \nabla^{lc}E^{2}  \rVert$$
    
    This is a contradiction and the conclusion then follows from it.
\end{proof}

 To sum up, we proved in this section that any shs manifold admits infinite SMAT connections adapted to an shs
structure $(\Omega,\alpha)$. We equally noted that for some shs manifolds, there is no statistical shs SMAT connection. The above
result indicates that the shs manifolds which contains statistical SMAT shs connections are very special types of manifolds. Hence, a natural question arises\vspace{0.2cm}

\textbf{Question 2: Which shs manifold contains a statitical SMAT shs connection ?}\vspace{0.2cm}

For reasons of simplicity, the statistical shs SMAT connections will be called \textbf{ the statistical shs connections}. In what follows, we attempt to provide an answer to this question. Our answer offers a new approach to the integrability for almost co-Khaler structures on almost contact manifolds.

\section{Answer to question 2 and statistical approach to the integrability of almost co-Kähler structure.
}
\label{sec:cokahlerstructure}
 \citet{morimoto1963normal} built up a natural almost complex structure $\Theta$ on the product of two almost contact manifolds $(M_{j},\Theta_{j},E_{j},\alpha_{j})_{j\in\{1,2\}}$ and revealed that $\Theta$ is integral if and only if both of the product factors are normal almost contact manifolds. A direct consequence of this fact lies in that the product of two odd-dimensional spheres $S^{2l+1}\times S^{2k+1}$ is complex. A natural question to ask is to know under what conditions on $M_{i}$, the product is a K\"ahler manifold. It is enticing to speculate that every factor manifold $M_{i}$ should be Sasakian as a solution. Nevertheless, it is not possible because it would imply that the result of two Sasakian spheres would be a K\"ahler manifold. This contradicts the result of Calabi and Eckmann holding that the product of odd dimensional spheres is complex and non Kähler \citet{calabi1953class}. The correct geometric structure on factor manifolds $M_i$ should be a co-Kähler structure, as corroborated by \citet{goldberg1968totally} and later by  \citet{capursi1984some}. 

\begin{defn}
    A co-Kähler manifold $(M^{2d+1},\Omega,\alpha)$ is a cosymplectic manifold admitting a compatible almost contact metric structure $(\Theta, \alpha, E, g)$
such that
    $$N^{\Theta}=0.$$ Recall that $N^{\Theta}$ is defined by 
    $$N^{\Theta}=[\Theta,\Theta]-\Theta[\Theta,.]-\Theta[.,\Theta.]+\Theta^{2}[.,.].$$
\end{defn}
In the co-Kähler manifold, the compatible (co-Hermitain) metric $g$ will be called a co-Kähler metric in this paper. It is very significant to note that Blair \cite{blair2010riemannian} used the term "cosymplectic," whereas Ougie \cite{ogiue1967cocomplex}, \cite{ogiue1968g} used the term "cocomplex''. This first name has been widely adopted in the literature and has been used in numerous papers \cite{watson1983new}, \cite{goldberg1969integrability}, \cite{olszak1981almost}, \cite{de1994compact}, \cite{chinea1993topology}, \cite{chinea1997spectral}, \cite{fujimoto1974cosymplectic},\cite{marrerocompact}. However, it may entail confusion because Libermann \cite{libermann1959automorphismes} used the same term to refer to the almost cosymplectic structure $(\Omega,\alpha)$ such that $\Omega$ and $\alpha$ are closed. We will use the term "co-Kähler" throughout this paper. This name was invested by  Bazzoni \cite{bazzoni2014structure}, 
Janssens, L.Vanhecke \cite{janssens1981almost}, etc. This terminology has more recently been adopted in the article of Hongjuin Li \cite{li2008topology} who confirmed that any closed co-Kähler manifold is a Kähler mapping torus. According to Li’s work, we infer that the co-Kähler manifolds are
really odd dimensional analogues of Kähler manifolds. A natural example of co-Kähler manifold is provided by a Kähler manifold product with the circle (or with $\mathbb{R}$). Some pertinent examples of co-Kähler manifolds are set foward by J. C. Marrero; E. Padrón-Fernández(\cite{marrero1998new}) and Blair\cite{blair2010riemannian}.

\subsection{Statistical characterization of a co-Kähler structure}
In 1969, S. I. Goldberg asserted that if the curvature operator of an almost Kähler manifold commutes with the almost complex structure( ie $ R^{\nabla^{lc}}\circ\Theta=\Theta\circ R^{\nabla^{lc}}$), then the manifold is a Kähler manifold. He conjectured in \cite{goldberg1969integrability1} that "Every compact Einstein almost Kählerian manifold is a Kähler manifold". Under certain curvature conditions, some progress has been implemented (see \cite{sekigawa1985some}, \cite{goldberg1998curvature}, \cite{sawaki1972almost}, \cite{sawaki1972sufficient}, \cite{sawaki1964almost}, \cite{gray1976curvature}, \cite{olszak1981almost}). The odd dimensional counterparts of Kählerian manifolds are co-Kähler manifolds. From this perspective, one can ask for
a Goldberg-like conjecture for these manifolds. In \cite{goldberg1969integrability}, the author identified a curvature condition for the integrability  of the co-symplectic manifolds (almost co-Kähler manifolds). In \cite{montano2010einstein}, the author introduced the analogues of the Goldberg conjecture in the case of cosymplectic manifolds (almost co-Kähler manifolds). The author's theorem states that,
"Every compact Einstein almost co-Kähler manifold, such that the
Reeb vector field is Killing, is a co-Kähler manifold".\vspace{0.2cm}

Our central objective in this section is to provide a new characterization of co-Kähler manifolds using information geometric ingredients.
\vspace{1ex}
\begin{defn}
    \textbf{A statistical shs manifold} is an shs manifold $(M^{2d+1},\Omega,\alpha)$ such that the set of SMAT shs connections $\mathcal{E}_{(\Omega,\alpha,g)}(M^{2d})$ contains a statistical one. 
\end{defn}

\vspace{2ex}

\begin{thm}
    An shs manifold $(M^{2d+1},\Omega,\alpha)$ is a co-Kähler manifold, if and only if, it is a statistical shs manifold. \label{thm:principal}
\end{thm}
\begin{proof}
Consider an shs manifold $(M^{2d+1},\Omega,\alpha)$ and select a compatible almost co-Kähler structure $(\Theta,g)$ defined in \cref{eq:72,eq:73,eq:74,eq:75,eq:76}. Suppose that the set of SMAT shs connections contains a statistical connection $(\nabla,\nabla^{*})$. It follows from \cref{eq:deriv}, that the stabilization 1-form $\alpha$ is closed. Therefore, the shs manifold $(M^{2d+1},\Omega,\alpha)$ is a co-symplectic manifold(almost co-Khaler manifold). Now, choose a local orthonormal $\Theta$-basis $\{E, X_{i}, X_{\Bar{i}}=\Theta X_{i}\}$ at each $p\in M^{2d+1}$. The components of the Hamiltonian form $\Omega$ in the $\Theta$-basis are provided by:
\begin{align}
          \Omega(E,X_{j})&=0 \\
          \Omega(E,X_{\Bar{j}})&=0 \\
        \Omega(X_{i},X_{j}) &=0\\
          \Omega(X_{i},X_{\Bar{j}})  &=g(X_i,X_j)=\delta_{ij}\\
          \Omega(X_{\Bar{i}},X_{\Bar{j}}) &=0.
 \end{align}
We shall now prove that $\nabla^{*}\Omega=0$. It is sufficient to check the identity in the following three particular cases: 
\begin{enumerate}
    \item By a direct computation, we obtain:   
\begin{equation}
(\nabla^{*}_{E}\Omega)(X_i,\Theta X_j)=E.\Omega(X_{i},\Theta X_{j})-\Omega(\nabla^{*}_{E}X_{i},\Theta X_{j})-\Omega(X_{i},\nabla^{*}_{E}\Theta X_{j}).
\end{equation}

 As $\nabla\Omega=0$, we get $\nabla^{*}_{E}\Theta X_{j}=\Theta\nabla_{E} X_{j}$. Shs connection $\nabla$ is a symmetric connection. Then,
 
 $$\nabla^{*}_{E}\Theta X_{j}=0\,\,, j=1,2,..,2d$$
 Accordingly, we easily obtain:

\begin{equation}
(\nabla^{*}_{E}\Omega)(X_i,\Theta X_j)=-\Omega(\nabla^{*}_{E}X_{i},\Theta X_{j}).
\end{equation}

Now, based on duality condition, we get:
 
$$E.\Omega(X_{i},\Theta X_{j})=E.g(X_{i},X_{j})=g(\nabla_{E}X_i,X_j)+g(X_i,\nabla^{*}_{E}X_j).$$

By a simple computation, we obtain:

\begin{equation}
    g(X_i,\nabla^{*}_{E}X_j)=0, \quad \text{for} \, i,j:1,...,2d.
    \label{eq:92}
\end{equation}

Departing from \cref{eq:92}, we get: 

\begin{equation}
    \Omega(\Theta X_i,\nabla^{*}_{E}X_j)=0, \quad \text{for} \, i,j:1,...,2d.
    \label{eq:93}
\end{equation}

It follows from \cref{eq:93} that:
\begin{equation}
(\nabla^{*}_{E}\Omega)(X_i,\Theta X_j)=0\quad \text{for} \, i,j:1,...,2d.
\end{equation}

\item Using a direct computation, we get:
$$(\nabla^{*}_{X_i}\Omega)(E,\Theta X_j)=X_i\Omega(E,\Theta X_j)-\Omega(\nabla^{*}_{X_i}E, \Theta X_j)-\Omega(E,\nabla^{*}_{X_i}\Theta X_j).$$

Since $i_{E}\Omega=0$, we obtain without difficulty that:
 
\begin{equation}
(\nabla^{*}_{X_i}\Omega)(E,\Theta X_j)=\Omega(\Theta X_j,\nabla^{*}_{X_i}E).
\end{equation}

With respect to \cref{eq:93} as well as the symmetry of $\nabla^{*}$, we get

$$\Omega(\Theta X_j,\nabla^{*}_{X_i}E)= \Omega(\Theta X_j,\nabla^{*}_{E}X_i)=0.$$

We accordingly deduce that:
$$(\nabla^{*}_{X_i}\Omega)(E,\Theta X_j)=0,\, \text{for} \, i,j:1,...,2d.$$

\item Grounded on a direct computation, we have for $i,j,k:1,...,2d,$

\begin{equation}
    \begin{split}
    (\nabla^{*}_{X_i}\Omega)(X_j,\Theta X_k)&=X_i.\Omega(X_j,\Theta X_k)-g(\nabla^{*}_{X_i}X_j,X_k)\\&-g(X_j,\nabla^{*}_{X_i}X_k)-\Omega(X_j,(\nabla^{*}_{X_i}\Theta)X_k).
    \end{split}
\end{equation}

Using the duality condition and \cref{eq:dualexp}, it follows that:

$$
(\nabla^{*}_{X_i}\Omega)(X_j,\Theta X_k) =- g((\nabla_{X_{i}}\Theta) X_j,\Theta X_k)-\Omega(X_j,(\nabla^{*}_{X_i}\Theta)X_k)
$$

Again, through a direct computation, we have
$$
 (\nabla^{*}_{X_i}\Omega)(X_j,\Theta X_k) =g(\Theta\nabla_{X_i}X_{j},\Theta X_k)-g(\nabla_{X_{i}}(\Theta X_{j}),\Theta X_k)-\Omega(X_j,(\nabla^{*}_{X_i}\Theta)X_k)$$
Using the duality once more, we obtain:

\begin{equation}
\begin{split}
    (\nabla^{*}_{X_i}\Omega)(X_j,\Theta X_k) = &g(\nabla_{X_i}X_j,X_k)+g(X_j,\nabla^{*}_{X_i}X_k)-X_{i}.g(X_j,X_k) \\ & +g(\Theta X_j,(\nabla^{*}_{X_i}\Theta)X_k)-\Omega(X_j,(\nabla^{*}_{X_i}\Theta)X_k).
\end{split}
\end{equation}
A simple observation yields:

\begin{equation}
\begin{split}
    (\nabla^{*}_{X_i}\Omega)(X_j,\Theta X_k) = &g(\nabla_{X_i}X_j,X_k)+g(X_j,\nabla^{*}_{X_i}X_k)-X_{i}.g(X_j,X_k) \\ & +\Omega(X_j,(\nabla^{*}_{X_i}\Theta)X_k)-\Omega(X_j,(\nabla^{*}_{X_i}\Theta)X_k).
\end{split}
\end{equation}

Referring to the duality condition once again, we deduce that:

\begin{equation}
(\nabla^{*}_{X_i}\Omega)(X_j,\Theta X_k)=0,\,\text{for}\,i,j,k=1,...,2d.
\end{equation}

\end{enumerate}
 
Finally, resting upon three previous relations, we then have:

\begin{equation}
\nabla^{*}\Omega=0.
\label{eq:nablaetoil}
\end{equation}
Bearing in mind \cref{eq:moy}, we therefore get:

\begin{equation}
    \nabla^{lc}\Omega=0,
    \label{eq:102}
\end{equation}
or equivalently,
\begin{equation}
    \nabla^{lc}\Theta=0.
\end{equation}

We deduce with reference to \cite{blair1966theory,blair2010riemannian} that the shs manifold $(M^{2d+1},\Omega,\alpha)$ is a co-Kähler manifold.\\
Conversely, it is well known departing from \cite{blair2010riemannian} that in any co-Kähler manifold, the Levi-Civita connection of the co-Hermitian metric satisfies the following identity:
\begin{equation}
    \nabla^{lc} E=0,\ \nabla^{lc}\Omega=0;
    \label{eq:paraco}
\end{equation}
equivalently,
\begin{equation}
    \nabla^{lc} \alpha=0,\ \nabla^{lc}\Omega=0.
    \label{eq:paracoo}
\end{equation}
Remember that the Levi-Civita connection of any metric $g$ is an autodual connection, i.e , $\nabla^{lc}=\nabla^{lc*}$. As a matter of fact, from \cref{eq:paraco,eq:paracoo}, the Levi-Civita of the co-K\"ahler metric $g$ is a statistical shs connection (ie, $\nabla^{lc}\in \mathcal{S}_{(\Omega,\alpha,g)}(M^{2d+1})$). This therefore completes the proof. 
\end{proof}
The following corollary is an easy consequence of the above theorem.
\vspace{2ex}
\begin{corollary}
    A compact shs manifold $(M^{2d+1},\Omega,\alpha)$ is a  statistical shs manifold if and only it is a K\"ahler mapping torus.
\end{corollary}
\begin{proof} 
The previous Theorem \ref{thm:principal} as well as the Hongjun Li theorem \cite{li2008topology} prove our assertion.
\end{proof}

\begin{remark}
     As a consequence of the work of 
\citet{chinea1993topology}, we deduce that the Betti numbers of a compact statistical shs manifold satisfy the following monotony property:

\[
   1 \leq b_{i}(M^{2d+1})\leq ...\leq b_{d}(M^{2d+1})\quad i\leq d, \]
   
 \[
        b_{d+1}(M^{2d+1})\geq ...\geq b_{2d}(M^{2d+1})\quad i\geq d+1, \]

\[
        b_{d+1}(M^{2d+1})=b_{d}(M^{2d+1}).\]

\end{remark}

According to Theorem \ref{thm:principal}, we deduce that the only shs manifolds that contain statistical shs connections are co-Kähler manifolds. In other words, a co-Kähler manifold is the only shs manifold which contains a statistical shs connection, which stands for an answer to the question raised in the above section.\vspace{0.2cm}

\begin{com}
   In \cite{acakpo2022stable}, the author attempted to combine cosymplectic and contact structures. To achieve his goals, he invested the basic cohomology of the foliation $\ell_{\Omega}$. The author stated that, if a closed shs manifold $(M^{2d+1},\Omega,\alpha)$ has the dimension of the second basic cohomology group of the $\ell_{\Omega}$ equal 1, then $M^{2d+1}$ is a co-symplectic manifold or contact manifold. The author's assumption is very strong. In general,  the basic cohomology of a foliation is often infinite dimensional. Now, with our statistical approach, we obtain $\ell_{\Omega}$ which is a Killing foliation, and therefore a Riemaniann foliation. According to the work of \cite{kacimi1985cohomologie}, we deduce that for a compact statistical shs manifold $(M^{2d+1},\Omega,\alpha)$, the basic cohomology of $\ell_{\Omega}$ is always finite-dimensional. Furthermore $M^{2d+1}$ is a co-K\"ahler manifold.
\end{com}

 \subsection{Statistical properties on a co-Kähler manifold}
 
\begin{lem}
    Let $(\nabla,\nabla^{*})$ be a statistical shs connection on a co-Kähler manifold $(M^{2d+1},\Omega,\alpha)$.
    The dual connection $\nabla^{*}$ is also an shs connection.
    \label{lem:dual ex}
\end{lem}
\begin{proof}

 Resting on \cref{eq:moy}, any statistical shs connection $(\nabla,\nabla^{*})$ satisfies the condition:
   \begin{equation}
       \nabla^{*}=2\nabla^{lc}-\nabla.
       \label{eq:recrimoy}
   \end{equation}

Hence, according to \cref{eq:recrimoy,eq:paraco,eq:50}, we get:
\begin{equation}
    \nabla^{*}E=0,\, \nabla^{*}\Omega=0.
    \label{eq:dualshs}
\end{equation}
It follows from \cref{eq:dualshs}, that the dual connection $\nabla^{*}$ is also an shs connection. 
\end{proof}

\begin{corollary}
    A statistical shs connection  $(\nabla,\nabla^{*})$ on a co-Kähler manifold $(M^{2d+1},\Omega,\alpha)$ satisfies the following identities:
 \begin{equation}
     \nabla\alpha=0,\,\, \nabla^{*}\alpha=0.
     \label{eq:parapaa}
 \end{equation}
\end{corollary}
\begin{proof}
    Let's use the duality condition. We therefore deduce that:
    
    \begin{equation}
(\nabla_{X}\alpha)Y=g(Y,\nabla^{*}_{X}E),
\end{equation}
for any vector fields $X,Y$ on $M^{2d+1}$. Now grounded on the above lemma, we deduce that:
    \begin{equation}
        \nabla\alpha=0.
    \end{equation}
    
Remember that the dual connection $\nabla^{*}$ of an shs connection satisfies:

$$\nabla^{*}\alpha=0.$$

This proves our assertion.
\end{proof}

\vspace{2ex}

\subsubsection{Doubly autoparallel submanifolds on shs manifold}

Notably, the concept of doubly autoparallelism is often present but crucial to multiple information geometry applications (see\cite{ohara1999information}, \cite{ohara2009information},\cite{nagaoka2017information}). In what follows, we prove in a statistical shs manifold $(\nabla,\nabla^{*})$ that the leaves of both foliations $\ell_{\Omega}$ and $\ell_{\alpha}$ are $\alpha$-totally geodesic, and therefore $\alpha$-autoparalell(doubly autoparall). For a symmetric connection, the autoparallity and totally-geodesibillility are equivalent notions(see\cite{kobayashi1969foundations}).
Let's introduce a family of $\alpha$-connections by:

\begin{equation}
    \nabla^{\alpha}=\frac{1+\alpha}{2}\nabla^{*}+\frac{1-\alpha}{2}\nabla,\,\, \alpha\in \mathbb{R},
    \label{eq:124}
\end{equation}

where $\nabla,\nabla^{*}$ are statistical shs connections. Referring to \cref{lem:dual ex} and \cref{eq:parapaa}, one can easily infer that: 

\begin{equation} \nabla^{\alpha}E=0,\,\nabla^{\alpha}\Omega=0,\,\,\nabla^{\alpha}\alpha=0\,\,,\text{for any}\, \alpha\in \mathbb{R}.
\label{eq:126}
\end{equation}

Hence, $(\nabla^{\alpha},\nabla^{-\alpha})_{a\in\mathbb{R}}$ are also statistical shs connections.\vspace{0.2cm}

By a simple calculation, the readers can prove the following useful statements:

\begin{enumerate}
    \item 
 $\nabla^{-\alpha}\Theta=\Theta\nabla^{\alpha},\,\text{for any}\, \alpha\in \mathbb{R},$
\item 
$ R^{\nabla^{\alpha}}(.,.)E=0,\,R^{\nabla^{\alpha}}(.,.)\Omega=0,\,R^{\nabla^{\alpha}}(.,.)\alpha=0\,,\text{for any}\, \alpha\in \mathbb{R},$
\item The Ricci $Ric^{\nabla^{\alpha}}$ of the $\alpha$-connections are symmetric.

\item  The manifold $(M^{2d+1},g,\mu,\nabla^{\alpha})$, where $\mu=\alpha\wedge\Omega^{d}$, is an equi-affine statistical manifold. For general references on equi-affine statistical manifolds, we refer to \cite{nomizu1994affine}.
\end{enumerate}
\vspace{0.2cm}

\begin{proposition}
    Let $(M^{2d+1},\Omega,\alpha,g)$ be a co-Kähler manifold. The following statements are satisfied:
    
    \begin{enumerate}
    \item The leaves of $\ell_{\Omega}$ are totally geodesic w.r.t. to $\nabla^{\alpha}$, for all $\alpha$’s,
    \item The leaves of $\ell_{\alpha}$ are totally geodesic 
 w.r.t. to $\nabla^{\alpha}$, for all $\alpha$’s.
    \end{enumerate}
    \label{pro:geode}
\end{proposition}

\begin{proof}
 The verification of both statements is established through
 
    \begin{equation}
   \nabla^{\alpha} i_{X}\Omega=i_{\nabla^{\alpha} X}\Omega,\,\forall X \in \ell_{\Omega},
   \end{equation}
   \vspace{-4pt}
   \begin{equation}
       \nabla^{\alpha} i_{Y}\alpha=i_{\nabla^{\alpha} Y}\alpha,\,\forall Y \in \ell_{\alpha}.
   \end{equation}
\end{proof}

\begin{corollary}
    Let $(M^{2d+1},\Omega,\alpha,g)$ be a co-Kähler manifold. The leaves of the foliation $\ell_{\Omega}$ and $\ell_{\alpha}$ are statistical submanifolds.
    \label{cor:stat}
\end{corollary}
\begin{proof}
    The claim follows from the above proposition.
\end{proof}

\begin{figure}[H]
    \centering
\begin{center}
    \begin{tikzpicture}[xscale = 2]
    \fill[gray!20] (0,0) -- (2,0) -- (3,2) -- (1,2) -- cycle;
    \draw (0,0) -- (2,0) -- (3,2) node[pos = 0.4, below right]{$\ell_\alpha = \text{ker }\alpha$} node[pos = 0.7, below right]{$\nabla^\alpha$-totally geodesic} -- (1,2) -- cycle;
    \draw[thick] (1.5,-1) -- ++(0,1);
    \draw[dashed, thick] (1.5,0) -- ++(0,1);
    \draw[->, thick] (1.5,1) -- ++(0,2) node[pos = 0.9, left=5pt]{$\ell_\Omega = <E>$}node[pos = 0.9, right=5pt]{$\nabla^\alpha$-totally geodesic};
\end{tikzpicture}
 \end{center}
\caption{Statistical leaves on a co-Kähler manifold.}
    \label{fig:staconnections}
\end{figure}
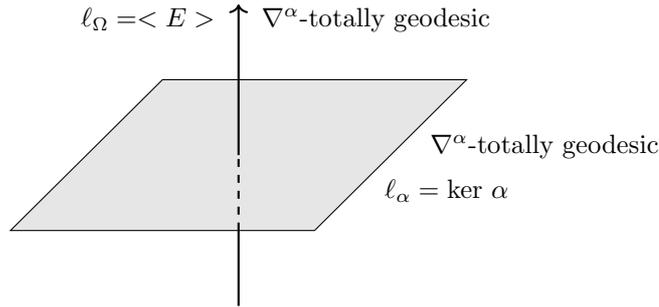

\subsection{Description of the set of statistical shs connections on a co-Kähler manifold}

Referring to \cref{eq:126}, the set of statistical shs connection on a co-Kähler manifold is expressed as:

$$\mathcal{S}_{(g,\Omega,\alpha)}(M^{2d+1})=\left\{ (\nabla,\nabla^*) \Big{\lvert}
\begin{array}{ccc}
\nabla E=0,\,\nabla \Omega=0,\,\nabla \alpha=0,\\
\text{or}\\
\nabla^{*}E=0,\,\nabla^{*}\Omega=0,\,\nabla^{*}\alpha=0
\end{array}\right\}$$

We close this section with a description of the set $\mathcal{S}_{(g,\Omega,\alpha)}(M^{2d+1})$ .

\vspace{2ex}

\begin{proposition}
    The space of statistical shs connections $\mathcal{S}_{(g,\Omega,\alpha)}(M^{2d+1})$ is an affine space modelled  on  the vector  space $(S^{2}_{E}(M^{2d+1})\oplus S^{3}_{E}(M^{2d+1}))\times (S^{2}_{E}(M^{2d+1})\oplus S^{3}_{E}(M^{2d+1})).$
\end{proposition}

\begin{proof}
    Let $(\nabla,\nabla^{*})$ be statistical shs connections on $(M^{2d+1},\Omega,\alpha)$. As a matter of fact, statistical shs connections $(\nabla,\nabla^{*})$ satisfy the following identity:

$$\nabla E=0,\, \nabla\Omega=0, \,\nabla\alpha=0;$$
equivalently,
$$\nabla^{*}E=0,\, \nabla^{*}\Omega=0,\,\nabla^{*}\alpha=0.$$

The statistical shs connections $(\nabla,\nabla^{*})$ satisfy the same conditions. For this reason, it is sufficient to work on one of them. As in Proposition \ref{eq:112}, any connection $\widetilde{\nabla}$ can be indicated in terms of:
$$\widetilde{\nabla}=\nabla+A,\quad\text{where}\quad A\in T^{2}_{1}M^{2d+1}.$$

By a direct computation, we obtain the following equation:

\begin{equation}
    g(\widetilde{\nabla}^{*}_{X}Y,Z)=g(\nabla^{*}_{X}Y,Z)-g(Y,A(X,Z)).
\end{equation}

 Now, the pair of connections $(\widetilde{\nabla},\widetilde{\nabla}^{*})$ is a statistical shs connection if and only if the following identities are satisfied:
\[
\begin{cases}
    &  S(X,Y,Z)=\Omega(A(X,Y),Z)\in \Gamma^{\infty}(S^{3}TM^{2d+1}),\\
    
    & Q(X,Y)=\alpha(A(X,Y))\in \Gamma^{\infty}(S^{2}TM^{2d+1}),\\
    
    & A(E.,.)=0.\\
\end{cases}
\label{eq:34}
\]
From the above, we deduce that the set $\mathcal{S}_{(g,\Omega,\alpha)}(M^{2d+1})$ is an affine space modelled on $(S^{2}_{E}(M^{2d+1})\oplus S^{3}_{E}(M^{2d+1}))\times (S^{2}_{E}(M^{2d+1})\oplus S^{3}_{E}(M^{2d+1})).$
\end{proof}

\section{ K\"ahler and co-K\"ahler metrics are always Fisher information metrics}
\label{se:Kal and cok}
As \cite{li2008topology} Li pointed out, co-Kähler manifolds are analogues of Kähler manifolds. From this perspective, in the literature there are analogous results in both cases. We recall some of these results. 

\begin{table}[h!]
    \centering
    \begin{tabular}{|m{15.5em}|m{18.5em}|}
         \hline
        K\"ahler structure & co-K\"ahler structure\\
        \hline
        \cite{goldberg1998curvature} ,\cite{deligne1975real} Betti numbers
    $b_{2k}\ne 0$; $b_{2k-1}=2s$; $b_{i-2}(M^{2d})\leq b_{i}(M^{2d})$ & \cite{chinea1993topology} $b_{i}(M^{2d+1})\ne 0$; $b_{1}(M^{2d+1})=2s+1$; $b_{i-2}\leq b_{i} \,,i\leq d+1$ ; 
   $1 \leq b_{i}(M^{2d+1})\leq ...\leq b_{d}(M^{2d+1})\quad i\leq d$,
   $$ b_{d+1}(M^{2d+1})\geq\ldots\geq b_{2d}(M^{2d+1})\;i\geq d+1,$$ 
        $$b_{d+1}(M^{2d+1})=b_{d}(M^{2d+1}).$$\\ 
        \hline
       \cite{deligne1975real}The Kähler manifold satisfies the strong Lefchetz property.  & \cite{chinea1993topology} The co-Kähler manifold satisfies the strong Lefschetz theorem. \\   
        \hline
 \cite{goldberg1969integrability1}\, Goldberg integrability condition for almost Kähler manifold & \cite{goldberg1969integrability}\, Goldberg integrability condition for almost co-Kähler manifold \\\hline       \cite{hasegawa2006note}\,Hasegawa theorem on solvmanifold Kähler manifold & \cite{fino2011some}\,Anna Fino and Luigi Vezzoni results on solvmanifold co-Kähler structure. \\
              \hline
       \cite{sullivan1973differential} Formality in sense of Sullivan of K\"ahler manifold & \cite{chinea1993topology}\,Formality in sense of Sullivan of co-Kähler manifold.\\
\hline
    \end{tabular}
    \caption{co-Kähler and Kähler correspondance results.}
    \label{tab:my_label}
\end{table}

   


We are now going to reveal that our statistical characterization of co-K\"ahler manifolds remains true with the K\"ahler manifold as well.

A symplectic connection \cite{tondeur1962affine}, \cite{lemlein1957spaces}, \cite{vaisman1985symplectic} on a symplectic manifold $(M^{2d},\Omega)$ is a torsion-free
linear connection $\nabla$ on $M^{2d}$, for which the symplectic 2-form $\Omega$ is parallel. Any symplectic manifold admits a symplectic connection (see \citet{bieliavsky2006symplectic},  \citet{vaisman1985symplectic}, \citet{lichnerowicz1982deformations}. According to Darboux theorem, on any Darboux charts, one can construct a symplectic connection $\nabla^{i}$  to have zero Christoffel symbols. The desired symplectic connection on $(M^{2d},\omega)$ is now obtained through gluing the local connections $\nabla^{i}$
with the help of the partition of unity.  \citet{bieliavsky2006symplectic} reported a general construction of a symplectic connection. They proved that the space of symplectic connections on $(M^{2d},\omega)$ is an affine space modeled on the space of symmetric covariant 3-tensor fields on $M^{2d}.$ \citet{vaisman1985symplectic}) set forward a curvature classification of symplectic connections.

 Let  $(g,\Theta)$ be an almost Hermitian structure on a symplectic manifold $(M^{2d},\omega)$. Choose a symplectic connection $\nabla$. It follows from a straightforward computation that:
\begin{equation}
    (\nabla_{Z}\omega)(X,Y) = g(\nabla^{*}_{X}\Theta Z-\Theta\nabla_{X}Z,Y),
    \label{eq:paaa}
\end{equation}
where $\nabla^{*}$ is the dual connection of $\nabla$ with respect to the Hermitian metric $g$. As a direct 
consequence of \cref{eq:paaa}, we obtain:
\begin{equation}
\nabla=\Theta^{-1}\nabla^{*}\Theta.
\label{eq:gaug}
\end{equation}
With regard to \cref{eq:gaug}, \cref{eq:dualexp} one can easily prove that the following identities hold for any symplectic connection $\nabla$ on $(M^{2b},\omega)$
\vspace{-6pt}
\begin{align}
        \nabla^{*}_{X}Y & =\nabla_{X}Y-\Theta(\nabla_{X} \Theta)Y,\\ 
  \nabla_{X}Y & =\nabla^{*}_{X}Y+\Theta^{-1}(\nabla^{*}_{X}\Theta)Y, \\
        R^{\nabla}(Y,X)Z & = \Theta^{-1}R^{\nabla^{*}}(X,Y)\Theta Z,\\
          T^{\nabla^{*}}(X,Y) &=\Theta^{-1}((\nabla^{*}_{Y}\Theta)X-(\nabla^{*}_{X}\Theta)Y),\\
          T^{\nabla^{*}}(X,Y)&=\Theta((\nabla_{Y}\Theta)X-(\nabla_{X}\Theta)Y),
\end{align}

 for any vector fields $X, Y, Z$ on $M^{2b}.$

\citet{hicks1965linear} introduced the notion of linear perturbations of connection $\nabla$ by:
$$\nabla\rightarrow\nabla^{\Theta}=\Theta^{-1}\nabla\Theta.$$
\vspace{-2pt}
Note that the Hicks definition of linear perturbation coincides exactly with the action of the gauge group of the affine space connections. In the literature, the linear perturbation $\nabla$ is also called the $\Theta$-gauge transformation of a connection $\nabla$ (see \cite{fei2017interaction}) or the gauge equation (see \cite{nguiffo2016foliations}). Note that in the case of two isometric metrics, Levi-Civita connections of both metrics are related by a linear perturbation. In 2009, \citet{schwenk2009codazzi} extended the
works of Hicks in the context of conjugate connections by:
\begin{equation}
\nabla^{*}=\nabla^{\Theta}=\Theta^{-1}\nabla\Theta.
\end{equation}
 \cref{eq:gaug}, brought us to conclude that, any symplectic manifold supports a linear perturbation of Hicks of SMAT connections as follows :
 \vspace{-20pt}
\begin{figure}[H]
    \centering
\begin{equation}
    \begin{tikzcd}
\begin{array}{cc}
\nabla^* = \Theta^{-1}\nabla \Theta  \\ 
\end{array}
    \ar[r, <-, bend left, "\Theta"]\ar[r, bend right, "\Theta^{-1}"] & \nabla \omega = 0
\end{tikzcd}
\end{equation}
\caption{Linear perturbation of a symplectic connection by an almost complex structure.}
    \label{fig:Linearper}
\end{figure}

\subsection{The dual connection of any symplectic connection is an almost symplectic connection}

\begin{defn}
   \cite{vaisman1985symplectic},\cite{tondeur1962affine} Let $(M^{2d},\omega)$ be 
an almost symplectic manifold. A linear connection $\nabla$ on $M^{2d}$ is 
called an almost symplectic connection if:
\begin{equation}
    \nabla\omega=0.
\end{equation}
\end{defn}

 \begin{proposition}
 The dual connection of any symplectic connection on an almsot Khaler manifold $(M^{2d},\omega,\Theta, g)$ is an almsot-symplectic connection.
 \label{pr:paralell}
 \end{proposition}
 \begin{proof}
     Let us choose a symplectic connection $\nabla$ on $(M^{2d},\omega,\Theta, g)$. According to \cref{eq:dualexp}, we have:
     \begin{equation}
\nabla^{*}_{X}Y=\nabla_{X}Y-\Theta(\nabla_{X} \Theta)Y,\,\,\forall X,Y
\end{equation}
Now by a direct calculation, we get:

\begin{equation}
   \Theta\nabla^{*}_{X}Y=\Theta\nabla_{X}Y+\nabla_{X}\Theta Y-\Theta\nabla_{X}Y
\end{equation}
From the above, we deduce that

\begin{equation}
    \nabla_{X}\Theta Y=\Theta\nabla^{*}_{X}Y,\,\,\forall X,Y.
    \label{eq:gauver}
\end{equation}

From a straightforward computation, we get

\begin{equation}
    (\nabla_{X}^{*}\omega)(Y,Z)=X.\omega(Y,Z)-\omega(\nabla_{X}^{*}Y,z)-\omega(Y,\nabla_{X}^{*}Z)
\end{equation}

Now using the duality condition, we obtain

\begin{equation}
   (\nabla_{X}^{*}\omega)(Y,Z)=g(\nabla_{X}\Theta Y-\Theta\nabla_{X}^{*}Y,Z)
\end{equation}
It follows from \cref{eq:gauver} that

\begin{equation}
    \nabla^{*}\omega=0.
\end{equation}
Then, $\nabla^{*}$ is an almost-symplectic connection.
 \end{proof}

\vspace{-20pt}
\begin{figure}[H]
    \centering
\begin{equation}
  \begin{tikzcd}
    \nabla^*\omega=0 \ar[r, <-, bend left, "g"]\ar[r, bend right, "g"] & 
    \begin{array}{c}
\nabla\omega=0 \\T\nabla=0
    \end{array}
\end{tikzcd} \hspace{1.5cm}  
\end{equation}
\caption{Linear perturbation of a symplectic connection is an almost symplectic connection.}
    \label{fig:Linearpsymper}
\end{figure}

 \subsubsection{Hermitian connections and the dual of a symplectic connection}

Two natural connections exist for a given hermitian metric $g$ on $M^{2d}$. The first is Levi-Civita connection $\nabla^{lc}$ of the metric $g$ and the second is the hermitian connection $\nabla^{\text{h}}$, satisfying $$\nabla^{\text{h}}g=0,\,\nabla^{\text{h}}\Theta=0,$$

and that the $(1,1)$-part of the torsion vanishes, i.e., 

\begin{equation}
    T^{\nabla^{\text{h}}}(\Theta X,\Theta Y)=T^{\nabla^{\text{h}}}(X,Y),\,\,\forall X,Y.
\end{equation}

There are two well-known Hermitian connections. The Bismut connection (see \cite{bismut1989local}) and the Chern connection (see \cite{chern1946characteristic},\cite{vanhecke1980chern}).  
\vspace{2ex}

\begin{proposition}
  The dual connection $\nabla^{*}$ of any symplectic connection $\nabla$ on an almost K\"ahler manifold $(M^{2d},\omega,\Theta,g)$ can be expressed as follows:
  \begin{enumerate}
      \item $\nabla^{*}=\nabla^{\text{h}}+L$, where $\omega(L(X,Y),Z)=\omega(L(X,Z),Y),\quad \forall X,Y,Z.$
   \item $\nabla^{*}=\nabla^{lc}-\frac{1}{3}\Theta(\nabla^{lc} \Theta)-\frac{1}{3}\Theta(\nabla^{lc}\Theta)+K$, where\; $\omega(K(X,Y),Z)=\omega(K(X,Z),Y)$, $\forall X,Y,Z$.
  \end{enumerate}
\end{proposition}
\begin{proof}
\begin{enumerate}
    \item The dual connection $\nabla^{*}$ can be stated as:

$$\nabla^{*}=\nabla+L,\,\, \text{where}\,\, L\in T^{2}_{1}M^{2d+1}.$$
Departing from proposition \ref{pr:paralell}, we deduce that:
\begin{equation}
\nabla^{h}\omega=0,\nabla^{*}\omega=0.
\label{eq:deuxpa}
\end{equation}
Relying on \cref{eq:deuxpa}, we obtain by means of a direct computation that:
\begin{equation}
(\nabla^{*}_{X}\omega)(Y,z)=(\nabla^{h}_{X}\omega)(Y,z)+\omega(L(X,Y),Z)-\omega(L(X,Z),Y).
\end{equation}
From the above, the conclusion follows.

\item 
Resting upon \cite{bieliavsky2006symplectic}, any symmetric connection $\nabla^{0}$ defines a symplectic connection as follows:

$$\widetilde{\nabla}=\nabla^{0}+\frac{1}{3}U(X,Y)+\frac{1}{3}U(Y,X),$$
where $ \nabla^0_{X}\omega (Y,Z)=\omega(U(X,Y),Z).$ Now let's consider $\nabla^{lc}=\nabla^{0}$. By a straightforward calculation, we obtain:

$$\widetilde{\nabla}_{X}Y=\nabla^{lc}_{X}Y-\frac{1}{3}\Theta(\nabla_{X}^{lc}\Theta)Y-\frac{1}{3}\Theta(\nabla_{X}^{lc}\Theta)Y,\quad\forall X,Y.$$

As in the first case, we conclude that

$$\nabla_{X}^{*}Y=\nabla^{lc}_{X}Y-\frac{1}{3}\Theta(\nabla_{X}^{lc}\Theta)Y-\frac{1}{3}\Theta(\nabla_{X}^{lc}\Theta)Y+ K(X,Y),$$
where 
$\omega(K(X,Y),Z)=\omega(K(X,Z),Y).$
\end{enumerate}
\end{proof}

\subsubsection{\texorpdfstring{$\alpha$-}{Lg}connections are always almost symplectic connections}
 
Let us consider a pair $(\nabla,\nabla^*)$ of SMAT connections in an almost K\"ahler manifold $(M^{2d},\omega,g,\Theta)$. Exactly as in the case of co-K\"ahler manifold, we introduce the family $\alpha$-connections for a pair $(\nabla,\nabla^{*})$ as:

\begin{equation}
    \nabla^{\alpha}=\frac{1+\alpha}{2}\nabla^{*}+\frac{1-\alpha}{2}\nabla,\,\, \alpha\in \mathbb{R},
\end{equation}

The 1-parameter family $(\nabla^{\alpha})_{\alpha\in\mathbb{R}}$ connections
 are distinguished  properties of being all almost symplectic connections.
\vspace{1ex}
\begin{proposition}    \label{pr:alphaprara}
    On any almost K\"ahler manifold, we have:
\begin{align}
        \nabla^{\alpha}\omega & =0;\\ 
  R^{\nabla^{\alpha}}(.,.)\omega & =0; \\
        R^{\nabla^{\alpha}}(.,.). & = \Theta^{-1} R^{\nabla^{-\alpha}}(.,.)\Theta;\\
          U(\Theta.,.)&=\Theta U(.,.), \, where\, \, U=\nabla^{*}-\nabla.
\end{align}
\end{proposition}
 \begin{proof}
    The result is therefore obtained by a straightforward computation.
 \end{proof}

\subsection{A new characterisation of Kähler  metrics }

Let's consider an almsot K\"ahler manifold $(M,\omega,\Theta, g)$. The following S. I. Goldberg conjecture \cite{goldberg1969integrability} is well-known concerning the integrability of the almost complex structure of an almost K\"ahler manifold.

\vspace{2ex}

\textbf{Goldberg conjecture}: "The almost complex structure of a compact Einstein almost K\"ahler manifold
 is integrable (and therefore the manifold is K\"ahler)".\vspace{0.2cm}
 
 There is some significant progress in certain curvature conditions s concerning the Golberg conjecture. \citet{sekigawa1987some} proved the validity of this conjecture for almost Kähler manifolds whose scalar curvature is non negative. A four-dimensional compact almost K\"ahler manifold which is Einsteinian and $*$-Einsteinian is a K\"ahler manifold \citet{sekigawa1990four}. A complete almost K\"ahler manifold of constant sectional curvature is a flat K\"ahler manifold( see \citet{sekigawa1990four}, \citet{oguro1994non}, \citet{olszak1981almost}). It is still open otherwise.\vspace{0.2cm}

 Before stating the main theorem on this subsection, let us put forward certain pertinent remarks.
\vspace{1ex}
 \begin{remark}     Consider a K\"ahler manifold $(M^{2d},\omega,\Theta,g)$. On open holomorphic Darboux charts $(U^{l},z^{1},....,z^{d})$ centered at some point $p$, choose the local symplectic connection $\nabla^{l}$ which has zero Christoffel symbols (ie $(\Gamma^{l}_{{ki}})^{s}=0)$:
     \[\nabla^{l}_{\partial z_{i}}\partial z_{j}=0 , \]   
 \[
   \nabla^{l}_{\partial z_{i}}\partial \Bar{z}_{j}=0. \]
With respect to the duality condition, one proves that:
$$\partial_{k}g_{i\Bar{j}}=(\Gamma^{l}_{{ki}})^{s}g_{\Bar{j}s}+(\Gamma^{l*}_{k\Bar{j}})^{s}g_{is},$$

$$\partial_{\Bar{k}}g_{i\Bar{j}}=(\Gamma^{l}_{{\Bar{k}i}})^{s}g_{\Bar{j}s}+(\Gamma^{l*}_{\Bar{k}\Bar{j}})^{s}g_{is},$$

where $\Gamma^{l*}_{kj,i}$ are the Christoffel symbols of the dual connection $\nabla^{l*}$ of $\nabla^{l}$.  With reference to the definition of $\nabla^{l}$, we deduce that:
$$g^{i\Bar{m}}\partial_{k}g_{i\Bar{j}}=(\Gamma^{l*}_{k\Bar{j}})^{\Bar{m}},$$
$$g^{i\Bar{m}}\partial_{\Bar{k}}g_{i\Bar{j}}=(\Gamma^{l*}_{\Bar{k}\Bar{j}})^{\Bar{m}}.$$

Finally, we infer that

$$g^{i\Bar{m}}(\partial_{k}g_{\Bar{j}i}-\partial_{\Bar{j}}g_{ki})=(\Gamma^{l*}_{k\Bar{j}})^{\Bar{m}}-(\Gamma^{l*}_{\Bar{j}k})^{\Bar{m}},$$

$$g^{i\Bar{m}}(\partial_{\Bar{k}}g_{i\Bar{j}}-\partial_{\Bar{j}}g_{i\Bar{k}})=(\Gamma^{l*}_{\Bar{k}\Bar{j}})^{\Bar{m}}-(\Gamma^{l*}_{\Bar{j}\Bar{k}})^{\Bar{m}}.$$

Clearly, on the local holomorphic Darboux chart $U^{l}$, the connection $\nabla^{l*}$ is a symmetric connection.
 \end{remark}
\vspace{1ex}
Proceeding in the same way as the above remark, let's define the space of all SMAT connections $(\nabla,\nabla^{*})$ on an almost K\"ahler manifold $(M^{2d},\omega,\Theta,g)$ as:


$$\mathcal{E}_{(g,\omega)}(M^{2d+1})=\left\{ (\nabla,\nabla^*) {\lvert}
\nabla\omega=0,\,T^{\nabla}=0\right\}$$
\vspace{1ex}

\begin{defn}
    A statistical almost-K\"ahler $(M^{2d},\omega,\Theta,g)$ is an almost K\"ahler manifold such that $\mathcal{S}_{(g,\omega)}(M^{2d})$ contains statistical connections. These connections will be called the statistical symplectic connections.
\end{defn}

\vspace{2ex}

\begin{proposition}
 An almost K\"ahler manifold $(M^{2d},\omega,\Theta,g)$ is a Kähler manifold if and only it is a statistical almost-K\"ahler manifold.
\label{th:AlmsKha}   
\end{proposition}   

\begin{proof}
The proof of the theorem relies on the same strategy adopted with the theorem \ref{thm:principal}. Let us choose $(\nabla,\nabla^*)$, a statistical symplectic connection. Based on the proposition \ref{pr:paralell}, we deduce that:
\begin{equation}
\nabla^{*}\omega=0.
\label{eq:107}
\end{equation}
It follows from \cref{eq:102} that:
\begin{equation}
    \label{eq:New}
    \nabla^{lc}\omega=0.
\end{equation}
From the above \cref{eq:New}, we conclude that $(M^{2d},\omega,g)$ is a Kähler manifold.\vspace{0.2cm}

Conversely, it is well known that in any Kähler manifold $(M^{2d},\omega,g)$, the Levi-Civita connection of the K\"ahler metric satisfies: 
\begin{equation}
   \nabla^{lc}\omega=0.
   \label{eq:paralevicivita}
\end{equation}
Departing from the above equation, we obtain  $\nabla^{lc}\in \mathcal{S}_{(g,\omega)}(M^{2d})$. We conclude that the Kähler manifold $(M^{2d},\omega,g)$ is a statistical almost-K\"ahler manifold.
\end{proof}

\begin{corollary}\label{lem:principal}
   We obtain the following properties on any K\"ahler manifold $(M^{2d}, \omega, g)$:
\begin{enumerate}
   \item $\nabla^{\alpha}=\nabla^{lc}-\frac{1}{2}\alpha\Theta(\nabla\Theta),\quad\forall a\in\mathbb{R}.$
   \item The Ricci curvature $Ric^{\nabla^{\alpha}}$ is symmetric.
\end{enumerate}   
\end{corollary}
\begin{proof}
    The result is therefore obtained by a straightforward computation.
 \end{proof}
\vspace{0.2cm}

\begin{remark}
    To sum up, with reference to proposition \ref{th:AlmsKha}, the Kähler metrics on an almost Kähler manifold can be characterized as the Hermitian metrics $g$ such that the dual of a symplectic connection with respect to $g$ is a symmetric connection. In other words, Kähler metrics are Hermitian metrics such that the dual of a symplectic connection is  also a symplectic connection.\vspace{-20pt}
\begin{figure}[H]
    \centering
\begin{equation}
    \begin{tikzcd}
\begin{array}{cc}
\nabla^*\omega=0\\
T\nabla^*=0
\end{array}
    \ar[r, <-, bend left, "g"]\ar[r, bend right, "g"] & \begin{array}{cc}
\nabla\omega=0\\
T\nabla=0
\end{array}
\end{tikzcd}
\end{equation}
\caption{Statistical characterization of Kahler metrics.}
    \label{fig:Statisticalcarac}
\end{figure}
From the above remark, we deduce that our statistical approach is a generalization of the well known result in K\"ahler geometry on the characterization of a K\"ahler metric, which holds that on an almost K\"ahler manifold $(M^{2d},\omega,g,\Theta)$, K\"ahler metrics are Hermitian metrics satisfying (see \cite{tian2012canonical}):
\begin{equation}
    \nabla^{lc}\omega=0.
\end{equation}
\end{remark}

\subsection{Statistical characterization of K\"ahler metrics in terms of parallel transport}
\label{subsect:stacha}
The Levi–Civita connection corresponds to the only symmetric metric connection that preserves the metric by parallel transport, i.e. 

\begin{equation}
g(X_{c(0)},Y_{c(0)})=g(\tau^{\nabla^{lc}}_{_{c(t)}}X,\tau^{\nabla^{lc}}_{c(t)}Y),
\end{equation}
 where 

\begin{equation}
    \nabla^{lc}_{\dot{c}(t)}X=0 \,\,\text{and}\,\,
    \nabla^{lc}_{\dot{c}(t)}Y=0.
\end{equation}
\vspace{1ex}
The dual connections refer to another way of preserving
the metric (see\cite{amari1987differential}). Denote by $\tau^{\nabla}\,\text{and}\, \tau^{\nabla^{*}}$ the parallel transport along curves relative to $\nabla$ and $\nabla^{*}$ , respectively,

\begin{equation}
g(X_{c(0)},Y_{c(0)})=g(\tau^{\nabla}_{_{c(t)}}X,\tau^{\nabla^{*}}_{c(t)}Y),
\end{equation}

where 
\begin{equation}
    \nabla_{\dot{c}(t)}X=0 \,\,\text{and}\,\,
    \nabla^*_{\dot{c}(t)}Y=0.
\end{equation}

 The K\"ahler metrics on a complex manifold $(M^{2d},\Theta)$ can be accounted for in terms of Riemannian metrics satisfying the following equation:

\begin{equation}
    \nabla^{lc}\Theta=\Theta\nabla^{lc}
    \label{eq:ka}
\end{equation}

In terms of parallel transport, condition \cref{eq:ka} is equivalent to: 

\begin{equation}
\tau^{\nabla^{lc}}\Theta X =\Theta\tau^{\nabla^{lc}}X.
\end{equation}

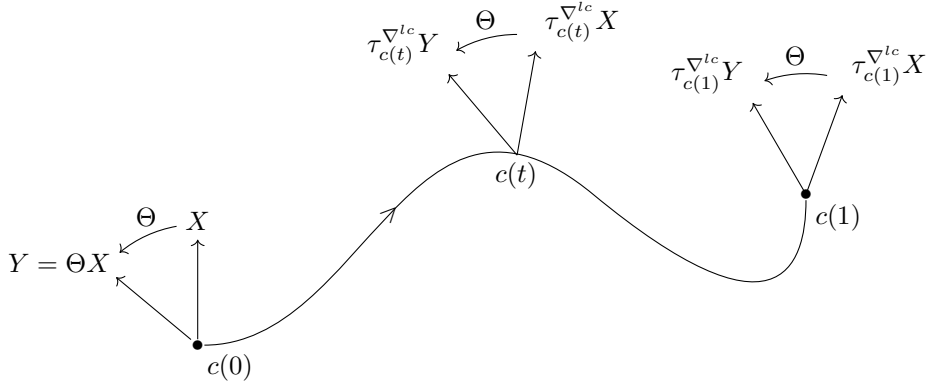
\begin{figure}[H]
    \centering
\begin{tikzpicture}[scale=2]
\node[inner sep=0pt] (c0) at (0,0) {\small $\bullet$};
\node[inner sep=0pt] (c1) at (4,1) {\small $\bullet$};
\draw (c0)node[ below right]{$c(0)$}..controls (1.1,0) and (1.4,2)..(2.6,1) node[midway, sloped]{\small $>$} node[pos = 0.77, below right]{$c(t)$}.. controls(3.8,0) and (4,0.5)..(c1)node[ below right]{$c(1)$};
\draw[->] (c0)--++(0,0.7) node[above]{$X$};
\draw[->] (c0)--(140:0.7) node[above left=-2pt]{$Y=\Theta X$};
\draw[->] (100:0.8) arc (100:130:0.8) node[midway, above]{$\Theta$};
\begin{scope}[xshift = 4cm, yshift = 1cm, rotate = -20]
\draw[->] (0,0)--++(0,0.7) node[above right]{$\tau_{c(1)}^{\nabla^{lc}}X$};
\draw[->] (0,0)--(140:0.7) node[above left]{$\tau_{c(1)}^{\nabla^{lc}}Y$};
\draw[->] (100:0.8) arc (100:130:0.8) node[midway, above]{$\Theta$};
\end{scope}
\begin{scope}[xshift = 2.1cm, yshift = 1.26cm, rotate = -10]
\draw[->] (0,0)--++(0,0.7) node[above right]{$\tau_{c(t)}^{\nabla^{lc}}X$};
\draw[->] (0,0)--(140:0.7) node[above left]{$\tau_{c(t)}^{\nabla^{lc}}Y$};
\draw[->] (100:0.8) arc (100:130:0.8) node[midway, above]{$\Theta$};
\end{scope}
\end{tikzpicture}
  \caption{Parallel transport and K\"ahler metrics.}
    \label{eq:tr} 
\end{figure}

Grounded on proposition \ref{pr:alphaprara}, in any almost K\"ahler manifold, we get the following identities:
\vspace{-17pt}
\begin{eqnarray}
    \Theta^{-1}\nabla^{-\alpha}\Theta=\nabla^{\alpha}\Leftrightarrow\nabla^{\alpha}\omega=0\Leftrightarrow\omega(X,Y)=\omega(\tau^{\nabla^{\alpha}}X,\tau^{\nabla^{\alpha}}Y);
    \label{eq:paraaaa1}\\
     \Theta^{-1}\nabla^{\alpha}\Theta=\nabla^{-\alpha}\Leftrightarrow\nabla^{-\alpha}\omega=0\Leftrightarrow \omega(X,Y)=\omega(\tau^{\nabla^{-\alpha}}X,\tau^{\nabla^{-\alpha}}Y).
     \label{eq:paraaaa2}
\end{eqnarray}
In terms of parallel transport, both of the following diagrams are commutative.
\[
\xymatrix {
     X  \ar[r]^{\tau^{\nabla^{\alpha}}} \ar[d]^{\Theta} &  \tau^{\nabla^{\alpha}}X  \ar[d]^{\Theta}  \\
     \Theta X \ar[r]^{\tau^{\nabla^{-\alpha}}} &  \Theta(\tau^{\nabla^{\alpha}}X)
     }\hspace{1cm}
\xymatrix {
     X  \ar[r]^{\tau^{\nabla^{-\alpha}}} \ar[d]^{\Theta} & \tau^{\nabla^{-\alpha}} X \ar[d]^{\Theta}  \\
     \Theta X \ar[r]^{\tau^{\nabla^{\alpha}}} &  \Theta(\tau^{\nabla^{-\alpha}} X)}
\]\vspace{-10pt}
Referring to \cref{eq:paraaaa1} and \cref{eq:paraaaa2}, we get:
\begin{eqnarray}\label{eq:trstt}
\omega(\tau^{\nabla^{\frac{\nabla^{\alpha}+\nabla^{-\alpha}}{2}}}X,\tau^{\nabla^{\frac{\nabla^{\alpha}+\nabla^{-\alpha}}{2}}}Y)=\omega(X,Y),\,\,\,\forall X,Y.
\end{eqnarray}
According to the duality condition, finally we have:
\begin{align}
g(\tau^{\nabla^{\frac{\nabla^{\alpha}+\nabla^{-\alpha}}{2}}}X,\tau^{\nabla^{\frac{\nabla^{\alpha}+\nabla^{-\alpha}}{2}}}Y)&=g(X,Y),\,\,\,\forall X,Y.\\
\omega(\tau^{\nabla^{\frac{\nabla^{\alpha}+\nabla^{-\alpha}}{2}}}X,\tau^{\nabla^{\frac{\nabla^{\alpha}+\nabla^{-\alpha}}{2}}}Y)&=\omega(X,Y),\,\,\,\forall X,Y.
\end{align}
From our statistical point of view, K\"ahler metrics are those Hermitian metrics corroborating  that $\nabla^{-\alpha}$ is a symplectic connection. Hence,
\vspace{2ex}
 \begin{equation}
\tau^{\frac{\nabla^{\alpha}+\nabla^{-\alpha}}{2}}=\tau^{\nabla^{lc}}.
\end{equation}
Consequently, with respect to \cref{eq:tr}, we have:
\begin{equation}
\omega(X,Y)=\omega(\tau^{\nabla^{lc}}X,\tau^{\nabla^{lc}}Y).
\end{equation}

\begin{figure}[H]
    \centering
\begin{tikzpicture}[scale = 4]
    \node at (0,0) {$\bullet$};
\begin{scope}[rotate = 40, scale = 0.7]
\draw[->] (0,0)--++(0,0.7) node[above]{$X$};
\draw[->] (0,0)--(140:0.7) node[left]{$\Theta X$};
\draw[->] (100:0.8) arc (100:130:0.8) node[midway, above]{$\Theta$};
\end{scope}
\begin{scope}[xshift = 2cm, yshift = 1cm, rotate = -160, scale = 0.5]
\draw[->] (0,0)--++(0,0.7) node[below]{$\tau_{c(t)}^{\nabla^{lc}}X$};
\draw[->] (0,0)--(140:0.7) node[right]{$\tau_{c(t)}^{\nabla^{lc}}\Theta X$};
\draw[->] (100:0.8) arc (100:130:0.8) node[midway, above]{$\Theta$};
\end{scope}
\begin{scope}[xshift = 2cm, yshift = 1cm, rotate = 10, scale = 0.5]
\draw[->,red] (0,0)--++(0,0.7) node[above]{$\tau_{c(t)}^{\nabla^{\alpha}}X$};
\draw[->,blue] (0,0)--(140:0.7) node[above left]{$\tau_{c(t)}^{\nabla^{-\alpha}}\Theta X$};
\draw[->] (100:0.8) arc (100:130:0.8) node[midway, above]{$\Theta$};
\end{scope}
    \node at (2,1) {$\bullet$};
    \draw[red] (0,0) .. controls (0,1).. (2,1)node[pos = 0.3, left]{$\tau_{\nabla^{\alpha}}$};
    \draw (0,0) .. controls (1.25,0.3) and (0.75,0.7).. (2,1) node[pos = 0.3, left]{$\tau_{\nabla^{lc}=\frac{\nabla^{\alpha}+\nabla^{-\alpha}}{2}}$};
    \draw[blue] (0,0) .. controls (2,0).. (2,1)node[pos = 0.3, above]{$\tau_{\nabla^{-\alpha}}$};   
\end{tikzpicture}
 \caption{Parallel transport and $\alpha$-statistical symplectic connections.}
    \label{eq:transport} 
\end{figure}
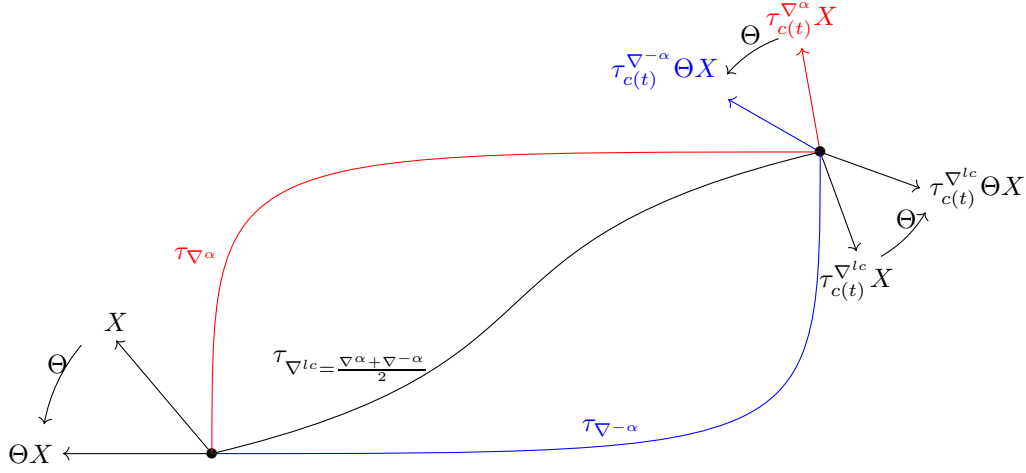

\subsection{Statistical generalization of S. Kobayashi theorem}

\label{subsection:Kobay}
Our new approach allows us to provide a generalisation of the S. Kobayashi theorem, which states that:
\vspace{1ex}
\begin{thm}\cite{kobayashi1961compact}
    A compact K\"ahler manifold $(M^{2d},\Omega,g)$ with a positive definite Ricci curvature is simply connected. Furthermore, its first Betti number is zero.
\end{thm}

\vspace{2ex}

Let's consider a statistical symplectic connection. We obtain the same result as in Kobayashi's theorem when the Ricci of a statistical symplectic connection is positive definite.

\vspace{2ex}

\begin{proposition}
     Consider the family of $\alpha$-statistical symplectic connections $\nabla^{\alpha}$ on a compact K\"ahler manifold $(M^{2d},\Omega,g)$. Assume that there exists an $\alpha_{0}\in\mathbb{R}$( but is not necessary $\alpha_{0}=0$) such that $Ric^{\nabla^{\alpha_{0}}}$ is positive definite. Therefore, its first Betti number is zero.
\end{proposition}

\begin{proof}
With  reference to lemma \ref{pr:alphaprara}, it is known that

   \begin{equation}
\nabla^{\alpha_{0}}\Omega^{d}=0.
   \end{equation} 
   
    Referring to the above formula, we deduce that  $(g,\nabla^{\alpha_{0}},\Omega^{d})$ has an equi-affine statistical structure. By assumption $Ric^{\nabla^{\alpha_{0}}}>0$, then, departing from \cite{opozda2015bochner}[Th 9.6], we conclude that its first Betti number is zero.
\end{proof}

\subsubsection{S.Goldberg integrability condition and statistical integrability condition}

It is interesting to face our statistical characterization of K\"ahler metrics  with the integrability condition of S.Golberg. Using straightforward calculations, we can prove that our statistical integrability condition implies the integrability condition of Goldberg \cite{goldberg1969integrability1}. Let us then recall the integrability condition of S. Goldberg.

\vspace{2ex}

\begin{thm}\cite{goldberg1969integrabilityn}
Consider an almost Kähler manifold $(M^{2d},\Omega,g,\Theta)$. If $R^{\nabla^{lc}}\circ \Theta=\Theta\circ R^{\nabla^{lc}} $, then $(M^{2d},\omega,\Theta,g)$ is a Kähler manifold.
\label{th:golb}
\end{thm}
\vspace{2ex}

\begin{proposition}
    Let $(M^{2d},\omega,\Theta,g)$ be an almost Kähler manifold. The following statements are equivalent:
        \begin{enumerate}
        \item[(1)] $M^{2d}$ is a statistical almost-K\"ahler manifold, 
        
        \item[(2)] $R^{\nabla^{lc}}\circ \Theta=\Theta\circ R^{\nabla^{lc}}$,
        
        \item[(3)] $M^{2d}$ is a Kähler manifold.
    \end{enumerate}
    \label{subsec:Gol}
\end{proposition}
\begin{proof}
Let's prove that $(1)$ implies $(2)$. Consider an almost Kähler manifold $(M^{2d+1},\Omega,g,\Theta)$. Remember that, from our characterization, a Kähler manifold is the one which contains statistical symplectic connections. Now, let's choose a statistical shs connection $(\nabla,\nabla^{*})$. Using Proposition \ref{pr:alphaprara}, we have: 
\begin{equation} \label{eq:122}
    \Theta R^{\nabla}(Y,X)Z = R^{\nabla^{*}}(X,Y)\Theta Z\,\, \text{and}\,\, R^{\nabla}(Y,X)\Theta Z = \Theta R^{\nabla^{*}}(X,Y)Z,\,\, \forall X,Y,Z.
\end{equation} 

On the other side, it is known that for a statistical connection $(\nabla,\nabla^{*})$, we have:
\begin{equation}
    R^{\nabla}(X,Y)Z+R^{\nabla^{*}}(X,Y)Z = 2R^{\nabla^{lc}}(X,Y)Z+2[U_{X},U_{Y}]Z,
    \label{eq:cur}
\end{equation} 
where $U=\nabla^{*}-\nabla=-\Theta(\nabla\Theta)$. Now, using \cref{eq:cur}, we obtain:
\begin{equation}
    2R^{\nabla^{lc}}(X,Y)\Theta Z= \Theta(R^{\nabla}(X,Y) Z+R^{\nabla^{*}}(X,Y) Z)-2[U_{X},U_{Y}]\Theta Z.
    \label{eq:125}
\end{equation}
Using proposition \ref{pr:alphaprara} once again, it follows that:
\begin{align}
        [U_{X},U_{Y}]\Theta Z  &=U_{X}U_{Y}\Theta Z-U_{Y}U_{X}\Theta Z  \\
         &=\Theta[U_{X},U_{Y}]Z, 
\end{align}
for any vector field $X, Y, Z$ in, $M^{2d+1}.$ Again, it follows from the above calculations that:
\begin{equation}
    2R^{\nabla^{lc}}(X,Y)\Theta Z= \Theta(R^{\nabla}(X,Y) Z+R^{\nabla^{*}}(X,Y) Z)-2\Theta[U_{X},U_{Y}]Z.
    \label{eq:129}
\end{equation}
According to \cref{eq:129}, the integrability condition of S.Goldberg and K.Yano is identified as:
\begin{equation}
   R^{\nabla^{lc}}(X,Y)\Theta Z=\Theta R^{\nabla^{lc}}(X,Y)Z.
\end{equation}
From Theorem \ref{th:golb}, we infer that $(2)$ implies $(3)$. Based on proposition \cref{th:AlmsKha}, we deduce that $(3)$ implies $(2)$.

\end{proof}

\subsection{Positive answer to question 1}
\label{subsec:Posit}
\subsubsection{The first approach to  the resolution of question 1}

Before providing an answer to question1, we need to recall the definition of the isostatistical embedding. 
\vspace{2ex}
\begin{defn}
    A smooth embedding $i$ from a statistical manifold $(M,g,T)$ to a statistical manifold $(N,h,K)$ is isostatistical if it preserves the statistical
structure, i.e. $g=i^{*}h$ and $T=i^{*}K$.
\end{defn}
\vspace{2ex}
\begin{thm}
    Any Kähler and co-Kähler metric is a Fisher information metric.
     \label{th:KacOka}
\end{thm}
\begin{proof}
Consider a Kähler manifold $(M^{2d},\Omega_{g},g)$. Relying on \cref{th:AlmsKha}, we know that the Kähler metric $g$ always admits statistical connections $(\nabla,\nabla^{*})$ verifying
\begin{equation}
    \nabla^{*}\Omega_{g}=0,\nabla\Omega_{g}=0.
    \label{eq:deuxpara}
\end{equation}
Let's define a totally symmetric tensor by: 
\begin{equation}
    T^{Ka}(X,Y,Z)=g(U(X,Y),Z),
    \label{eq:Tensortotallysy}
\end{equation}
where $U=\nabla^{*}-\nabla.$ Resting on \cref{eq:deuxpara} and \cref{eq:Tensortotallysy}, we deduce that $(M^{2d},\Omega_g,g,T^{Ka})$ is a statistical manifold (Amari-Chentsov structure). According to the works of \citet{le2006statistical}, we deduce that the Kähler manifold $(M^{2d}, \Omega_g, g, T^{Ka})$ can be embedded isostatically into $\mathcal{P}(\Xi^{\text{N}})$ the set of probability density functions on the sample space $\Xi^{N}$, for $N$ sufficiently large, such that the Kähler metric $g$ is induced by the
Fisher metric $g^{F}$ on $\mathcal{P}(\Xi^{\text{N}})$ and $T^{Ka}$ is induced by the Amari-Chentsov tensor $T^{A-C}$ on $\mathcal{P}(\Xi^{\text{N}})$. For the co-K\"ahler metrics, the proof is based on the same strategy adopted for K\"ahler metrics. It follows, from Theorem \ref{thm:principal}, that on a co-K\"ahler manifold $(M^{2d+1},\Omega,\alpha,g)$, we have statistical connections $(\nabla,\nabla^{*})$ satisfying:
\begin{equation}
    \nabla E=0,\, \nabla\Omega=0 ,\nabla\alpha=0;
    \label{eq:Nic}
\end{equation}
\begin{equation}
    \nabla^{*} E=0,\, \nabla^{*}\Omega=0,\nabla^{*} \alpha=0.
    \label{eq:Nice}
\end{equation}
Let's define a totally symmetric tensor by: 
\begin{equation}
    T^{CoKa}(X,Y,Z)=g(U(X,Y),Z),
    \label{eq:Tensortotally}
\end{equation}
where $U=\nabla^{*}-\nabla.$ We deduce from \cref{eq:Nic,eq:Nice,eq:Tensortotally} that $(M^{2d+1},\Omega,\alpha,g,T^{CoKa})$ is a statistical manifold (Amari-Chentsov structure). Investing once again \citet{le2006statistical} theorem, the result follows at once.
\end{proof}

\subsubsection{The second approach to the resolution of the question 1}

At this stage, we propose another approach based on an explicit construction. Consider a co-K\"ahler manifold $(M^{2d+1},\Omega,\alpha,g)$. Let's build up a family of symmetric connections on $(M^{2d+1},\Omega,\alpha,g)$ as follows:

$$\nabla^{\epsilon,a}=\frac{1+a}{2}\nabla^{\epsilon}+\frac{1-a}{2}\nabla^{-\epsilon},\,\,\forall a\in\mathbb{R},$$

where
\begin{equation}
\nabla^{\epsilon}=\nabla^{lc}+\epsilon\alpha(.)\alpha(.)E,\,\,\forall \epsilon\in\mathbb{R}.
\end{equation}
By a direct computation, we get 
\begin{equation}
    (\nabla^{\epsilon})^{*}=\nabla^{-\epsilon},
\end{equation}
Hence, 
\begin{equation}
    (\nabla^{\epsilon,a})^{*}=\nabla^{\epsilon,-a}.
\end{equation}
Consequently, the pair $(\nabla^{\epsilon,a},\nabla^{\epsilon,-a})$ is a statistical connection.\\
Now, with a straightforward calculation, one can notice that
\begin{equation}
\nabla^{\epsilon,a}E=a\epsilon\alpha(.)E,\, \nabla^{\epsilon,a}\Omega=0,\,\,\forall \epsilon\in\mathbb{R}\,\,\forall a\in\mathbb{R}.
\label{eq:mycosn}
\end{equation}

Thus, departing from \cref{eq:mycosn}, it follows that the foliation $\ell_{\Omega}$ is $\nabla^{\epsilon,a}$-geodesic.

\vspace{1ex}

Now, let's use the duality condition determined as follows:
\begin{equation}
X.g(Y,Z)=g(\nabla^{\epsilon,a}_{X}Y,Z)+g(Y,\nabla_{X}^{\epsilon,-a}Z)\,\,\forall X\in\Gamma(TM^{2d+1}),\,\, \forall Y\in\ell_{\Omega},\,\,
    \forall Z\in\ell_{\alpha}.
    \label{eq:dualmtipl}
\end{equation}

With respect to \cref{eq:dualmtipl}, we deduce that $\ell_{\alpha}$ is $\nabla^{\epsilon,-a}$-geodesic.

\vspace{2ex}

\begin{figure}[H]
    \centering
\begin{center}
    \begin{tikzpicture}[xscale = 2]
    \fill[gray!20] (0,0) -- (2,0) -- (3,2) -- (1,2) -- cycle;
    \draw (0,0) -- (2,0) -- (3,2) node[pos = 0.4, below right]{$\ell_\alpha = \text{ker }\alpha$} node[pos = 0.7, below right]{$\nabla^{\epsilon,-a}$-totally geodesic} -- (1,2) -- cycle;
    \draw[thick] (1.5,-1) -- ++(0,1);
    \draw[dashed, thick] (1.5,0) -- ++(0,1);
    \draw[->, thick] (1.5,1) -- ++(0,2) node[pos = 0.9, left=5pt]{$\ell_\Omega = <E>$}node[pos = 0.9, right=5pt]{$\nabla^{\epsilon,a}$-totally geodesic};
\end{tikzpicture}
 \end{center}
 \caption{Canonical statistical connections on a co-K\"ahler manifold.}
    \label{eq:canonicalconn} 
\end{figure}
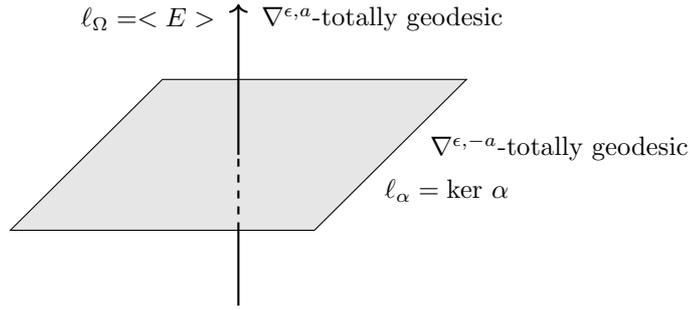
 Finally, for a co-K\"ahler manifold $(M^{2d+1},\Omega,g)$, we have

\begin{figure}[H]
    \centering
\begin{center}
    \begin{tikzcd}
\begin{array}{cc}
\nabla^{\epsilon,-a}\omega=0\\
\nabla^{\epsilon,-a}E=-a\epsilon\alpha(.)E\\
T^{\nabla^{\epsilon,-a}}=0
\end{array}
    \ar[r, <-, bend left, "g"]\ar[r, bend right, "g"] & \begin{array}{cc}
\nabla^{\epsilon,a}\omega=0\\
\nabla^{\epsilon,a}E=a\epsilon\alpha(.)E\\
T^{\nabla^{{\epsilon,a}}}=0
\end{array}
\end{tikzcd}
\end{center}
 \caption{Co-K\"ahler metric and canonical statistical connections.}
    \label{eq:157trans} 
\end{figure}
\vspace{2ex}
Hence, $(M^{2d+1},\Omega,g,\nabla^{\epsilon,a},\nabla^{\epsilon,-a})$ is a statistical manifold. Now, referring to \cite{le2006statistical}, we deduce that there exists an isostatistical embedding
\begin{equation}
    i:(M^{2d+1},\Omega,\alpha,g,T_{\epsilon,a}^{CoKa})\hookrightarrow ( \mathcal{P}(\Xi^\text{N}), g^{F},T^{A-C}),
    \label{eq:cokalerembending}
\end{equation} 
for $N$ sufficiently large, such that
$$g=i^{*}g^{F} \,\text{and}\,\, T_{\epsilon,a}^{coKa}=i^{*} T^{A-C};$$
where $T_{\epsilon,a}^{coKa}=g(.,U_{\epsilon,a})$, and $U_{\epsilon,a}=\nabla^{\epsilon,-a}-\nabla^{\epsilon,a}$. Accordingly, any co-K\"ahler metric coincides with the Fisher–information metric of $ \mathcal{P}(\Xi^\text{N})$.
\vspace{1ex}

Now consider a K\"ahler manifold $(M^{2d},\Omega,g,\Theta)$, we can construct a co-Kahler manifold as follows $(W^{2d+2}=M^{2d}\times \mathbb{R},\Omega,\alpha,\Tilde{g},\Tilde{\Theta})$, where $\alpha=dt$, $E=(0,\frac{\partial}{\partial_{t}})$, $\Tilde{\Theta}(.,fE)=(\Theta,0) $, $\Tilde{g}=g+\alpha^{2}$ and $\Tilde{\Omega}=\Tilde{g}(\Tilde{\Theta},.)$.
Now, based on \cref{eq:cokalerembending}, there exists a isostatistical embedding $i:(W^{2d+2}=M^{2d}\times \mathbb{R},\Tilde{g},T_{\epsilon,a}^{coKa})\rightarrow (\mathcal{P}(\Xi^\text{N}),g^{F},T^{A-C}) $, such that $\widetilde{g}=i^{*}g^{F}$. Now, identify $M^{2d}$ with the leaf $t=0$ of $W^{2d+2}$; by restriction:

\begin{equation}
    g=\widetilde{g}|_{M^{2d}}=g^{F}(\theta)_{ij}=4\int_{\Omega}\frac{\partial\sqrt{p_{\theta}}}{\partial\theta_{i}}\frac{\partial\sqrt{p_{\theta}}}{\partial\theta_{j}}d\mu.
\end{equation}
\begin{com}  
To summarize this subsection, K\"ahler and co-K\"ahler manifolds, can be viewed as being a parametric family of probability density functions, whereas K\"ahler and coK\"ahler metrics can be regarded as Fisher information metrics. We then obtain the answer to question 1.
\end{com}

\subsection{K\"ahler manifolds and co-K\"ahler manifolds as submanifold of statistical models}
\label{subsec:kahsub}
Scrutinizing through literature, there are three main families of statistical models which are identified in terms of:

\begin{enumerate}
    \item \textbf{The family of discrete distribution} The family of discrete distribution( also called the family of categorical distributions in \cite{nielsen2020elementary}) is a statistical model which has a constant positive curvature. This family belongs to spherical geometry.
    \vspace{2ex}
    \item \textbf{The family of location-scale families} belongs to Euclidean geometry(Euclidean manifold)
\begin{enumerate}
        \item Multivariate normal distributions.(see\cite{costa2015fisher})
\vspace{2ex}
        \item  Logistic Model\cite{rylov2016constant}.
        \vspace{2ex}
       \item  Weibull Statistical Model.\cite{rylov2016constant}
    \end{enumerate}
    \vspace{2ex}
    \item \textbf{The family of location families} belongs to Euclidean geometry(Euclidean manifold)
\end{enumerate}
Fore additional details about the above stated families refer to (\cite{nielsen2020elementary}).
\vspace{2ex}

In information geometry, the study of the nature of submanifolds in a statistical model is of paramount importance
(see\cite{amari2000methods}). From the estimation perspective in information geometry, the existence of efficient estimators depends on the autoparallelity of the corresponding submanifolds in a set of probability
distributions(see\cite{amari2000methods}). In \cite{lauritzen1987statistical}, Lauritzen largely focused on categorizing completely $\alpha$-geodesic submanifolds within specific statistical models (family of normal distribution, family of inverse gaussian family, family of Gamma densities, etc) owing to the significance of these submanifolds in statistical models. In the previous subsection, we demonstrated that K\"ahler and co-K\"ahler manifolds are sub-manifolds of statistical models $(\mathcal{P}(\Xi^{N}),g^{F}, T^{A-C})$(or $(\mathcal{P}(\Xi^{N}),g^{F},\nabla^{a})$). It is quite natural to understand their nature. We first deal with compact K\"ahler and co-K\"ahler manifolds cases, then we handle non-compact cases.
\vspace{2ex}

\begin{proposition}
    Compact K\"ahler and co-K\"ahler manifolds are not $\nabla^{0}$-totally geodesic submanfiolds of $\mathcal{P}(\Xi^{N})$. 
    \label{prop:kaandco}
\end{proposition}
\begin{proof}

First of all, we address compact K\"ahler manifolds. Consider a compact K\"ahler manifold $(M^{2d},\Omega_{g},g)$. According to Theorem\ref{th:KacOka}, we know that 
 $(M^{2d},\Omega_{g},g)$ can be embedded isostatiscally in $(\mathcal{P}(\Xi^{N}),g^{F},T^{A-C})$, for some finite $N$. By assumption, $M^{2d}$ is compact. Then, according to Le's\cite{le2006statistical} work, we deduce that for $N$ sufficiently large $(M^{2d},\Omega_{g},g)$ can be embedded isostatiscally in a probability simplex $(\Delta^{N},g^{F},T^{A-C})$  on a discrete sample space $\Xi$ with $|\Xi|=N+1$. The probability simplex $\Delta^{N}$ is indicated as follows:
 \begin{equation}
     \Delta^{N}=\{p=(p_{i})\in\mathbb{R}^{N+1}, p_{i}>0,\, \sum_{i}^{N+1}p_{i}=1\}
 \end{equation}
and the Fisher metric on $\Delta^{N}$ is specified in terms of
\begin{equation}
    g_{p}^{F}(v,w)=(\sum_{i=1}^{N}\frac{v^{i}w^{j}}{p^{i}})+\frac{v^{N+1}w^{N+1}}{1-\sum_{j=1}^{N}p_{j}},\forall v,w\in T_{p}\Delta^{N}\simeq \mathbb{R}^{N}.
\end{equation}
For further details concerning the probability simplex, consult(\cite{amari2007methods},\cite{amari1987differential}). Departing from \cite{amari2000methods}, \cite{amari1987differential},\cite{kass2011geometrical},\cite{burbea1984informative}, a new parameterization is typically used: $q^{i}=2\sqrt{p_{i}}\,\, i=1,...,N+1$. With this new parametrization, we can identify $\Delta^{N}$ with the positive portion of sphere of radius $2$, denoted $2S_{+}^{N}$. Therefore, the Fisher metric $g^{F}$ is the standard constant positive curvature$(c=\frac{1}{4})$ on $2S_{+}^{N}$, obtained by the restriction of the ambient Euclidean metric of $\mathbb{R}^{N+1}$. Now assume that $(M^{2d},\Omega_{g},g)$ is $\nabla^{0}$-totally geodesic. Hence,
\begin{equation}
    R^{\nabla^{lc}}(X,Y,Z)=\frac{1}{4}\{ g(Y,Z)X-g(X,Z)Y \},\,\,\forall X,Y,Z\in \Gamma(TM^{2d}).
    \label{eq:cont}
\end{equation}
The above \cref{eq:cont} leads to a contradiction  with regard to Goldberg's result\cite{goldberg1969integrabilityn}[cor 1.2], which states that K\"ahler manifolds of non-zero constant curvature do not exist. The result follows at once. For the compact co-K\"ahler manifolds, the proof relies on exactly the same philosophy as in K\"ahler manifolds case. We use the same reasoning displayed above and then we conclude using the results obtained by \citet{blair1966theory}, \,\citet{olszak1981almost},\,\citet{olszak1987almost},\,\citet{goldberg1969integrability1}, which state that co-K\"ahler manifolds of non-zero constant curvature
do not exist in any dimension.
\end{proof}

\begin{corollary}
For compact K\"ahler and co-K\"ahler manifolds, the Hellinger distance is
a lower bound on both geodesic distance of K\"ahler metrics and on the geodesic distance of co-K\"ahler metrics, for $N$ sufficiently large, i,e:

\begin{equation}
    d_{\text{Helliger}}(p,q)=\sqrt{\sum_{i=1}^{N}(\sqrt{p_{i}}-\sqrt{q_{i}})^{2}}\leq d_{K}(p,q), \,\forall p,q\in M^{2d}
\end{equation}
\vspace{2ex}
\begin{equation}
    d_{\text{Helliger}}(p,q)=\sqrt{\sum_{i=1}^{N}(\sqrt{p_{i}}-\sqrt{q_{i}})^{2}}\leq d_{coK}(p,q),\,\,\forall p,q\in M^{2d+1},
\end{equation}
where $d_{K}$ and $d_{coK}$ corresponds to geodesic distance of the K\"ahler metric and geodesic distance of the co-K\"ahler metric, respectively.
\label{cor:Hell}
\end{corollary}
\vspace{2ex}
From the above synthesis, it is noteworthy that compact K\"ahler and co-K\"ahler manifolds cannot be expected to be $\nabla^{0}$ totally geodesic submanifolds of $\mathcal{P}(\Xi^{N})$. We can now ask whether they can at least be $\nabla^{1}$-geodesic. The connection $\nabla^{1}=\nabla^{e}$ is called in information geometry, the exponential connection and $\nabla^{-1}=\nabla^{m}$ is called the mixture connection. Even for e-connections, there are topological obstructions for $M^{2d}$ to be $\nabla^{e}$-geodesic. To gain a better insight into this notion, consider a compact K\"ahler manifold $(M^{2d},\Omega,g)$ whose fundamental group is finite. If $M^{2d}$ is $\nabla^{e}$-geodesic, then $M^{2d}$ would be an exponential model, or more rigorously, the restriction of the family of categorical distributions on the discrete finite sample space $\Xi$ on $M^{2d}$ should be an exponential family. This refers to the fact that $\nabla^{e}$-totally geodesic submanifolds of the space of all strictly positive probability distributions on a finite sample space are exponential families (see\cite{amari1987differential},\cite{amari2007methods}). This is impossible as locally flat compact manifolds have an infinite fundamental group (see\cite{ay2002dually}). In the case of co-K\"ahler manifolds, this obstruction doesn't exist, as , according to \cite{li2008topology}, a compact co-K\"ahler manifold is a K\"ahler mapping torus; therefore, $M^{2d+1}$ fibers over a circle $\mathbb{S}^{1}$. Consequently, its fundamental group is infinite.\vspace{0.2cm}

In the case it would work on K\"ahler manifolds, according to \cite{amari1987differential} and \cite{amari2007methods}, there exist local charts $(U_{j},x^{j})$ on $M^{2d}$ such that $(\Gamma^{e}_{ij})^{k}=0$. Now, assume that the charts $U^{j}$ intersect  Darboux charts $U^{i}$ on $M^{2d}$. Then, on the commune charts $(U_{ij}^{l},y^{l},y^{d+l})$, we get $\Omega_{g}|_{U^{l}_{ij}}=\sum dy^{l}\wedge dy^{d+l}$ and $\nabla^{e}_{\partial_{y^{j}}}\partial_{y^{j}}=0$ for $j=1,..,2d.$ Let $l:U_{ij}^{l}\rightarrow M^{2d}$ be an inclusion map. The quadruplet $(\Omega_{g}|_{U^{l}_{ij}},l^{*}g,l^{*}\nabla^{e},l^{*}\nabla^{m})$ is a special K\"ahler structure and consequently, $(\nabla^{e},\nabla^{m})$ are statistical symplectic connections on $U_{ij}^{l}$. Fore further details on special K\"ahler manifold, refer to \cite{freed1999special},\cite{hitchin1999moduli}.

\vspace{2ex}

Now, let's examine the non-compact manifolds case. Non-compact K\"ahler and co-K\"ahler manifolds are neither $\nabla^{0}$-totally geodesic in the location-scale families nor in the family of discrete distribution. The proof of this assertion invests exactly the same arguments used for the compact case. The only possibility lies with the family of location families. In this case, according to \citet{wolf1972spaces}, K\"ahler and co-K\"ahler manifolds are isometric to a flat cylinder over an
Euclidean torus.

\subsection{Statistical symplectic connections and a K\"ahler Potential}
\label{subsec:Stasymp}
\begin{proposition}
 $\alpha$-statistical symplectic connections have the following components in local holomorphic coordinates $(z_1,z_2,...,z_{d})$:

\begin{equation}
 2(\Gamma^{\alpha }_{ij})^{k}=-\alpha\Gamma^{k}_{ij}+2(\Gamma^{lc}_{ij})^{k}-\alpha\partial_i(\Theta^{l}_j)\Theta^{r}_{l}\delta_{rk}-\alpha\Theta^{l}_j\Gamma^{p}_{il}\Theta_{s}^{p}\delta_{pk} 
\end{equation}

\begin{equation}
    2(\Gamma^{\alpha }_{i\Bar{j}})^{k}=-\alpha\Gamma^{k}_{i\Bar{j}}+2(\Gamma^{lc}_{i\Bar{j}})^{k}-\alpha\partial_i(\Theta^{l}_{\Bar{j}})\Theta^{r}_{l}\delta_{rk}-\alpha\Theta^{l}_{\Bar{j}}\Gamma^{p}_{il}\Theta_{s}^{p}\delta_{pk}
\end{equation}

\begin{equation}
    2(\Gamma^{\alpha }_{ij})^{\Bar{k}}=-\alpha\Gamma^{\Bar{k}}_{ij}+2(\Gamma^{lc}_{ij})^{\Bar{k}}-\alpha\partial_i(\Theta^{l}_j)\Theta^{\Bar{r}}_{l}\delta_{\Bar{r}\Bar{k}}-\alpha\Theta^{l}_j\Gamma^{s}_{il}\Theta_{s}^{\Bar{p}}\delta_{\Bar{p}\Bar{k}}
\end{equation}

\begin{equation}
    2(\Gamma^{\alpha }_{i\Bar{j}})^{\Bar{k}}=-\alpha\Gamma^{\Bar{k}}_{i\Bar{j}}+2(\Gamma^{lc}_{i\Bar{j}})^{\Bar{k}}-\alpha\partial_i(\Theta^{l}_{\Bar{j}})\Theta^{\Bar{r}}_{l}\delta_{\Bar{r}\Bar{k}}-\alpha\Theta^{l}_{\Bar{j}}\Gamma^{s}_{il}\Theta_{s}^{\Bar{p}}\delta_{\Bar{p}\Bar{k}}
\end{equation}
\label{prop:loca}
\end{proposition}

\begin{proof}

Grounded on corollary \ref{lem:principal}, we have:

\begin{equation}
    \nabla^{\alpha}=\nabla^{lc}-\frac{1}{2}\alpha \Theta(\nabla\Theta) \,,  \forall a\in\mathbb{R}.
\end{equation}

Let's choose a local coordinate $(z_1,..,z_{d})$ and hence a local basis
$(\partial z_1,..,\partial z_d)$ for $T_{\mathbb{C}}M^{2d}$. We define the Christoffel symbols $(\Gamma^{\alpha}_{ij})^{k}$ by
\[
 \nabla^{\alpha}_{\partial z_{i}}\partial z_{j}=(\Gamma^{\alpha }_{ij})^{k}\partial_{z_{k}}+(\Gamma^{\alpha}_{ij})^{\Bar{k}}\partial_{\Bar{z}_{k}} , \]   
 \[
   \nabla^{\alpha}_{\partial z_{i}}\partial \Bar{z}_{j}=(\Gamma^{\alpha }_{i\Bar{j}})^{k}\partial_{z_{k}}+(\Gamma^{\alpha}_{i\Bar{j}})^{\Bar{k}}\partial_{\Bar{z}_{k}}, \]
Using a direct computation, we obtain

$$2\nabla^{\alpha}_{\partial_i}\partial_j+\alpha\nabla_{\partial_i}\partial_j=2\nabla^{lc}_{\partial_i}\partial_j-\alpha\Theta(\nabla_{\partial_i}\Theta^{l}_j\partial_l)$$

$$=2\nabla^{lc}_{\partial_i}\partial_j-\alpha\Theta(\partial_i\Theta^{l}_j\partial_l+\Theta^{l}_j\nabla_{\partial_i}\partial_l)$$

$$=2\nabla^{lc}_{\partial_i}\partial_j-\alpha\partial_i(\Theta^{l}_j)\Theta^{r}_{l}\partial_r-\alpha\Theta^{l}_j\Theta(\nabla_{\partial_i}\partial_l)$$

$$=2\nabla^{lc}_{\partial_i}\partial_j-\alpha\partial_i(\Theta^{l}_j)\Theta^{r}_{l}\partial_r-\alpha\Theta^{l}_j\Theta(\Gamma^{s}_{il}\partial_{s})$$

Finally, we get

$$2\nabla^{\alpha}_{\partial_i}\partial_j+\alpha\nabla_{\partial_i}\partial_j=2\nabla^{lc}_{\partial_i}\partial_j-\alpha\partial_i(\Theta^{l}_j)\Theta^{r}_{l}\partial_r-\alpha\Theta^{l}_j\Gamma^{s}_{il}\Theta_{s}^{p}\partial_{p}$$

It follows form the above that:

\begin{equation}
    (2(\Gamma^{\alpha }_{ij})^{k}+\alpha\Gamma^{k}_{ij})\partial_{k}=2(\Gamma^{lc}_{ij})^{k}\partial_{k}-\alpha\partial_i(\Theta^{l}_j)\Theta^{r}_{l}\partial_r-\alpha\Theta^{l}_j\Gamma^{s}_{il}\Theta_{s}^{p}\partial_{p}
    \label{eq:150}
\end{equation}

We deduce that:
\begin{equation}
    (2(\Gamma^{\alpha }_{ij})^{k}+\alpha\Gamma^{k}_{ij})g_{k\Bar{l}}=2(\Gamma^{lc}_{ij})^{k}g_{k\Bar{l}}-\alpha\partial_i(\Theta^{l}_j)\Theta^{r}_{l}g_{r\Bar{l}}-\alpha\Theta^{l}_j\Gamma^{p}_{il}\Theta_{s}^{p}g_{p\Bar{l}}
\end{equation}

Thus,

\begin{equation}
 2(\Gamma^{\alpha }_{ij})^{k}+\alpha\Gamma^{k}_{ij}=2(\Gamma^{lc}_{ij})^{k}-\alpha\partial_i(\Theta^{l}_j)\Theta^{r}_{l}\delta_{rk}-\alpha\Theta^{l}_j\Gamma^{p}_{il}\Theta_{s}^{p}\delta_{pk} 
 \label{eq:152}
\end{equation}
Now replacing $j$ by $\Bar{j}$ in the above formula, we get

\begin{equation}
    2(\Gamma^{\alpha }_{i\Bar{j}})^{k}+\alpha\Gamma^{k}_{i\Bar{j}}=2(\Gamma^{lc}_{i\Bar{j}})^{k}-\alpha\partial_i(\Theta^{l}_{\Bar{j}})\Theta^{r}_{l}\delta_{rk}-\alpha\Theta^{l}_{\Bar{j}}\Gamma^{p}_{il}\Theta_{s}^{p}\delta_{pk}
    \label{eq:153}
\end{equation}

Referring to \cref{eq:150}, we obtain
\begin{equation}
    (2(\Gamma^{\alpha }_{ij})^{\Bar{k}}+\alpha\Gamma^{\Bar{k}}_{ij})\partial_{\Bar{k}}=2(\Gamma^{lc}_{ij})^{k}\partial_{\Bar{k}}-\alpha\partial_i(\Theta^{l}_j)\Theta^{r}_{l}\partial_r-\alpha\Theta^{l}_j\Gamma^{s}_{il}\Theta_{s}^{p}\partial_{p}
\end{equation}

Hence,

\begin{equation}
    (2(\Gamma^{\alpha }_{ij})^{\Bar{k}}+\alpha\Gamma^{\Bar{k}}_{ij})g_{l\Bar{k}}=2(\Gamma^{lc}_{ij})^{\Bar{k}}g_{l\Bar{k}}-\alpha\partial_i(\Theta^{l}_j)\Theta^{\Bar{r}}_{l}g_{\Bar{r}l}-\alpha\Theta^{l}_j\Gamma^{s}_{il}\Theta_{s}^{\Bar{p}}g_{\Bar{p}l}
\end{equation}
we deduce that
\begin{equation}
    2(\Gamma^{\alpha }_{ij})^{\Bar{k}}+\alpha\Gamma^{\Bar{k}}_{ij}=2(\Gamma^{lc}_{ij})^{\Bar{k}}-\alpha\partial_i(\Theta^{l}_j)\Theta^{\Bar{r}}_{l}\delta_{\Bar{r}\Bar{k}}-\alpha\Theta^{l}_j\Gamma^{s}_{il}\Theta_{s}^{\Bar{p}}\delta_{\Bar{p}\Bar{k}}
    \label{eq:156}
\end{equation}
Replacing $j$ by $\Bar{j}$ in the above formula, we obtain
\begin{equation}
    2(\Gamma^{\alpha }_{i\Bar{j}})^{\Bar{k}}+\alpha\Gamma^{\Bar{k}}_{i\Bar{j}}=2(\Gamma^{lc}_{i\Bar{j}})^{\Bar{k}}-\alpha\partial_i(\Theta^{l}_{\Bar{j}})\Theta^{\Bar{r}}_{l}\delta_{\Bar{r}\Bar{k}}-\alpha\Theta^{l}_{\Bar{j}}\Gamma^{s}_{il}\Theta_{s}^{\Bar{p}}\delta_{\Bar{p}\Bar{k}}
    \label{eq:last}
\end{equation}
Finally, relying upon \cref{eq:152}, \cref{eq:153}, \cref{eq:156},\cref{eq:last} we obtain:

\begin{equation}
 2(\Gamma^{\alpha }_{ij})^{k}=-\alpha\Gamma^{k}_{ij}+2(\Gamma^{lc}_{ij})^{k}-\alpha\partial_i(\Theta^{l}_j)\Theta^{r}_{l}\delta_{rk}-\alpha\Theta^{l}_j\Gamma^{p}_{il}\Theta_{s}^{p}\delta_{pk} 
\end{equation}

\begin{equation}
    2(\Gamma^{\alpha }_{i\Bar{j}})^{k}=-\alpha\Gamma^{k}_{i\Bar{j}}+2(\Gamma^{lc}_{i\Bar{j}})^{k}-\alpha\partial_i(\Theta^{l}_{\Bar{j}})\Theta^{r}_{l}\delta_{rk}-\alpha\Theta^{l}_{\Bar{j}}\Gamma^{p}_{il}\Theta_{s}^{p}\delta_{pk}
\end{equation}

\begin{equation}
    2(\Gamma^{\alpha }_{ij})^{\Bar{k}}=-\alpha\Gamma^{\Bar{k}}_{ij}+2(\Gamma^{lc}_{ij})^{\Bar{k}}-\alpha\partial_i(\Theta^{l}_j)\Theta^{\Bar{r}}_{l}\delta_{\Bar{r}\Bar{k}}-\alpha\Theta^{l}_j\Gamma^{s}_{il}\Theta_{s}^{\Bar{p}}\delta_{\Bar{p}\Bar{k}}
\end{equation}

\begin{equation}
    2(\Gamma^{\alpha }_{i\Bar{j}})^{\Bar{k}}=-\alpha\Gamma^{\Bar{k}}_{i\Bar{j}}+2(\Gamma^{lc}_{i\Bar{j}})^{\Bar{k}}-\alpha\partial_i(\Theta^{l}_{\Bar{j}})\Theta^{\Bar{r}}_{l}\delta_{\Bar{r}\Bar{k}}-\alpha\Theta^{l}_{\Bar{j}}\Gamma^{s}_{il}\Theta_{s}^{\Bar{p}}\delta_{\Bar{p}\Bar{k}}
\end{equation}
\end{proof}

\begin{corollary}
    There exists a real-valued smooth function $\Phi: U \rightarrow\mathbb{R}$ known as the Kähler potential 
such that
\begin{equation}
 2(\Gamma^{\alpha }_{ij})^{k}=-\alpha\Gamma^{k}_{ij}+2g^{k\Bar{l}}\partial_{i}\partial_{j}\partial_{\Bar{l}}\Phi-\alpha\partial_i(\Theta^{l}_j)\Theta^{r}_{l}\delta_{rk}-\alpha\Theta^{l}_j\Gamma^{p}_{il}\Theta_{s}^{p}\delta_{pk} 
\end{equation}

\begin{equation}
    2(\Gamma^{\alpha }_{i\Bar{j}})^{k}=-\alpha\Gamma^{k}_{i\Bar{j}}-\alpha\partial_i(\Theta^{l}_{\Bar{j}})\Theta^{r}_{l}\delta_{rk}-\alpha\Theta^{l}_{\Bar{j}}\Gamma^{p}_{il}\Theta_{s}^{p}\delta_{pk}
\end{equation}

\begin{equation}
    2(\Gamma^{\alpha }_{ij})^{\Bar{k}}=-\alpha\Gamma^{\Bar{k}}_{ij}-\alpha\partial_i(\Theta^{l}_j)\Theta^{\Bar{r}}_{l}\delta_{\Bar{r}\Bar{k}}-\alpha\Theta^{l}_j\Gamma^{s}_{il}\Theta_{s}^{\Bar{p}}\delta_{\Bar{p}\Bar{k}}
\end{equation}

\begin{equation}
    2(\Gamma^{\alpha }_{i\Bar{j}})^{\Bar{k}}=-\alpha\Gamma^{\Bar{k}}_{i\Bar{j}}-\alpha\partial_i(\Theta^{l}_{\Bar{j}})\Theta^{\Bar{r}}_{l}\delta_{\Bar{r}\Bar{k}}-\alpha\Theta^{l}_{\Bar{j}}\Gamma^{s}_{il}\Theta_{s}^{\Bar{p}}\delta_{\Bar{p}\Bar{k}}
\end{equation}

$$\partial_{k}\partial_{\Bar{i}}\partial_{j}\Phi=\Gamma^{a}_{{ki},\Bar{j}}+\Gamma^{-a}_{k\Bar{j},i}$$

$$\partial_{\Bar{k}}\partial_{\Bar{i}}\partial_{j}\Phi=\Gamma^{a}_{{\Bar{k}i},\Bar{j}}+\Gamma^{-a}_{\Bar{k}\Bar{j},i}$$

\end{corollary}
\begin{proof}

It's well known that the Christoffels symbols of the Levi-Civita connection $g$ are identified in a local coordinate
$(U,z_1,..,z_d)$ by:
\begin{equation}
(\Gamma^{lc})^{\Bar{k}}_{ij}=0,\,(\Gamma^{lc})^{\Bar{k}}_{i\Bar{j}}=(\Gamma^{lc})^{k}_{i\Bar{j}}=0,\, (\Gamma^{lc})^{k}_{ij}=g^{k\Bar{l}}\frac{\partial g_{i{\Bar{l}}}}{\partial z_{j}} 
\end{equation}
There exists a
real-valued function $\Phi$ on a neighbourhood of $p\in M^{2d}$ such that
\begin{equation}
g_{i\Bar{k}}=\partial_{i}\partial_{\Bar{k}}\Phi.
\label{eq:171}
\end{equation}
We accordingly have:
$$(\Gamma^{lc})^{k}_{ij}=g^{k\Bar{l}}\frac{\partial g_{j{\Bar{l}}}}{\partial z_{i}}=g^{k\Bar{l}}\partial_{i}\partial_{j}\partial_{\Bar{l}}\Phi.$$
With respect to \ref{prop:loca}, we conclude that
\begin{equation}
 2(\Gamma^{\alpha }_{ij})^{k}=-\alpha\Gamma^{k}_{ij}+2g^{k\Bar{l}}\partial_{i}\partial_{j}\partial_{\Bar{l}}\Phi-\alpha\partial_i(\Theta^{l}_j)\Theta^{r}_{l}\delta_{rk}-\alpha\Theta^{l}_j\Gamma^{p}_{il}\Theta_{s}^{p}\delta_{pk} 
\end{equation}

\begin{equation}
    2(\Gamma^{\alpha }_{i\Bar{j}})^{k}=-\alpha\Gamma^{k}_{i\Bar{j}}-\alpha\partial_i(\Theta^{l}_{\Bar{j}})\Theta^{r}_{l}\delta_{rk}-\alpha\Theta^{l}_{\Bar{j}}\Gamma^{p}_{il}\Theta_{s}^{p}\delta_{pk}
\end{equation}

\begin{equation}
    2(\Gamma^{\alpha }_{ij})^{\Bar{k}}=-\alpha\Gamma^{\Bar{k}}_{ij}-\alpha\partial_i(\Theta^{l}_j)\Theta^{\Bar{r}}_{l}\delta_{\Bar{r}\Bar{k}}-\alpha\Theta^{l}_j\Gamma^{s}_{il}\Theta_{s}^{\Bar{p}}\delta_{\Bar{p}\Bar{k}}
\end{equation}

\begin{equation}
    2(\Gamma^{\alpha }_{i\Bar{j}})^{\Bar{k}}=-\alpha\Gamma^{\Bar{k}}_{i\Bar{j}}-\alpha\partial_i(\Theta^{l}_{\Bar{j}})\Theta^{\Bar{r}}_{l}\delta_{\Bar{r}\Bar{k}}-\alpha\Theta^{l}_{\Bar{j}}\Gamma^{s}_{il}\Theta_{s}^{\Bar{p}}\delta_{\Bar{p}\Bar{k}}
\end{equation}

Additionally, departing from \cref{eq:171} and based on the duality condition, one can get the following expressions:

$$\partial_{k}\partial_{\Bar{i}}\partial_{j}\Phi=\Gamma^{a}_{{ki},\Bar{j}}+\Gamma^{-a}_{k\Bar{j},i};$$

$$\partial_{\Bar{k}}\partial_{\Bar{i}}\partial_{j}\Phi=\Gamma^{a}_{{\Bar{k}i},\Bar{j}}+\Gamma^{-a}_{\Bar{k}\Bar{j},i},$$

where $\Gamma^{a}_{kj,i}=g_{il}(\Gamma^{a}_{kj})^{l}\,,\forall i, j, k.$
\end{proof}

\subsubsection{K\"ahler potential and Kullback-Leibler divergence}
\label{subs:KahlandKL}
    We start this subsection by recalling an outstanding result known to all specialists. In \cite{matumoto1993any}, Matumoto  revealed that a divergence exists for any statistical manifold. $(M,g,T)$. However, it is
not unique and there are infinitely many divergences that provide the same geometrical structure. Referring to the work of \cite{matumoto1993any} Matumoto, we deduce with our statistical characterization of K\"ahler metrics that, for any K\"ahler metric $g$, there exists a divergence or contrast function $D$ such that on a neighborhood $\Delta\subset U\times U\simeq U\times\Bar{U}$ of the diagonal ($\Bar{U}$ denotes the neighborhood defined by conjugated
local coordinates), it is expressed  by:

\begin{equation}
    g_{i,\Bar{j}}(p)=-\frac{\partial^{2}D(z,\Bar{z})}{\partial z^{i}\partial \Bar{z}^{j}}|_{\Delta}
    \label{eq:187}
\end{equation}
Remember that the K\"ahler metric can equally be indicated in a local coordinate by the K\"ahler potential by:
\begin{equation}
g_{i,\Bar{j}}(p)=\frac{\partial^{2}\Phi(z,\Bar{z})}{\partial z^{i}\partial \Bar{z}^{j}}
\label{eq:pote}
\end{equation}

The affinity between \cref{eq:187} and \cref{eq:pote} is not accidental. These facts are suggestive that there exists a connection
between the Kähler potential and the divergence function. Resting on \cref{eq:187} and \cref{eq:pote}, we obtain on the common holomorphic coordinate, the following formula:

\begin{equation}
    \frac{\partial^{2}(D(z,\Bar{z})-\Phi(z,\Bar{z}))}{\partial z^{i}\partial \Bar{z}^{j}}=0.
\end{equation}

Now, $\partial\Bar{\partial}$-Lemma is estimated in terms of

\begin{equation}
    D(z,\Bar{z})=\Phi(z,\Bar{z})+h(z)+\Bar{h}(z),
\end{equation}
where $h$ is an holomorphic function. This yields the
relation between the divergence function and the Kähler potential.
\vspace{2ex}

We infer from Theorem \ref{th:KacOka} that a K\"ahler metric is always obtained by pulling back the Fisher metric of some statistical model $(\mathcal{P}(\Xi^{N}),g^{F},T^{A-C})$. Let's consider an isostatistical embedding $i:(M^{2d},\Omega,g,T^{Ka})\hookrightarrow ( \mathcal{P}(\Xi^{N}), g^{F},T^{A-C});$ for $N$ sufficiently large, where:

\begin{equation}
    g=i^{*}g^{F} \,\text{and}\,\, T^{Ka}=i^{*} T^{A-C}.  
\end{equation}

Consider $D_{KL}$ the Kullback-Leibler divergence on $\mathcal{P}(\Xi^{N})$, defined as:
\begin{equation}
    D_{KL}(p,q)=\text{E}_{p}(log(\frac{p}{q}))\,\,, p,q\in \mathcal{P}(\Xi^{N}).
\end{equation}

It is notably related to the Fisher information metric $g^{F}$ as follows:

\begin{equation}
    g^{F}_{ij}(\theta_{0})=\frac{\partial^{2}D_{KL}(p_{\theta_{0}},p_{\theta})}{\partial\theta_{i}\theta_{j}}|_{\theta=\theta_{0}}
    \label{eq:pull}
\end{equation}

With respect to \cref{eq:pull}, we obviously locally have 

\begin{equation}
    D_{KL}(i(z),\Bar{i(z)})=D(z,\Bar{z})+r(z)+\Bar{r}(z),
\end{equation}
 , where $r$ is some holomorphic function. Now, using \cref{eq:pote}, we easily locally obtain
 
 \begin{equation}
    D_{KL}(i(z),\Bar{i(z)})=\Phi(z,\Bar{z})+f(z)+\Bar{f}(z),
\end{equation}
\vspace{2ex}
where $f$ is some holomorphic function.
 \vspace{2ex}

We conclude this section by investigating the case of real analytic K\"ahler metrics.
\vspace{2ex}

\subsubsection{A Real analytic K\"ahler metric and a Fisher-information metric of exponential family }
\label{subs:analy}
Let us now consider a real analytic Kähler
metric $g$. Recall that a K\"ahler metric is a real analytic if in a fixed local coordinate $(z_{1},...,z_{d})$ on a neighbourhood $U$ at any point $z\in M^{2d}$, there exists a real analytical K\"ahler potential $\Phi:U\rightarrow\mathbb{R}$ such that:
\begin{equation}
g_{i\Bar{j}}=\frac{\partial^{2}\Phi(z,\Bar{z})}{\partial z^{i}\partial \Bar{z}^{j}}.
\end{equation}

\begin{thm}
    Any real analytic K\"ahler metric is locally is the Fisher information of an exponential family.
\end{thm}

\begin{proof}
   Let's consider a real analytic K\"ahler metric $g$ on $M^{2d}$. In the neighbourhood $U$ of each point $z\in M^{2d}$, there exist local
holomorphic coordinates $(z^{1},...,z^{d})$ centered at $z$ (unique modulo unitary linear transformations) such that 
\vspace{-4pt}
\begin{equation}
g_{i\Bar{j}}(z)=\frac{\partial^{2}\Phi}{\partial z^{i}\partial \Bar{z}^{j}}(z)=\delta_{ij}
\label{eq:kalerdia}
\end{equation}
These holomorphic
coordinates $(z^{1},...,z^{d})$ are
called normal or geodesic coordinates or Bochner’s coordinates around the point $p$ (see\cite{bochner1947curvature} and \cite{calabi1953isometric}). With respect the analycity of the K\"ahler metric $g$, it follows from\cite{calabi1953isometric}, that the K\"ahler potential can be analytically extended
in a neighbourhood $V\subset U\times U\simeq  U\times \Bar{U} $
of the diagonal. This extension is denoted by $\Phi(z,\Bar{w})$. Now consider the diastasis function $D^{g}(z,w)$ determined by Calabi \cite{calabi1953isometric} as:

\begin{equation}
    D^{g}(z,w)=\Phi(z,\Bar{w})+\Phi(w,\Bar{w})-\Phi(z,\Bar{w})-\Phi(w,\Bar{z}),\quad (z,w)\in V.
    \label{eq:cala}
\end{equation}
According to\cite{calabi1953isometric}, the  diastasis function satisfies the following properties
\begin{align}
\label{pr:distat}
        D(z,z) & =0,\\ \label{pr:distat0}
    g_{i\Bar{j}}(z) & =\frac{\partial^{2}D^{g}(z,w)}{\partial z^{i}\partial \Bar{z}^{j}}=\frac{\partial^{2}\Phi(z,\Bar{z})}{\partial z^{i}\partial \Bar{z}^{j}},\, \text{for}\, q\, \text{fixed};\\\label{pr:distat1}
    D^{g}(z,w) & =(d_{R}(z,w))^{2}+o(d_{R}(z,w))^{4}.   
\end{align}
where $d_{R}$ is the geodesic distance with respect to the analytic K\"ahler metric $g$ between $z$ and $w$. Referring the above equations \ref{pr:distat}, \ref{pr:distat0}, \ref{pr:distat1}, we deduce that the diastasis function $D^{g}$ is a divergence function or a
contrast function on $U$. Now, we can use Eguchi construction \cite{eguchi1992geometry} in order to define statistical structure as follows:
\begin{align}
       \Gamma^{D^{g}}_{ij,k}(z) & =-\partial_{z^{i}}\partial_{z^{j}}\partial_{\Bar{z}_k}D^{g}(z,w)\restriction_{z=w},\\ 
    \Gamma^{*D^{g}}_{ij,k}(z)& =-\partial_{z^{k}}\partial_{\Bar{z}^{i}}\partial_{\Bar{z}^{j}} D^{g}(z,w)\restriction_{z=w};\\ 
    T_{ijk}^{D^{g}}(z)& =\Gamma^{*D^{g}}_{ij,k}(z)-\Gamma^{D^{g}}_{ij,k}(z).
\end{align}

Obviously, using \cref{eq:kalerdia} on $U$, we have 
\begin{equation}
    \Gamma^{D^{g}}_{ij,k}(z)=\Gamma^{*D^{g}}_{ij,k}(z)=0,\,\text{and}\,\, T_{ijk}^{D^{g}}(z)=0.
    \label{eq:divv}
\end{equation}
Based on \cref{eq:divv}, the diastasic function $D^{g}$ coincides with Bregman divergence (canonical divergence). For simplicity, we denote $\Phi(z,\Bar{z})=\Phi(z)$. Resting on \cite{banerjee2005clustering}, it's possible to find  a measure $\mu(x)$ such that
\begin{equation}
   \Phi(z) =log\int exp(<z,x>)d\mu(x)
\end{equation}
Therefore, the exponential family is provided by
\begin{equation}
    p(x,z)=exp((<z,x>-\Phi(z)).
\end{equation}
As a matter of fact,

$$g_{i\Bar{j}}(z) =g^{F}_{i\Bar{j}}(z)=\frac{\partial^{2}(log\int exp(<z,x>)d\mu(x))}{\partial z^{i}\partial \Bar{z}^{j}}.$$

The result follows at once.
\end{proof}

 The following corollary is an easy consequence of the above theorem.
\vspace{1ex}
\begin{corollary}
    Any  real analytic K\"ahler manifold $(M, g)$ is locally an exponential family.
\end{corollary}
\begin{proof}
    The claim follows from the above theorem.
\end{proof}
\vspace{2ex}

\begin{figure}[H]
    \centering
\begin{center}
\begin{tikzpicture}

    \draw[smooth cycle, tension=0.4, fill=white, pattern color=brown, pattern=north west lines, opacity=0.7] plot coordinates{(2,2) (-0.5,0) (3,-2) (5,1)} node at (3,2.3) {$(M^{2d},\omega,g)$};

     \draw[smooth cycle, pattern color=blue, pattern=crosshatch dots] 
         plot coordinates {(4, 0) (3.5, 0.6) (3.0, 1.2) (2.5, 1.2) (2.2, 0.8) (2.3, 0.5) (2.6, 0.3) (3.5, 0.0)}; node [label={[label distance=-0.2cm, xshift=.55cm, yshift=2cm, fill=white]:$(U_j,p(x,z)=exp((<z,x>-\Phi(z)))$}] {};
\end{tikzpicture}
 
\end{center}

\caption{An real analytic K\"ahler manifold is locally an exponential family.}
    \label{fig:anal} 
\end{figure}
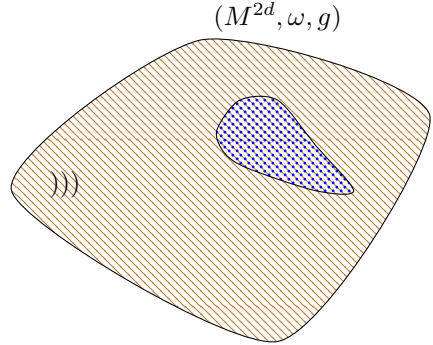

\section{Examples}

\begin{enumerate}
    \item \textbf{Statistical almost Kahler manifolds (K\"ahler manifolds)}
    \vspace{2ex}
    \begin{enumerate}
        \item Special K\"ahler manifolds,
\vspace{2ex}
        \item Holomorphic statistical manifolds,
        \vspace{2ex}
       \item  Conic K\"ahler manifolds,
        \vspace{2ex}
       \item  Projective special K\"ahler manifolds.
    \end{enumerate}
    \vspace{2ex}
    \item \textbf{Statistical shs manifolds (co-K\"ahler manifolds)}\\
Consider a statistical almost K\"ahler manifold $(M^{2d},g,\Omega,\Theta,\nabla,\nabla^{*})$. We can construct a statistical shs manifold as follows:

$$(W^{2d+1}=M^{2d}\times\mathbb{R},\widetilde{g},\widetilde{\omega},\widetilde{\alpha},\widetilde{\nabla},\widetilde{\nabla^{*}}),$$

\vspace{1ex}
where $\widetilde{\alpha}=dt$, $E=(0,\frac{\partial}{\partial t})$, $\Tilde{\Theta}(.,fE)=(\Theta,0) $, $\Tilde{g}=g+dt^{2}$ and $\Tilde{\omega}=\Tilde{g}(\Tilde{\Theta},.)$
and the shs statistical connections are indicated by:

$$\widetilde{\nabla}_{X}Y=\nabla_{X}Y,\,\widetilde{\nabla}_{E}X=\widetilde{\nabla}_{X}E=0,\,\,\widetilde{\nabla}_{E}E=0,\,\,\, \forall X,Y\in \Gamma(TM^{2d}),$$
\vspace{-10pt}
$$\widetilde{\nabla^*}_{X}Y=\nabla^{*}_{X}Y,\,\widetilde{\nabla^{*}}_{E}X=\widetilde{\nabla^{*}}_{X}E=0,\,\widetilde{\nabla^{*}}_{E}E=0,\,\,\forall X,Y\in\Gamma(TM^{2d}).$$

\end{enumerate}

\section*{Summary and Perspectives}
Information geometry grants the opportunity to handle other connections(statistical connections) which generalise Levi-Civita connection. Therefore, multiple theorems of Riemanian geometry can be generalized with statistical connections. In the current research paper, we have proved a strong link between Information geometry and K\"ahler geometry. Through we did not solve Goldberg's conjecture, we provided given new characterisations of K\"ahler metrics and co-K\"ahler metrics with statistical ingredients. In particular, we corroborate that any K\"ahler manifold is a family of probability distributions and any K\"ahler metrics is a Fisher information metric. We also confirmed that a real analytic K\"ahler metric is always locally the Fisher information metric of some exponential family. As a final note, we would assert that this work can be taken further and built upon. In future research, we intend to solve Golberg's conjecture using a statistical approach.

\section*{Acknowledgement}
I am thankful to Michel Nguiffo-Boyom and Stephane Puechmorel for their introduction to the field of information geometry. I am sincerely grateful to Frédéric Barbaresco and all members of the working group "Nord-Bassin parisien" for stimulating conversations and clarifications on information geometry concepts. Thanks to Charles-Michel Marle for constructive discussions and insightful suggestions on cosymplectic structures which played an intrinsic role in this article. I would like to thank Andreas Guitart and Samir Marouani for their valuable help in presenting this article.

\bibliography{main}

\end{document}